\documentclass[10pt]{amsart}  

\usepackage[utf8]{inputenc}

\usepackage[left=1in,right=1in,top=1in,bottom=1.17in]{geometry}

\usepackage{graphicx,amssymb,amsmath,amsthm,wrapfig}
\usepackage{caption}
\usepackage{subcaption}
\usepackage{indentfirst}
\usepackage{cancel}
\usepackage{lipsum}
\usepackage{epstopdf}
\usepackage{epsfig}
\usepackage{comment}
\usepackage{enumerate,color}
\usepackage{empheq}
\usepackage{hyperref}

\everymath{\displaystyle}
\theoremstyle{plain}
\newtheorem{theorem}{Theorem}[section]
\newtheorem{corollary}[theorem]{Corollary}
\newtheorem{lemma}[theorem]{Lemma}
\newtheorem{definition}[theorem]{Definition}
\newtheorem{proposition}[theorem]{Proposition}
\newtheorem{remark}[theorem]{Remark}

\numberwithin{equation}{section}

\newcommand{\la}{\left\langle}
\newcommand{\ra}{\right\rangle}
\newcommand{\RNum}[1]{\uppercase\expandafter{\romannumeral #1\relax}}

\title[Stability of shock profiles for the Navier-Stokes-Poisson System]{Long-Time Behavior towards Shock Profiles for the Navier-Stokes-Poisson System}  
\author[M.-J. Kang, B. Kwon, and W. Shim]{Moon-Jin Kang, Bongsuk Kwon, and Wanyong Shim}

\address{(Moon-Jin Kang) Department of Mathematical Sciences, Korea Advanced Institute of Science and Technology, Daejeon, 34141, Korea}
\email{moonjinkang@kaist.ac.kr}

\address{(Bongsuk Kwon) Department of Mathematical Sciences, Ulsan National Institute of Science and Technology, Ulsan, 44919, Korea}
\email{bkwon@unist.ac.kr}

\address{(Wanyong Shim) Department of Mathematical Sciences, Korea Advanced Institute of Science and Technology, Daejeon, 34141, Korea}
\email{wyshim25@kaist.ac.kr}
\date{\today}

\thanks{\textbf{Acknowledgment.} M.-J. Kang was partially supported by the National Research Foundation of Korea (NRF-2019R1A5A1028324 and  NRF-RS-2024-00361663). B. Kwon was supported by Basic Science Research Program through the National Research Foundation of Korea (NRF) funded by the Ministry of science, ICT and future planning (NRF-2020R1A2C1A01009184). W. Shim was supported by the National Research Foundation of Korea (NRF) grant funded by the Korean government (MSIT) (No.2022R1A4A1032094).}

\begin{document}

\begin{abstract}
We study the stability of shock profiles in one spatial dimension for the isothermal Navier-Stokes-Poisson (NSP) system, which describes the dynamics of ions in a collision-dominated plasma. The NSP system admits a one-parameter family of smooth traveling waves, called shock profiles, for a given far-field condition satisfying the Lax entropy condition. In this paper, we prove that if the initial data is sufficiently close to a shock profile in $H^2$-norm, then the global solution of the Cauchy problem tends to the smooth manifold formed by the parametrized shock profiles as time goes to infinity. This is achieved using the method of $a$-contraction with shifts, which does not require the zero mass condition.\\


\end{abstract}

\maketitle

\tableofcontents

\section{Introduction}

\subsection{The Navier-Stokes-Poisson system}
We consider an isothermal plasma in one spatial dimension, where the dynamics are primarily governed by collisions, as in a dusty plasma. The dynamics of ions in the \textit{collision-dominated plasma} can be described by the one-dimensional compressible Navier-Stokes-Poisson system in Lagrangian mass coordinates \cite{GGKS}:
\begin{subequations} \label{NSP}
\begin{align}
& \label{NSP11} v_t - u_x = 0,\\
& \label{NSP22} u_t + p(v)_x   =  \left( \frac{\nu u_x}{v} \right)_x - \frac{\phi_x}{v}, \\
& \label{NSP33} - \lambda^2 \left( \frac{\phi_x}{v} \right)_x = 1 - ve^{\phi}
\end{align}
\end{subequations}
for $t>0$ and $x \in \mathbb{R}$. Here $v=1/n$ is the specific volume, for $n>0$ the density of ions, and $\phi$ is  the electric potential. The function $p(v)$ denotes the pressure given by $p(v) = K v^{-1}$. The constants $K>0$, $\nu>0$ and $\lambda>0$ represent the absolute temperature, viscosity coefficient and Debye length, respectively. For simplicity, we normalize by setting $K = 1$, $\nu = 1$, and $\lambda = 1$, as all the results in this paper hold for any positive constants $K$, $\nu$, and $\lambda$. In the Poisson equation \eqref{NSP33}, we have assumed that the electron density $n_e$ is determined by the Boltzmann relation, $n_e = e^\phi$, which is justified by the physical observation that electrons reach the equilibrium state much faster than ions for varying potential in a plasma \cite{Ch}.

To study the stability of shock profiles in the NSP system \eqref{NSP}, we consider the corresponding Cauchy problem with the initial data:
\begin{equation} \label{ic}
(v,u)(0,x) = (v_0,u_0)(x), \quad \lim_{x \rightarrow \pm \infty} (v_0,u_0)(x) = (v_\pm, u_\pm)
\end{equation}
for given constant end states $(v_\pm, u_\pm)$. The far-field data of $\phi$ is given by
\begin{equation} \label{qnc}
\lim_{x \rightarrow \pm \infty} \phi(t,x) = \phi_\pm,
\end{equation}
under the quasi-neutral condition $\phi_\pm =  \log v_\pm^{-1}$ at $x= \pm\infty$.

Before turning our attention to solutions to the Cauchy problem \eqref{NSP}-\eqref{ic}, we first rewrite the NSP system \eqref{NSP} in divergence form using \eqref{NSP33}. This reformulation is crucial in our analysis, as will be noted in Section 2.5. From the Poisson equation \eqref{NSP33}, we derive the following relation:
\begin{equation} \label{trf}
\frac{\phi_x}{v} = \left[ \frac{1}{v} + \frac{\lambda^2}{v} \left( \frac{\phi_x}{v} \right)_x - \frac{\lambda^2}{2} \left( \frac{\phi_x}{v} \right)^2 \right]_x.
\end{equation}
Substituting this into \eqref{NSP22}, the system \eqref{NSP} is rewritten as
\begin{subequations} \label{NSP'}
\begin{align}
& \label{NSP1} v_t - u_x = 0,\\
& \label{NSP2} u_t + \tilde{p}(v)_x  = \left( \frac{ u_x}{v} \right)_x + \Phi(v,\phi)_x, \\
& \label{NSP3} - \left( \frac{\phi_x}{v} \right)_x = 1 - ve^{\phi},
\end{align}
\end{subequations}
where the modified pressure $\tilde{p}$ and the electric force $\Phi$ arising from non-neutral plasma density are given by
\begin{equation} \label{p(v)}
\tilde{p}(v) := p(v) + \frac{1}{v}
\end{equation}
and
\begin{equation} \label{Phi}
\Phi (v,\phi) := \frac{1}{2} \left( \frac{\phi_x}{v} \right)^2 - \frac{1}{v}  \left( \frac{\phi_x}{v} \right)_x,
\end{equation}
respectively.

\subsection{Shock profiles and orbital stability}
We consider the shock profiles, a special class of traveling wave solutions of the form $(v,u,\phi)(t,x)=(\bar{v},\bar{u},\bar{\phi})(x - \sigma t)$ to \eqref{NSP}, where $\sigma$ denotes the shock speed. These waves are smooth monotone profiles that connect two distinct constant states at infinity, $(v_-, u_-, \phi_-)$ and $(v_+, u_+, \phi_+)$. We substitute the ansatz $(\bar{v},\bar{u},\bar{\phi})(x-\sigma t)$ into \eqref{NSP} to obtain the governing equations for the shock profiles:
\begin{equation} \label{shODE}
\begin{split}
& -\sigma \bar{v}' - \bar{u}' =0, \\
& - \sigma \bar{u}' + p(\bar{v})' = \left( \frac{\bar{u}'}{\bar{v}} \right)' -  \frac{\bar{\phi}'}{\bar{v}}, \\
& - \left( \frac{\bar{\phi}'}{\bar{v}} \right)' = 1 - \bar{v} e^{\bar{\phi}}
\end{split}
\end{equation}
with the far-field condition
\begin{equation} \label{ffc}
\lim_{\xi \rightarrow \pm \infty} (\bar{v},\bar{u},\bar{\phi}) (\xi) = (v_\pm,u_\pm,\phi_\pm),
\end{equation}
where $ \textstyle \xi := x - \sigma t$ and $'$ denotes $\textstyle \frac{d}{d \xi}$. 
Integrating the first two equations of \eqref{shODE} over $\mathbb{R}$, along with \eqref{qnc} and \eqref{trf}, we obtain the Rankine-Hugoniot condition:
\begin{equation} \label{RH}
\begin{cases}
-\sigma (v_+ - v_-) - (u_+ - u_-) = 0, \\
-\sigma (u_+ - u_-) - \left( \tilde{p}(v_+) - \tilde{p}(v_-) \right) = 0,
\end{cases}
\end{equation} 
which determines the shock speed
\[
\sigma = \sigma_\pm := \pm \sqrt{- \frac{\tilde{p}(v_+) - \tilde{p}(v_-)}{v_+ - v_-}}.
\]
In this paper, without loss of generality, we restrict ourselves to the 2-shock profile with $\sigma = \sigma_+$, which satisfies the Lax entropy condition $\sqrt{-\tilde{p}'(v_+)} < \sigma_+ < \sqrt{-\tilde{p}'(v_-)}$, or equivalently,
\begin{equation} \label{Laxc2} 
v_- < v_+. 
\end{equation}
For this case, the existence and uniqueness of small-amplitude shock profiles have been treated in \cite{DLZ}:

\begin{proposition} [\cite{DLZ}] \label{Prop.1.1}
For given $(v_-,u_-)$ with $v_- >0$, there exists a positive constant $\delta_0$ such that if $(v_+,u_+)$ satisfies \eqref{RH}, \eqref{Laxc2} and
\[
\lvert v_+ - v_- \rvert \sim \lvert u_+ - u_- \rvert =: \delta_S < \delta_0,
\]
then \eqref{shODE} admits a unique (up to a shift) shock profile $(\bar{v},\bar{u},\bar{\phi})(\xi)$ satisfying \eqref{ffc} and
\begin{equation} \label{vup'}
\sigma \bar{v}' = - \bar{u}' >0, \quad \underline{C} \bar{u}' \leq \bar{\phi}' \leq \overline{C} \bar{u}'
\end{equation}
for some positive constants $\underline{C},\overline{C}$. Moreover, the unique solution satisfying $ \textstyle \bar{v}(0) = \frac{v_-+v_+}{2}$ verifies the derivative bounds
\begin{equation} \label{shderiv}
\begin{cases}
\displaystyle \bigg\lvert \frac{d^k}{d\xi^k} (\bar{v}-v_+, \bar{u}-u_+, \bar{\phi} - \phi_+ ) \bigg\rvert \leq C_k \delta_S^{k+1} e^{-\theta \delta_S \lvert \xi \rvert}, & \xi > 0 \\
\displaystyle \bigg\lvert \frac{d^k}{d\xi^k} (\bar{v}-v_-, \bar{u}-u_-, \bar{\phi} - \phi_- ) \bigg\rvert \leq C_k \delta_S^{k+1} e^{-\theta \delta_S \lvert \xi \rvert}, & \xi < 0
\end{cases}
\end{equation}
for $ k \in \mathbb{N} \cup \{0\}$, where $C_k>0$ and $\theta>0$ are generic constants.
\end{proposition}

This proposition demonstrates that the shock profiles form a one-parameter family $(\bar{v}, \bar{u}, \bar{\phi})(x - \sigma t - \delta)$, where $\delta \in \mathbb{R}$ represents a shift. This shift results from the translational invariance of the shock profiles as solutions to \eqref{shODE}-\eqref{ffc}. Given this invariance, it is natural to investigate their orbital stability. More precisely, we define the orbital stability of the shock profiles as follows:

\begin{definition} [Orbital stability] \label{def:orbit} 
We say that the shock profile $(\bar{v},\bar{u},\bar{\phi})$ is orbitally stable, if for any $\varepsilon>0$ there exists $\gamma>0$ such that $\lVert (v_0 - \bar{v},u_0-\bar{u}) \rVert < \gamma$ implies
\[
\inf_{\delta \in \mathbb{R}} \lVert (v,u,\phi)(t,\cdot) - (\bar{v},\bar{u},\bar{\phi})(\cdot - \sigma t - \delta) \rVert < \varepsilon \quad \text{for all } t>0,
\]
where $(v,u,\phi)$ is a solution to \eqref{NSP}-\eqref{ic}.
\end{definition}

\subsection{Literature}
Studies on the stability of shock profiles have generated a considerable amount of literature over the past few decades, with particular focus on viscous shock profiles in the Navier-Stokes (NS) equations.
 We provide an overview of several notable works for weak shocks in one-dimensional NS equations. This problem was first studied by Matsumura and Nishihara \cite{MN}. In their work, the authors proved time-asymptotic stability under perturbations with zero integral, the so-called zero mass condition, using the anti-derivative method (see also \cite{Go, KM}). Later, this zero mass condition was successfully removed using pointwise estimate methods by Liu \cite{L}, Liu and Zeng \cite{LZ}, and Szepessy and Xin \cite{SX}. On the other hand, Howard and Zumbrun \cite{HZ}, and Mascia and Zumbrun \cite{MZ, MZ1} also proved the linear and nonlinear stability of viscous shock waves, without the zero mass condition, by establishing the Green function for the associated linearized equations. Recently, Kang and Vasseur \cite{KV3} obtained $L^2$ stability for the shocks by using the method of $a$-contraction with shifts (see also \cite{KV4,KV5,KVW2,KVW3}). Their result in \cite{KV3} also does not require the zero mass assumption.

For the one-dimensional NSP shocks, Duan, Liu and Zhang \cite{DLZ} proved the existence and asymptotic stability by using the classical anti-derivative method. Zhao \cite{Zh} extended the stability result to the case of large initial density oscillations, and Li, Mei and Yuan \cite{LMY} showed stability under space-periodic perturbations.

We note that the results in \cite{DLZ, Zh, LMY} rely on zero mass type assumptions. In this paper, we establish the nonlinear stability of shock profiles without such restrictions, incorporating the method of $a$-contraction with shifts.

\subsection{Main result}
We present our main theorem. 
\begin{theorem} \label{Main}
Let $(\bar{v},\bar{u},\bar{\phi})$ be the shock profile solution, described in Proposition \ref{Prop.1.1} with its strength $ \delta_S$.  There exist positive constants $\delta_1$ and $\varepsilon_0$ such that if $\delta_S<\delta_1$ and the initial data $(v_0,u_0)$ satisfies
\[
\lVert (v_0 - \bar{v}, u_0-\bar{u} ) \rVert_{H^2(\mathbb{R})} < \varepsilon_0, 
\]
then the Cauchy problem \eqref{NSP}-\eqref{ic} admits a unique global-in-time solution $(v,u,\phi)$ satisfying 
\[
\begin{split}
v(t,x) - \bar{v}(x-\sigma t - X(t)) & \in C \left( 0,\infty; H^2(\mathbb{R}) \right), \\
u(t,x) - \bar{u}(x-\sigma t - X(t)) & \in C \left( 0,\infty; H^2(\mathbb{R}) \right), \\
\phi(t,x) - \bar{\phi}(x-\sigma t - X(t)) & \in C \left( 0,\infty; H^3(\mathbb{R}) \right)
\end{split}
\]
for some shift function $X(t)$. Moreover, the solution $(v,u,\phi)(t,x)$ and shift $X(t)$ satisfy
\begin{equation} \label{behavior}
\lim_{t \rightarrow +\infty }\lVert (v,u,\phi)(t,\cdot) - (\bar{v},\bar{u},\bar{\phi})(\cdot - \sigma t - X(t)) \rVert_{L^\infty} = 0
\end{equation}
and
\begin{equation} \label{Xdecay}
\lim_{t \rightarrow +\infty} \lvert \dot{X}(t) \rvert =0. 
\end{equation}
 \end{theorem}

\begin{remark} \label{Rem}
The time-asymptotic behavior \eqref{behavior} in Theorem \ref{Main} implies that the shock profiles for the NSP system \eqref{NSP} are orbitally stable in the sense of Definition \ref{def:orbit}. Moreover, from \eqref{Xdecay}, we see that the shift function $X(t)$ grows at most sublinearly as $ t \to + \infty$, i.e.,
\[
\lim_{t \rightarrow +\infty} \frac{X(t)}{t} = 0.
\]
This implies that the solution converges to a family of shifted shock profiles in a controlled manner.
\end{remark}

\vspace{1em}
\textit{Plan of the paper}: In Section 2, we present the statement of an a priori estimate for perturbations around the shock profile and use it to prove Theorem \ref{Main}. The proof of the a priori estimate is provided in the subsequent sections. In Section 3, we introduce a modulated relative functional that enables the application of the relative entropy method to the NSP system. In Section 4, we obtain an energy estimate using the method of $a$-contraction with shifts, together with the relative functional. Section 5 focuses on establishing elliptic estimates for the Poisson equation, while Section 6 provides higher-order estimates and completes the proof of the a priori estimate.

\section{A priori estimate and Proof of Theorem \ref{Main}}
In this section, we provide an a priori estimate for $H^2$-perturbations around the shock and outline the proof of Theorem \ref{Main}, with the detailed proof deferred to Appendices B and C.

\subsection{Local existence}
We first present the local existence of solutions to the NSP system \eqref{NSP}.
\begin{proposition} \label{local}
Let $\underline{v}$, $\underline{u}$, and $\underline{\phi}$ be smooth monotone functions such that
\[
\underline{v}(x) = v_\pm, \quad \underline{u}(x) = u_\pm, \quad \underline{\phi}(x) = \phi_\pm \quad \text{for} \ \pm x \geq 1.
\]
Then, for any constants $M_0,M_1,\underline{\kappa}_0,\overline{\kappa}_0,\underline{\kappa}_1$, and $\overline{\kappa}_1$ with
\[
0 < M_0 < M_1  \quad \text{and} \quad 0 < \underline{\kappa}_1 < \underline{\kappa}_0 < \overline{\kappa}_0 < \overline{\kappa}_1,
\]
there exists a finite time $T_0>0$ such that if the initial data $(v_0,u_0)$ satisfies
\[
\lVert v_0 - \underline{v} \rVert_{H^2(\mathbb{R})} + \lVert u_0 - \underline{u} \rVert_{H^2(\mathbb{R})} \leq M_0 \quad \text{and} \quad \underline{\kappa}_0 \leq v_0(x) \leq \overline{\kappa}_0 \quad \text{for all} \ x \in \mathbb{R},
\]
the Cauchy problem \eqref{NSP}-\eqref{ic}, with \eqref{qnc}, admits a unique solution $(v,u,\phi)$ on $[0,T_0]$ satisfying
\[
\begin{split}
v-\underline{v} & \in C \left( 0,T_0; H^2(\mathbb{R}) \right), \\
u-\underline{u} & \in C \left( 0,T_0; H^2(\mathbb{R}) \right) \cap L^2 \left( 0,T_0; H^3(\mathbb{R}) \right), \\
\phi-\underline{\phi} & \in C \left( 0,T_0; H^3 (\mathbb{R}) \right),
\end{split}
\]
with
\[
\lVert \left( v-\underline{v}, u-\underline{u} \right) \rVert_{L^\infty(0,T_0;H^2(\mathbb{R}))} + \lVert \phi -\underline{\phi} \rVert_{L^\infty(0,T_0;H^3(\mathbb{R}))} \leq M_1
\]
and
\[
\underline{\kappa}_1 \leq v(t,x) \leq \overline{\kappa}_1 \quad \text{for all} \ (t,x) \in [0,T_0] \times \mathbb{R}.
\]
\end{proposition}

\begin{proof}
The local existence can be obtained using a standard energy method (see \cite{MN0}), and we omit the proof for brevity.
\end{proof}

\subsection{Construction of shift}
Next, we define the shift function $X: \mathbb{R}_+ \rightarrow \mathbb{R}$ as a solution to the ODE:
\begin{equation} \label{shiftODE}
\begin{split}
\dot{X}(t) & = - \frac{M}{\delta_S} \bigg( \int_\mathbb{R} a(t,x - X(t)) \bar{u}_x(x-\sigma t -X(t)) (u(t,x)-\bar{u}(x-\sigma t - X(t))) \, dx \\
& \qquad + \frac{1}{\sigma} \int_\mathbb{R} a(t,x - X(t)) \partial_x \tilde{p} \left(\bar{v} (x-\sigma t -X(t)) \right) (u(t,x)-\bar{u}(x-\sigma t - X(t))) \, dx \bigg),
\end{split}
\end{equation}
with $X(0) =0$, where the function $a(t,x)$ is defined in \eqref{a} and $\textstyle M :=\frac{5\sqrt{2}}{2v_-^3}$. The existence of the shift $X(t)$ is ensured by the standard existence theorem for ODEs, as shown in \cite{HKK, KVW2} for the Navier-Stokes equations.

In the following, we use abbreviated notation for functions shifted by $X(t)$. For any function $g: \mathbb{R} \rightarrow \mathbb{R}$, we define the shifted function of $g$ as
\begin{equation*}
g^X (\cdot) : = g(\cdot - X(t)), \quad t \geq 0.
\end{equation*}

\subsection{A priori estimate}
The \emph{a priori} estimate is stated as follows.
\begin{proposition} \label{Apriori}
Let $T>0$ be a positive constant. Suppose that $(v,u,\phi)$ is the solution to \eqref{NSP}-\eqref{ic} on $[0,T]$ and $(\bar{v},\bar{u},\bar{\phi})$ is the shock profile described in Proposition \ref{Prop.1.1}. Then, there exist positive constants $\delta_1$ and $\varepsilon_1$ such that if the solution $(v,u,\phi)$ satisfies
\begin{equation} \label{eps}
\sup_{0 \leq t \leq T} \Big( \lVert (v-\bar{v}^X,u-\bar{u}^X) (t,\cdot) \rVert_{H^2} + \lVert (\phi-\bar{\phi}^X)(t,\cdot) \rVert_{H^3} \Big) \leq \varepsilon_1
\end{equation}
and the shock strength $\delta_S$ is less than $\delta_1$, then it holds that
\begin{equation} \label{apriori}
\begin{split}
& \lVert (v-\bar{v}^X,u-\bar{u}^X, \phi - \bar{\phi}^X ) (t,\cdot) \rVert_{H^2}^2 + \delta_S \int_0^t \lvert \dot{X}(\tau) \rvert^2 \, d\tau \\
& + \frac{\sigma}{\sqrt{\delta}_S}\int_0^t \int_\mathbb{R} \bar{v}^X_x \bigg\lvert \tilde{p}(v)-\tilde{p}(\bar{v}^X) - \frac{u-\bar{u}^X}{2C_*} \bigg\rvert^2 \, dx d\tau + \int_0^t \int_\mathbb{R} \bar{v}^X_x \lvert u- \bar{u}^X \rvert^2 \, dx d \tau \\
& + \int_0^t \int_\mathbb{R} \left( \lvert (v-\bar{v}^X)_x \rvert^2 + \lvert (v-\bar{v}^X)_{xx} \rvert^2 + \lvert (u-\bar{u}^X)_x \rvert^2 + \lvert (u-\bar{u}^X)_{xx} \rvert^2 \right) \, dx d \tau \\
& + \int_0^t \int_\mathbb{R} \left( \lvert (u-\bar{u}^X)_{xxx} \rvert^2 + \lvert (\phi-\bar{\phi}^X)_{xx} \rvert^2 + \lvert (\phi-\bar{\phi}^X)_{xxx} \rvert^2 \right) \, dx d \tau \\
& \quad \leq C \lVert (v_0-\bar{v},u_0-\bar{u}) \rVert_{H^2}^2
\end{split}
\end{equation}
for all $t \in [0,T]$ and some constants $C>0$ and $C_*>0$.
\end{proposition}

\subsection{Proof of Theorem \ref{Main}}
Together with Propositions \ref{local} and \ref{Apriori}, we use the continuation argument to prove the global existence. We also use Proposition \ref{Apriori} to prove the long-time behavior as desired. Since the proof is standard and similar to that of the previous results \cite{HKK,KVW2}, we present it in the Appendices.

\subsection{Main ideas for the proof of Proposition \ref{Apriori}}
To establish the a priori estimate \eqref{apriori}, we employ the method of $a$-contraction with shifts for the energy estimates of zeroth order. The $a$-contraction method was originally introduced in the study of stability of (viscous) shocks for the (viscous) system of conservation laws (e.g., \cite{KV1,KV2,KV3,KV4,V}). Recently, it has been used for the Navier-Stokes system with some effects, such as internal capillarity \cite{HKKL} and boundary effects \cite{HuKKL2}.

Especially, in our case, the method will be used to control the main terms from the isothermal Navier-Stokes part of \eqref{NSP'}. However, to control the electric force term in \eqref{NSP2}, we may use a dissipative energy introduced in \cite{DLZ}. So, for the energy estimates, the  $a$-contraction method will be applied to the relative functional associated to the dissipative energy, as in \eqref{relentropy}, obtained by modulating the relative entropy of the NS equations. 

After obtaining the energy estimate for zeroth-order perturbations, we proceed with elliptic estimates for the Poisson equation and higher-order estimates. In particular, to estimate the first derivative of the perturbation $(v-\bar{v})$, we use the divergence form of the system given by \eqref{NSP1}-\eqref{NSP2}. This is essential to control the first derivative, but it also requires extending the estimates to the second derivative (see Lemma~\ref{Lemma:v1}). Thus, we complete our analysis by obtaining the $H^2$-estimates, for which we employ the original form \eqref{NSP11}-\eqref{NSP22}.

\section{A modulated relative functional for the NSP system}
The method of $a$-contraction with shifts is based on the relative entropy method, first introduced in \cite{Daf,DiP}. To apply this method to the NSP system, we first derive the relative entropy of the isothermal Navier-Stokes equations. Recalling \eqref{trf}, we find that when $\lambda=0$, the NSP system \eqref{NSP} is reduced to the isothermal Navier-Stokes equations:
\begin{equation} \label{NS}
\begin{aligned}
& v_t - u_x = 0, \\
& u_t + \tilde{p}(v)_x  = \left( \frac{u_x}{v} \right)_x,
\end{aligned}
\end{equation}
where $\tilde{p}(v)$ is defined in \eqref{p(v)}. Note that \eqref{NS} has a mathematical entropy (as a mechanical energy)
\[
\tilde{\eta}(U) = \frac{u^2}{2} + Q(v),
\]
where $U=(v,u)$ and the internal energy $Q(v)$ is defined as
\[
Q(v) := -2 \log v.
\]
Then, the relative entropy between $U$ and $\bar{U}=(\bar{v},\bar{u})$ is defined by
\[
\tilde{\eta}(U |\bar{U}) := \tilde{\eta}(U) - \tilde{\eta}(\bar U) - D\eta(\bar U) \cdot (U-\bar U) =  \frac{ \lvert u-\bar{u} \rvert^2}{2} + Q(v|\bar{v}),
\]
where the relative internal energy $Q(v|\bar{v})$ is given by
\[
Q(v|\bar{v}) = Q(v) - Q(\bar{v}) - Q'(\bar{v})(v-\bar{v}).
\]
In general, the relative quantity associated to a scalar function $f : \mathbb{R} \to \mathbb{R}$ is defined as
\[
f(w_1|w_2) : = f(w_1) - f(w_2) - f'(w_2) (w_1-w_2).
\]

We now define a modulated functional based on the relative entropy $\tilde{\eta}(U|\bar{U})$ for \eqref{NS}. Specifically, to study long-time behavior towards the shock profile solution to \eqref{NSP}, we introduce the modulated relative functional $\eta(W|\bar{W}^X)$ between a solution $W:=(v,u,\phi)$ to \eqref{NSP} and the shock profile $\bar{W}^X := (\bar{v}^X,\bar{u}^X,\bar{\phi}^X)$ shifted by $X(t)$, as follows:
\begin{equation} \label{relentropy}
\begin{split}
\eta (W | \bar{W}^X) & := \frac{\lvert u - \bar{u}^X \rvert^2 }{2} + Q(v | \bar{v}^X) - \frac{(v - \bar{v}^X)(\phi-\bar{\phi}^X)_{xx}}{(\bar{v}^X)^2} \\
& \quad + \frac{ e^{-\bar{\phi}^X} \lvert (\phi-\bar{\phi}^X)_{xx} \rvert^2}{2(\bar{v}^X)^3 } + \frac{ e^{-\bar{\phi}^X} \lvert (\phi - \bar{\phi}^X)_x \rvert^2}{2(\bar{v}^X)^2},
\end{split}
\end{equation}
where the shifted shock is defined by
\begin{equation} \label{shsh}
(\bar{v}^X,\bar{u}^X,\bar{\phi}^X)(t,x) := (\bar{v},\bar{u},\bar{\phi})(x - \sigma t - X(t)).
\end{equation}
Before investigating the property of the functional $\eta(W|\bar{W}^X)$, we first consider the relative quantities for the modified pressure $\tilde{p}(v) = 2/v$ and the internal energy $Q(v) = - 2\log v$. These quantities satisfy the bounds presented in the following lemma:

\begin{lemma} \label{Lemma.2.1}
Let $v_- > 0$ be a given constant. Then, there exist positive constants $C>0$ and $\delta_*>0$ such that for any $0 < \bar{\delta} < \delta_*$ and any $(v,\bar{v}) \in \mathbb{R}_+^2 $ satisfying  $\lvert v - \bar{v} \rvert < \bar{\delta} $ and $ \lvert \bar{v} - v_- \rvert < \bar{\delta} $, the following bounds hold:
\begin{subequations}
\begin{align}
\label{relbd1} \tilde{p} ( v | \bar{v}) & \leq  \tilde{p}(\bar{v})^{-1} \lvert \tilde{p}(v) - \tilde{p}(\bar{v}) \rvert^2 + C \lvert \tilde{p}(v) - \tilde{p}(\bar{v}) \rvert^3,\\
\label{relbd2} Q (v | \bar{v}) & \geq \frac{1}{\tilde{p}(\bar{v})^2} \lvert \tilde{p}(v) - \tilde{p}(\bar{v}) \rvert^2 - \frac{4}{3 \tilde{p}(\bar{v})^3} (\tilde{p}(v) - \tilde{p}(\bar{v}) )^3, \\
\label{relbd3} Q (v | \bar{v}) & \geq \frac{1}{\bar{v}^2} \lvert v - \bar{v} \rvert^2 - \frac{2}{3 \bar{v}^3} ( v - \bar{v} )^3, \\
\label{relbd4} Q (v | \bar{v}) & \leq \frac{1}{\bar{v}^2} \lvert v - \bar{v} \rvert^2 + C \lvert v - \bar{v} \rvert^3.
\end{align}
\end{subequations}
\end{lemma}

\begin{proof}
For the proof of \eqref{relbd1} and \eqref{relbd2}, we refer to the proof of Lemma~2.6 in \cite{KV3}. Although the proof in \cite{KV3} is presented for the case where the pressure $p(v)$ is given by the $\gamma$-law with $\gamma>1$, the same computations apply to the isothermal case, where $\gamma=1$. Thus, it remains only to show the bounds \eqref{relbd3} and \eqref{relbd4}.

Consider $\delta_* \leq v_- / 2$. Then, by the assumptions, we have $\lvert \bar{v} - v_- \rvert < v_- /2$ which yields the following upper and lower bounds on $\bar{v}$:
\[
\frac{v_-}{2} < \bar{v} < \frac{3 v_-}{2}.
\]
On the other hand, the relative quantity $Q(v|\bar{v})$ is written as
\[
Q(v|\bar{v}) = - 2 \log v + 2 \log \bar{v} + \frac{2}{\bar{v}} (v-\bar{v}).
\]
Setting $F_1(v) := Q(v|\bar{v})$, we have
\[
F_1'(v) = - \left( \frac{2}{v} - \frac{2}{\bar{v}} \right), \quad F_1''(v) = \frac{2}{v^2}, \quad F_1'''(v) = - \frac{4}{v^3}, \quad F^{(4)}_1 = \frac{12}{v^4}.
\]
Thus, we obtain the following expansion of $Q(v|\bar{v})$ around $v = \bar{v}$ by applying Taylor's theorem:
\[
Q(v|\bar{v}) = F_1(v) = \frac{1}{\bar{v}^2} ( v - \bar{v} )^2 - \frac{2}{3\bar{v}^3} (v-\bar{v})^3 + \frac{F_1^{(4)}(\bar{v})}{4!} (v-\bar{v})^4 + \frac{F_1^{(5)}(v_*)}{5 !} (v-\bar{v})^5,
\]
where $v_*$ lies between $v$ and $\bar{v}$. By the boundedness of $\bar{v}$, we have $ \textstyle F_1^{(4)}(\bar{v}) > \frac{64}{27 v_-^4} > 0$, and therefore we obtain the bound \eqref{relbd3} for sufficiently small $\bar{\delta} < \delta_*$.

To prove \eqref{relbd4}, we write the expansion of $Q(v|\bar{v})$:
\[
Q(v|\bar{v})  = \frac{1}{\bar{v}^2} ( v - \bar{v} )^2 - \frac{2}{3v_*^3} (v-\bar{v})^3
\]
for $v_*$ on the ray from $\bar{v}$ to $v$. Here, if $\bar{\delta}$ is sufficiently small, then $v_*$ is bounded below, and there exists a positive constant $C>0$ such that the bound \eqref{relbd4} holds.
\end{proof}

Note that the result of the above lemma implies that the bounds \eqref{relbd1}-\eqref{relbd4} are valid for the solution $v$ and the shock $\bar{v}$ considered in our problem. More precisely, these bounds hold for $v$ and $\bar{v}^X$ when $\lVert v - \bar{v}^X \rVert_{L^\infty}$ and $ \lvert v_+ - v_- \rvert$ are sufficiently small. With this, we conclude this section by establishing the key property of the modulated relative functional $\eta(W|\bar{W}^X)$.

\begin{lemma} \label{entsim}
Under the assumptions in Proposition \ref{Apriori}, there exist positive constants $c>0$ and $C>0$ such that 
\[
\eta(W|\bar{W}^X) \geq c \left( \lvert u - \bar{u}^X \rvert^2 + \lvert v - \bar{v}^X \rvert^2 + \lvert ( \phi - \bar{\phi}^X )_{xx} \rvert^2 + \lvert ( \phi - \bar{\phi}^X )_x \rvert^2 \right)
\]
and
\[
\eta(W|\bar{W}^X) \leq C \left( \lvert u - \bar{u}^X \rvert^2 + \lvert v - \bar{v}^X \rvert^2 + \lvert ( \phi - \bar{\phi}^X )_{xx} \rvert^2 + \lvert ( \phi - \bar{\phi}^X )_x \rvert^2 \right).
\]
\end{lemma}

\begin{proof}
By the definition \eqref{relentropy} and the bound \eqref{relbd3}, we have
\begin{equation} \label{ent00}
\begin{split}
\eta (W |\bar{W}^X) & \geq \frac{\lvert u - \bar{u}^X \rvert^2}{2} + \frac{\lvert v - \bar{v}^X \rvert^2}{(\bar{v}^X)^2} - \frac{ (v-\bar{v}^X)(\phi - \bar{\phi}^X)_{xx}}{(\bar{v}^X)^2} \\
& \quad + \frac{ e^{-\bar{\phi}^X}\lvert ( \phi - \bar{\phi}^X )_{xx} \rvert^2}{2(\bar{v}^X)^3} + \frac{ e^{-\bar{\phi}^X} \lvert ( \phi - \bar{\phi}^X )_x \rvert^2}{2(\bar{v}^X)^2} - \frac{2}{3(\bar{v}^X)^3} (v-\bar{v}^X)^3.
\end{split}
\end{equation}
Note that, by the Poisson equation of \eqref{shODE}, it holds that
\[
\bigg\lvert \frac{e^{-\bar{\phi}^X}}{\bar{v}^X} \bigg\rvert = \bigg\lvert 1 - \frac{e^{-\bar{\phi}^X}}{\bar{v}^X} \left( \frac{\bar{\phi}^X_x}{\bar{v}^X} \right)_x \bigg\rvert \geq 1 - C \delta_S^2 \geq c
\]
for some constant $c>0$, provided that $\delta_S$ is sufficiently small. Thus, the quadratic terms on the right-hand side of \eqref{ent00} are bounded below as follows:
\[
\begin{split}
& \frac{\lvert v - \bar{v}^X \rvert^2}{(\bar{v}^X)^2} - \frac{ (v-\bar{v}^X )(\phi - \bar{\phi}^X )_{xx}}{(\bar{v}^X)^2} + \frac{ e^{-\bar{\phi}^X}\lvert (\phi - \bar{\phi}^X )_{xx} \rvert^2}{2(\bar{v}^X)^3} \\
& \quad \geq  \frac{\lvert v - \bar{v}^X \rvert^2}{(\bar{v}^X)^2} - \frac{ (v-\bar{v}^X)(\phi - \bar{\phi}^X)_{xx}}{(\bar{v}^X)^2} + \frac{ c \lvert ( \phi - \bar{\phi}^X )_{xx} \rvert^2}{2(\bar{v}^X)^2} \\
& \quad \geq  c \left( \lvert v- \bar{v}^X \rvert^2 + \lvert (\phi - \bar{\phi}^X )_{xx} \rvert^2 \right),
\end{split}
\]
where $c>0$ is a generic constant. Using this, we have
\[
\begin{split}
\eta(W|\bar{W}^X) & \geq c \left( \lvert u - \bar{u}^X \rvert^2 + \lvert v - \bar{v}^X \rvert^2 + \lvert ( \phi - \bar{\phi}^X )_{xx} \rvert^2 + \lvert ( \phi - \bar{\phi}^X )_x \rvert^2 \right) \\
& \quad - C \lVert v-\bar{v}^X \rVert_{L^\infty} (v-\bar{v}^X)^2.
\end{split}
\]
Since the norm $\lVert v-\bar{v}^X \rVert_{L^\infty}$ is controlled by the small parameter $\varepsilon_1$ (see Lemma~\ref{lemma:4.6}), we can conclude that there exists a constant $c>0$ such that
\[
\eta(W|\bar{W}^X) \geq c \left( \lvert u - \bar{u}^X \rvert^2 + \lvert v - \bar{v}^X \rvert^2 + \lvert ( \phi - \bar{\phi}^X )_{xx} \rvert^2 + \lvert ( \phi - \bar{\phi}^X )_x \rvert^2 \right).
\]
On the other hand, the upper bound on $\eta(W|\bar{W}^X)$ is obtained by using the bound \eqref{relbd4} and applying Young's inequality to the third term of \eqref{relentropy}.
\end{proof}

\section{Estimate of the weighted relative functional with the shift}

In this section, we estimate the modulated relative functional, defined in \eqref{relentropy}, between the solution to the NSP system and the shifted shock profile by using the method of $a$-contraction. This estimate yields an energy inequality for perturbations around the shock. We define the perturbation as
\[
(\tilde{v},\tilde{u},\tilde{\phi}) (t,x) : = (v,u,\phi)(t, x) - (\bar{v}^X,\bar{u}^X,\bar{\phi}^X)(t,x),
\]
where the shifted shock $(\bar{v}^X, \bar{u}^X, \bar{\phi}^X)$, defined in \eqref{shsh}, satisfies
\begin{subequations}
\begin{align}
& \label{seqn1} \bar{v}^X_t - \bar{u}^X_x = -\dot{X}(t) \bar{v}^X_x,\\
\begin{split}
& \label{seqn2} \bar{u}^X_t +  \tilde p(\bar{v}^X)_x - \left(\frac{\bar{u}^X_x}{\bar{v}^X} \right)_x  = \Phi(\bar{v}^X,\bar\phi^X)_x - \dot{X}(t) \bar{u}^X_x, 
\end{split} \\
& \label{seqn3} - \left(\frac{\bar{\phi}^X_x}{\bar{v}^X} \right)_x =1-\bar{v}^X e^{\bar{\phi}^X}.
\end{align}
\end{subequations}
Then, the perturbation equations are given by
\begin{subequations}
\begin{align}
& \label{1a'} \tilde{v}_t - \tilde{u}_x = \dot{X}(t) \bar{v}^X_x,\\
\begin{split}
& \label{1b'} \tilde{u}_t + \left( p(v) - p(\bar{v}^X) \right)_x - \left(\frac{u_x}{v}-\frac{\bar{u}^X_x}{\bar{v}^X} \right)_x  = - \left( \frac{\phi_x}{v} - \frac{\bar{\phi}^X_x}{\bar{v}^X} \right) + \dot{X}(t) \bar{u}^X_x, 
\end{split} \\
& \label{1c'} - \left(\frac{\phi_x}{v}-\frac{\bar{\phi}^X_x}{\bar{v}^X} \right)_x =\bar{v}^X e^{\bar{\phi}^X} - v e^{\phi}
\end{align}
\end{subequations}
with the initial data
\begin{equation} \label{pic}
(\tilde{v},\tilde{u})(0,x) = (\tilde{v}_0,\tilde{u}_0)(x) := (v_0 - \bar{v}, u_0 - \bar{u})(x).
\end{equation}
Note that the second equation $\eqref{1b'}$ can be written in the divergence form
\begin{equation} \label{1b}
\tilde{u}_t + \left( \tilde{p}(v) - \tilde{p}(\bar{v}^X) \right)_x - \left(\frac{u_x}{v}-\frac{\bar{u}^X_x}{\bar{v}^X} \right)_x  =  \left( \Phi(v,\phi) - \Phi(\bar{v}^X,\bar{\phi}^X) \right)_x + \dot{X}(t) \bar{u}^X_x,
\end{equation}
where $\Phi$ is defined as \eqref{Phi}.

The goal of this section is to establish the following lemma.

\begin{lemma} \label{RE}
Under the assumptions in Proposition \ref{Apriori}, there exists a positive constant $C>0$ such that
\begin{equation} \label{Ree}
\begin{split}
& \lVert (\tilde{v},\tilde{u}) (t,\cdot) \rVert_{L^2}^2 + \lVert \tilde{\phi} (t,\cdot) \rVert_{H^2}^2 + \int_0^t \left( \delta_S \lvert \dot{X} \rvert^2 + G_1 + G^S + D \right) \, d\tau \\
& \leq C \left( \lVert ( \tilde{v}_0,\tilde{u}_0 ) \rVert_{L^2}^2 + \lVert \tilde{v}_{0x} \rVert_{L^2}^2 \right) \\
& \quad  + C \left( \sqrt{\delta_1} + \varepsilon_1 \right) \int_0^t \int_\mathbb{R} \left( \tilde{v}_x^2 + \bar{v}^X_x \tilde{\phi}^2 + \tilde{\phi}_x^2 + \tilde{\phi}_{xx}^2 + \tilde{\phi}_{xt}^2 + \tilde{\phi}_{xxt}^2 \right) \, dx d\tau
\end{split}
\end{equation}
for all $t \in [0,T]$, where 
\[
\begin{split}
G_1 & := \frac{\sigma}{\sqrt{\delta_S}} \int_\mathbb{R} \bar{v}^X_x \bigg \lvert \tilde{p}(v) - \tilde{p}(\bar{v}^X) - \frac{\tilde{u}}{2C_*} \bigg \rvert^2 \, dx, \quad G^S := \int_\mathbb{R} \bar{v}^X_x \tilde{u}^2 \, dx, \quad D := \int_\mathbb{R} \tilde{u}_x^2 \, dx.
\end{split}
\]
\end{lemma}

\subsection{Construction of weight function}
To prove Lemma~\ref{RE}, we establish weighted estimates for the functional $\eta(W|\bar{W}^X)$, following the theory of $a$-contraction with shifts. For this purpose, we define the weight function $a=a(t,x)$ as
\begin{equation} \label{a}
a(t,x) := 1 + \frac{u_- - \bar{u}(x-\sigma t)}{\sqrt{\delta_S}},
\end{equation}
where $\delta_S$ is the shock strength defined in Proposition \ref{Prop.1.1}. Then, the weight function $a$ and its $x$-derivative satisfy the bounds
\begin{equation} \label{abound}
1 \leq a \leq 1 + \sqrt{\delta_S} < C, \quad \partial_x a = - \frac{\bar{u}'}{\sqrt{\delta_S}} = \frac{\sigma \bar{v}'}{\sqrt{\delta_S}} > 0.
\end{equation}

\subsection{The identity for the weighted relative functional}
In this and the subsequent subsections, we estimate the following quantity:
\begin{equation} \label{went}
\int_\mathbb{R} a^X(t,x) \eta \left( W(t,x) | \bar{W}^X(t,x) \right) \, dx,
\end{equation}
where the shifted weight function $a^X(t,x)$ is defined by
\[
a^X(t,x) := a(t,x-X(t)).
\]
We begin by writing the identity for \eqref{went}. 
\begin{lemma} \label{Id}
Let $a$ be the weight function defined by \eqref{a} and $X:[0,T] \rightarrow \mathbb{R}$ be any Lipschitz continuous function. Let $W=(v,u,\phi)$ be a solution to \eqref{NSP}, and $\bar{W}^X=(\bar{v}^X,\bar{u}^X,\bar{\phi}^X)$ be the shifted shock. Then, the following identity holds:
\begin{equation} \label{Idd}
\begin{split}
\frac{d}{dt} \int_\mathbb{R} a^X \eta(W|\bar{W}^X) \, dx & = \dot{X}(t) Y(W) + \mathcal{J}^{bad} (W) - \mathcal{J}^{good}(W) + \sum_{j=1}^6 \mathcal{P}_{j}(W),
\end{split}
\end{equation}
where the terms $Y(W)$, $\mathcal{J}^{bad}(W)$, and $\mathcal{J}^{good}(W)$ are defined by
\[
\begin{split}
Y(W) & := - \int_\mathbb{R} a^X_x \eta(W|\bar{W}^X) \, dx + \int_\mathbb{R} a^X \left( \bar{u}^X_x \tilde{u} - \tilde{p}'(\bar{v}^X)\bar{v}^X_x \tilde{v} - \frac{\bar{v}^X_x\tilde{\phi}_{xx}}{(\bar{v}^X)^2} \right) \, dx, \\
\mathcal{J}^{bad}(W) & := \int_\mathbb{R} a^X_x \left(\tilde{p}(v)-\tilde{p}(\bar{v}^X) \right) \tilde{u} \, dx - \int_\mathbb{R} a^X \bar{u}^X \tilde{p}(v|\bar{v}^X) \, dx \\
& \quad - \int_\mathbb{R} a^X_x \frac{\tilde{u}\tilde{u}_x}{v} \, dx + \int_\mathbb{R} a^X_x \frac{\bar{u}^X_x \tilde{v} \tilde{u}}{v\bar{v}^X} \, dx + \int_\mathbb{R} a^X \frac{\bar{u}^X_x \tilde{v} \tilde{u}_x}{v\bar{v}^X} \, dx,
\end{split}
\]
and
\[
\begin{split}
\mathcal{J}^{good}(W) & := \frac{\sigma}{2} \int_\mathbb{R} a^X_x \tilde{u}^2 \, dx + \sigma \int_\mathbb{R} a^X_x Q(v|\bar{v}^X) \, dx + \int_\mathbb{R} a^X \frac{\tilde{u}_x^2}{v} \, dx \\
& \quad + \frac{\sigma}{2} \int_\mathbb{R} a^X_x \frac{e^{-\bar{\phi}^X}\tilde{\phi}_{xx}^2}{(\bar{v}^X)^3} \, dx + \frac{\sigma}{2} \int_\mathbb{R} a^X_x \frac{e^{-\bar{\phi}^X}\tilde{\phi}_x^2}{(\bar{v}^X)^2} \, dx,
\end{split}
\]
respectively, and the terms $\mathcal{P}_{j}(W)$ are as follows:
\[
\begin{split}
\mathcal{P}_1 (W) & := \int_\mathbb{R} a^X_x \frac{\tilde{u} \tilde{\phi}_{xx}}{(\bar{v}^X)^2} \, dx + \int_\mathbb{R} a^X_x \left( \frac{ \sigma \tilde{v} \tilde{\phi}_{xx}}{(\bar{v}^X)^2} + \frac{\tilde{v}\tilde{\phi}_{xt}}{(\bar{v}^X)^2} \right) \, dx, \\
\mathcal{P}_2 (W) & := - \int_\mathbb{R} a^X \frac{2\bar{\phi}^X_{xx} \tilde{u}_x \tilde{v}}{v^2\bar{v}^X} \, dx - \int_\mathbb{R} a^X \left( \frac{2\sigma \bar{v}^X_x \tilde{v} \tilde{\phi}_{xx}}{(\bar{v}^X)^3} + \frac{2\bar{v}^X_x \tilde{v} \tilde{\phi}_{xt}}{(\bar{v}^X)^3} \right) \, dx \\
& \quad - \int_\mathbb{R} a^X \frac{\bar{v}^X_x \tilde{\phi}_{xt} \tilde{\phi}}{(\bar{v}^X)^2} \, dx + \int_\mathbb{R} a^X \left( \frac{\bar{\phi}^X_x}{\bar{v}^X} \right)_x \frac{e^{-\bar{\phi}^X}\tilde{v} \tilde{\phi}_{xxt}}{v(\bar{v}^X)^2} \, dx, \\
\mathcal{P}_3 (W) & := - \int_\mathbb{R} a^X_x \tilde{u} \left( \frac{2\bar{\phi}^X_{xx}\tilde{v}}{v^2\bar{v}^X} + \frac{\bar{\phi}^X_x \tilde{\phi}_x}{v^2} - \frac{(\bar{\phi}^X_x)^2 \tilde{v}}{v^2\bar{v}^X} + \frac{\bar{v}^X_x \tilde{\phi}_x}{v^3} + \frac{\bar{\phi}^X_x \tilde{v}_x}{v^3} - \frac{3\bar{v}^X_x \bar{\phi}^X_x \tilde{v}}{v^3 \bar{v}^X}  \right) \, dx \\
& \quad - \int_\mathbb{R} a^X_x \frac{e^{-\bar{\phi}^X}\tilde{\phi}_{xt}\tilde{\phi}_{xx}}{(\bar{v}^X)^3} \, dx + \int_\mathbb{R} a^X_x \left( \frac{\bar{\phi}^X_x}{\bar{v}^X} \right)_x \frac{e^{-\bar{\phi}^X}\tilde{\phi}_{xt}\tilde{v}}{v(\bar{v}^X)^2} \, dx,
\end{split}
\]

\[
\begin{split}
\mathcal{P}_4 (W) & := - \int_\mathbb{R}  a^X \tilde{u}_x \left( \frac{\bar{\phi}^X_x \tilde{\phi}_x}{v^2} - \frac{(\bar{\phi}^X_x)^2 \tilde{v}}{v^2\bar{v}^X} + \frac{\bar{v}^X_x \tilde{\phi}_x}{v^3} + \frac{\bar{\phi}^X_x \tilde{v}_x}{v^3} - \frac{3\bar{v}^X_x \bar{\phi}^X_x \tilde{v}}{v^3 \bar{v}^X}  \right) \, dx \\
& \quad + \int_\mathbb{R} a^X \frac{2 e^{-\bar{\phi}^X}\bar{v}^X_x \tilde{\phi}_{xt} \tilde{\phi}_{xx}}{(\bar{v}^X)^4} \, dx - \sigma \int_\mathbb{R} a^X \left( \bigg( \frac{e^{-\bar{\phi}^X}}{(\bar{v}^X)^3} \bigg)_x \frac{\tilde{\phi}_{xx}^2}{2} + \bigg( \frac{e^{-\bar{\phi}^X}}{(\bar{v}^X)^2} \bigg)_x \frac{\tilde{\phi}_x^2}{2} \right) \, dx \\
& \quad + \int_\mathbb{R} \left( a^X_x \frac{e^{-\bar{\phi}^X} \tilde{\phi}_{xt}}{(\bar{v}^X)^2} + a^X \frac{e^{-\bar{\phi}^X} \tilde{\phi}_{xxt}}{(\bar{v}^X)^2} - a^X \frac{2 e^{-\bar{\phi}^X} \bar{v}^X_x \tilde{\phi}_{xt}}{(\bar{v}^X)^3} \right) \\
& \qquad \qquad \times \left( \frac{\bar{v}^X_x \tilde{\phi}_x}{(\bar{v}^X)^2} + \frac{\bar{\phi}^X_x \tilde{v}_x}{v\bar{v}^X} - \frac{\bar{v}^X_x \bar{\phi}^X_x \tilde{v}}{v^2 \bar{v}^X} \right) \, dx \\
& \quad - \int_\mathbb{R} a^X \left( \frac{\bar{\phi}^X_x}{\bar{v}^X} \right)_x \frac{2 e^{-\bar{\phi}^X} \bar{v}^X_x \tilde{\phi}_{xt} \tilde{v}}{v(\bar{v}^X)^3} \, dx - \int_\mathbb{R} a^X \left( \frac{\bar{\phi}^X_x}{\bar{v}^X} \right)_x \frac{e^{-\bar{\phi}^X} \tilde{\phi}_x \tilde{\phi}_{xt}}{(\bar{v}^X)^2} \, dx, \\
\mathcal{P}_5 (W) & := - \dot{X}(t) \int_\mathbb{R} a^X \left( \frac{2\bar{v}^X_x \tilde{v} \tilde{\phi}_{xx}}{(\bar{v}^X)^3} + \bigg( \frac{e^{-\bar{\phi}^X}}{(\bar{v}^X)^3} \bigg)_x \frac{\tilde{\phi}_{xx}^2}{2} + \bigg( \frac{e^{-\bar{\phi}^X}}{(\bar{v}^X)^2} \bigg)_x \frac{\tilde{\phi}_x^2}{2} \right) \, dx,
\end{split}
\]
and
\[
\begin{split}
\mathcal{P}_6 (W) & := - \int_\mathbb{R} \left( a^X_x \tilde{u} + a^X \tilde{u}_x \right) \left( \frac{\tilde{\phi}_x^2}{2v^2} - \frac{(\bar{\phi}^X_x)^2 \tilde{v}^2}{2v^2(\bar{v}^X)^2} + \frac{\bar{\phi}^X_{xx} \tilde{v}^2}{v^2 (\bar{v}^X)^2} + \frac{ \tilde{v}^2 \tilde{\phi}_{xx}}{v^2 (\bar{v}^X)^2} + \frac{2 \tilde{v} \tilde{\phi}_{xx}}{v^2 \bar{v}^X} \right)  \, dx \\
& \quad - \int_\mathbb{R} \left( a^X_x \tilde{u} + a^X \tilde{u}_x \right) \left( \frac{\tilde{v}_x \tilde{\phi}_x}{v^3} - \frac{\bar{v}^X_x \bar{\phi}^X_x \tilde{v}^3}{v^3 (\bar{v}^X)^3} - \frac{3\bar{v}^X_x \bar{\phi}^X_x \tilde{v}^2}{v^3(\bar{v}^X)^2} \right) \, dx \\
& \quad + \int_\mathbb{R} \left( a^X_x \frac{e^{-\bar{\phi}^X} \tilde{\phi}_{xt}}{(\bar{v}^X)^2} + a^X \frac{e^{-\bar{\phi}^X} \tilde{\phi}_{xxt}}{(\bar{v}^X)^2} - a^X \frac{2 e^{-\bar{\phi}^X} \bar{v}^X_x \tilde{\phi}_{xt}}{(\bar{v}^X)^3} \right) \\
& \qquad \qquad \times \left( \frac{\tilde{v}\tilde{\phi}_{xx}}{v\bar{v}^X} + \frac{\tilde{v}_x \tilde{\phi}_x}{v\bar{v}^X} - \frac{\tilde{v}_x \tilde{v} \tilde{\phi}_x}{v^2 \bar{v}^X} - \frac{\bar{v}^X_x \tilde{v} \tilde{\phi}_x}{v^2 \bar{v}^X} - \frac{\bar{v}^X_x \tilde{v} \tilde{\phi}_x}{v(\bar{v}^X)^2} - \frac{\bar{\phi}^X_x \tilde{v} \tilde{v}_x}{v^2 \bar{v}^X} \right) \, dx \\
& \quad + \int_\mathbb{R} a^X \frac{\tilde{\phi}_{xt}}{(\bar{v}^X)^2} \left( \bar{v}^X_x ( 1 + \tilde{\phi} - e^{\tilde{\phi}}) + \bar{v}^X ( 1 + \tilde{\phi} - e^{\tilde{\phi}})_x \right) \, dx \\
& \quad + \int_\mathbb{R} a^X \frac{\tilde{\phi}_{xt}}{(\bar{v}^X)^2} \left( \tilde{v}_x ( 1 - e^{\tilde{\phi}} ) + \tilde{v} ( 1 - e^{\tilde{\phi}} )_x \right) \, dx.
\end{split}
\]
\end{lemma}

\begin{remark} \label{RemP}
In the identity \eqref{Idd}, the terms $\mathcal{P}_j$ represent some errors derived from the electric force term in the momentum equation. These errors are grouped for convenience in their later estimates. Specifically, $\mathcal{P}_1, \dots, \mathcal{P}_4$ are roughly ordered by the size of the coefficients of the perturbations within the integrands. The size is evaluated in terms of the power of $\delta_S$, determined by the bounds $\lvert a^X_x \rvert \leq C \delta_S^{3/2}$ and $\lvert \bar{v}^X_x \rvert \leq C \delta_S^2$, together with \eqref{shbounds}. The term $\mathcal{P}_5$ consists of some terms with $\dot{X}(t)$. Finally, $\mathcal{P}_6$ refers to the terms involving cubic or higher-order products of perturbations. Each $\mathcal{P}_j$ is estimated in the proof of Lemma~\ref{LR3}.
\end{remark}

\begin{proof}[Proof of Lemma~\ref{Id}]
First, we compute the evolution of the relative entropy $\tilde\eta(\cdot|\cdot)$ for the first two equations of \eqref{NSP}, i.e., for the NS system \eqref{NS} with forcing term $\Phi(v,\phi)_x$. For that, we apply the relative entropy method (see  \cite{HKKL, HKL, KV3}) to the system of the two equations:
\begin{equation}\label{eq:NS-abs}
\partial_t U + \partial_x A(U) = 
\begin{pmatrix}
0 \\  \partial_x \Big(\frac{\partial_x u}{v}\Big) + \partial_x \Phi(v,\phi)
\end{pmatrix}, \quad 
U := \begin{pmatrix}
v\\u
\end{pmatrix}, \quad 
A(U) : =\begin{pmatrix}
-u \\ \tilde{p}(v) 
\end{pmatrix}.
\end{equation}
Thus, we obtain the following identity for the relative entropy of the above quantity $U$ and $\bar U^X:= \begin{pmatrix} \bar v^X\\ \bar u^X \end{pmatrix}$ (satisfying \eqref{seqn1}-\eqref{seqn2}):
\begin{equation*}
\begin{split}
& \frac{d}{dt} \int_\mathbb{R} a^X \tilde\eta(U|\bar{U}^X) \, dx \\
& \quad = \dot{X}(t)\bigg( -\int_\mathbb{R}  a^X_\xi \tilde\eta(U|\bar{U}^X)  \,d x +\int_\mathbb{R} a^X D^2\tilde\eta(U)\bar{U}_x^X (U-\bar U)\, dx \bigg) \\
& \qquad + \int_\mathbb{R} a^X_x \left(\tilde{p}(v)-\tilde{p}(\bar{v}^X) \right) \tilde{u} \, dx - \int_\mathbb{R} a^X \bar{u}^X \tilde{p}(v|\bar{v}^X) \, dx  - \sigma \int_\mathbb{R} a^X_x \tilde\eta(U|\bar{U}^X) \, dx \\
& \qquad - \int_\mathbb{R} a^X \frac{\tilde{u}_x^2}{v} \, dx - \int_\mathbb{R} a^X_x \frac{\tilde{u}\tilde{u}_x}{v} \, dx + \int_\mathbb{R} a^X_x \frac{\bar{u}^X_x \tilde{v} \tilde{u}}{v\bar{v}^X} \, dx + \int_\mathbb{R} a^X \frac{\bar{u}^X_x \tilde{v} \tilde{u}_x}{v\bar{v}^X} \, dx  \\
& \qquad +\int_\mathbb{R} a^X \tilde{u} \left( \Phi(v,\phi) - \Phi(\bar{v}^X,\bar{\phi}^X) \right)_x \, dx, 
\end{split}
\end{equation*}
which, together with the definition of $\tilde\eta$, yields
\begin{equation} \label{genent}
\begin{split}
&\frac{d}{dt} \int_\mathbb{R} a^X \Big(\frac{\tilde u}{2} + Q(v|\bar{v}^X) \Big) \, dx \\
& \quad = \dot{X}(t)\bigg( -\int_\mathbb{R}  a^X \Big(\frac{\tilde u}{2} + Q(v|\bar{v}^X) \Big) d x +\int_\mathbb{R} a^X \Big( \bar{u}^X_x \tilde{u} - \tilde{p}'(\bar{v}^X)\bar{v}^X_x \tilde{v} \Big)\,  dx \bigg) \\
& \qquad + \int_\mathbb{R} a^X_x \left(\tilde{p}(v)-\tilde{p}(\bar{v}^X) \right) \tilde{u} \, dx - \int_\mathbb{R} a^X \bar{u}^X \tilde{p}(v|\bar{v}^X) \, dx  - \frac{\sigma}{2} \int_\mathbb{R} a^X_x \tilde{u}^2 \, dx  \\
& \qquad  - \sigma \int_\mathbb{R} a^X_x Q(v|\bar{v}^X) \, dx - \int_\mathbb{R} a^X \frac{\tilde{u}_x^2}{v} \, dx - \int_\mathbb{R} a^X_x \frac{\tilde{u}\tilde{u}_x}{v} \, dx + \int_\mathbb{R} a^X_x \frac{\bar{u}^X_x \tilde{v} \tilde{u}}{v\bar{v}^X} \, dx \\
& \qquad + \int_\mathbb{R} a^X \frac{\bar{u}^X_x \tilde{v} \tilde{u}_x}{v\bar{v}^X} \, dx +\int_\mathbb{R} a^X \tilde{u} \left( \Phi(v,\phi) - \Phi(\bar{v}^X,\bar{\phi}^X) \right)_x \, dx.
\end{split}
\end{equation}
Thus, it remains to compute the last term.

Recalling the definition \eqref{Phi} of $\Phi$, one can obtain by integration by parts
\begin{equation} \label{Phi1}
\begin{split}
\int_\mathbb{R} a^X \tilde{u} \left( \Phi(v,\phi) - \Phi(\bar{v}^X,\bar{\phi}^X) \right)_x \, dx & = - \frac{1}{2} \int_\mathbb{R} \left( a^X \tilde{u} \right)_x \left[ \left( \frac{\phi_x}{v} \right)^2 - \left( \frac{\bar{\phi}^X_x}{\bar{v}^X} \right)^2 \right] \, dx \\
& \quad + \int_\mathbb{R} \left( a^X \tilde{u} \right)_x \left[ \frac{1}{v} \left( \frac{\phi_x}{v} \right)_x - \frac{1}{\bar{v}^X} \left( \frac{\bar{\phi}^X_x}{\bar{v}^X} \right)_x \right] \, dx.
\end{split}
\end{equation}
The first and second terms on the right-hand side of \eqref{Phi1} are expanded as
\[
\begin{split}
& - \frac{1}{2} \int_\mathbb{R} \left( a^X \tilde{u} \right)_x \left[ \bigg( \frac{\phi_x}{v} \bigg)^2 - \bigg( \frac{\bar{\phi}^X_x}{\bar{v}^X} \bigg)^2 \right] \, dx \\
& \quad = - \frac{1}{2} \int_\mathbb{R}  \left( a^X_x \tilde{u} + a^X \tilde{u}_x \right) \left[ \frac{\tilde{\phi}_x^2}{v^2} + \frac{2\bar{\phi}^X_x \tilde{\phi}_x}{v^2} + \bigg( \frac{1}{v^2} - \frac{1}{(\bar{v}^X)^2} \bigg) (\bar{\phi}^X_x)^2 \right] \, dx
\end{split}
\]
and
\[
\begin{split}
& \int_\mathbb{R} \left( a^X \tilde{u} \right)_x \left[ \frac{1}{v} \bigg( \frac{\phi_x}{v} \bigg)_x - \frac{1}{\bar{v}^X} \bigg( \frac{\bar{\phi}^X_x}{\bar{v}^X} \bigg)_x \right] \, dx \\
& \quad = \int_\mathbb{R} \left( a^X_x \tilde{u} + a^X \tilde{u}_x \right) \left[ \frac{\tilde{\phi}_{xx}}{(\bar{v}^X)^2} + \bigg( \frac{1}{v^2} - \frac{1}{(\bar{v}^X)^2} \bigg) \left( \tilde{\phi}_{xx} + \bar{\phi}^X_{xx} \right) \right] \, dx \\
& \qquad - \int_\mathbb{R} \left( a^X_x \tilde{u} + a^X \tilde{u}_x \right) \left[ \frac{\tilde{v}_x\tilde{\phi}_x}{v^3} + \frac{\bar{v}^X_x \tilde{\phi}_x}{v^3} + \frac{\bar{\phi}^X_x\tilde{v}_x}{v^3} + \bigg( \frac{1}{v^3} - \frac{1}{(\bar{v}^X)^3} \bigg) \bar{v}^X_x \bar{\phi}^X_x \right] \, dx,
\end{split}
\]
respectively. Thus, we can write the right-hand side of \eqref{Phi1} as
\begin{equation} \label{Phi2}
RHS = \int_\mathbb{R} a^X \frac{ \tilde{u}_x \tilde{\phi}_{xx}}{( \bar{v}^X )^2} \, dx + \sum_{j=1}^5 \mathcal{P}_{1j},
\end{equation}
where
\[
\begin{split}
\mathcal{P}_{11}(W) & := \int_\mathbb{R} a^X_x \frac{\tilde{u} \tilde{\phi}_{xx}}{(\bar{v}^X)^2} \, dx, \quad \mathcal{P}_{12} (W) := - \int_\mathbb{R} a^X \frac{2\bar{\phi}^X_{xx} \tilde{u}_x \tilde{v}}{v^2\bar{v}^X} \, dx, \\
\mathcal{P}_{13}(W) & := - \int_\mathbb{R}  a^X_x \tilde{u} \left( \frac{2\bar{\phi}^X_{xx}\tilde{v}}{v^2\bar{v}^X} + \frac{\bar{\phi}^X_x \tilde{\phi}_x}{v^2} - \frac{(\bar{\phi}^X_x)^2 \tilde{v}}{v^2\bar{v}^X} + \frac{\bar{v}^X_x \tilde{\phi}_x}{v^3} + \frac{\bar{\phi}^X_x \tilde{v}_x}{v^3} - \frac{3\bar{v}^X_x \bar{\phi}^X_x \tilde{v}}{v^3 \bar{v}^X}  \right) \, dx, \\
\mathcal{P}_{14}(W) & := - \int_\mathbb{R} a^X \tilde{u}_x \left( \frac{\bar{\phi}^X_x \tilde{\phi}_x}{v^2} - \frac{(\bar{\phi}^X_x)^2 \tilde{v}}{v^2\bar{v}^X} + \frac{\bar{v}^X_x \tilde{\phi}_x}{v^3} + \frac{\bar{\phi}^X_x \tilde{v}_x}{v^3} - \frac{3\bar{v}^X_x \bar{\phi}^X_x \tilde{v}}{v^3 \bar{v}^X}  \right) \, dx,
\end{split}
\]
and
\[
\begin{split}
\mathcal{P}_{15}(W) & := - \frac{1}{2} \int_\mathbb{R} \left( a^X_x \tilde{u} + a^X \tilde{u}_x \right) \left( \frac{\tilde{\phi}_x^2}{v^2} - \frac{(\bar{\phi}^X_x)^2 \tilde{v}^2}{v^2(\bar{v}^X)^2} \right) \, dx \\
& \quad - \int_\mathbb{R} \left( a^X_x \tilde{u} + a^X \tilde{u}_x \right) \left( \frac{\bar{\phi}^X_{xx} \tilde{v}^2}{v^2 (\bar{v}^X)^2} + \frac{ \tilde{v}^2 \tilde{\phi}_{xx}}{v^2 (\bar{v}^X)^2} + \frac{2 \tilde{v} \tilde{\phi}_{xx}}{v^2 \bar{v}^X} \right) \, dx \\
& \quad - \int_\mathbb{R} \left( a^X_x \tilde{u} + a^X \tilde{u}_x \right) \left( \frac{\tilde{v}_x \tilde{\phi}_x}{v^3} - \frac{\bar{v}^X_x \bar{\phi}^X_x \tilde{v}^3}{v^3 (\bar{v}^X)^3} - \frac{3\bar{v}^X_x \bar{\phi}^X_x \tilde{v}^2}{v^3(\bar{v}^X)^2} \right) \, dx.
\end{split}
\]
The terms $\mathcal{P}_{11}, \dots, \mathcal{P}_{15}$ are absorbed in $\mathcal{P}_1,\mathcal{P}_2,\mathcal{P}_3,\mathcal{P}_4$, and $\mathcal{P}_6$, respectively. For the first term on the right-hand side of \eqref{Phi2}, we use \eqref{1a'} to obtain
\begin{equation} \label{Phi3}
\begin{split}
\int_\mathbb{R} a^X \frac{\tilde{u}_x \tilde{\phi}_{xx}}{(\bar{v}^X)^2} \, dx & = \int_\mathbb{R} a^X \frac{\tilde{v}_t \tilde{\phi}_{xx}}{(\bar{v}^X)^2} \, dx - \dot{X}(t) \int_\mathbb{R} a^X \frac{\bar{v}^X_x \tilde{\phi}_{xx}}{(\bar{v}^X)^2} \, dx \\
& = \frac{d}{dt} \int_\mathbb{R} a^X \frac{\tilde{v} \tilde{\phi}_{xx}}{(\bar{v}^X)^2} \, dx + \int_\mathbb{R} a^X \frac{\tilde{v}_x \tilde{\phi}_{xt}}{(\bar{v}^X)^2} \, dx \\
& \quad  + \dot{X}(t) \left( \int_\mathbb{R} a^X_x \frac{\tilde{v} \tilde{\phi}_{xx}}{(\bar{v}^X)^2} \, dx - \int_\mathbb{R} a^X \frac{\bar{v}^X_x \tilde{\phi}_{xx}}{(\bar{v}^X)^2} \, dx \right)  + \sum_{j=1}^3 \mathcal{P}_{2j},
\end{split}
\end{equation}
where
\[
\begin{split}
\mathcal{P}_{21}(W) & := \int_\mathbb{R} a^X_x \left( \frac{ \sigma \tilde{v} \tilde{\phi}_{xx}}{(\bar{v}^X)^2} + \frac{\tilde{v}\tilde{\phi}_{xt}}{(\bar{v}^X)^2} \right) \, dx, \quad \mathcal{P}_{22}(W)  := - \int_\mathbb{R} a^X \left( \frac{2\sigma \bar{v}^X_x \tilde{v} \tilde{\phi}_{xx}}{(\bar{v}^X)^3} + \frac{2\bar{v}^X_x \tilde{v} \tilde{\phi}_{xt}}{(\bar{v}^X)^3} \right) \, dx, \\
\mathcal{P}_{23}(W) & := - \dot{X}(t) \int_\mathbb{R} a^X \frac{2\bar{v}^X_x \tilde{v} \tilde{\phi}_{xx}}{(\bar{v}^X)^3} \, dx .
\end{split}
\]
The first term in the last line of \eqref{Phi3} is part of $\dot{X}(t)Y(W)$, and the terms $\mathcal{P}_{21}, \mathcal{P}_{22}, \mathcal{P}_{23}$ are included in $\mathcal{P}_1,\mathcal{P}_2$, and $\mathcal{P}_5$, respectively. Next, in order to rewrite the second term in the second line of \eqref{Phi3}, we use the identity
\[
\begin{split}
\tilde{v}_x & = \left[ e^{-\bar{\phi}^X} \bigg( \frac{\tilde{\phi}_x}{\bar{v}^X} - \frac{\bar{\phi}^X_x \tilde{v}}{v\bar{v}^X} - \frac{\tilde{v} \tilde{\phi}_x}{v\bar{v}^X} \bigg)_x + \bar{v}^X \left( 1 + \tilde{\phi} - e^{\tilde{\phi}} \right) + \tilde{v} \left( 1 - e^{\tilde{\phi}} \right) - \bar{v}^X \tilde{\phi} \right]_x \\
& = \left[ \frac{e^{-\bar{\phi}^X}\tilde{\phi}_{xx}}{\bar{v}^X} - \frac{e^{-\bar{\phi}^X} \bar{v}^X_x \tilde{\phi}_x}{(\bar{v}^X)^2} - e^{-\bar{\phi}^X} \bigg( \frac{\bar{\phi}^X_x \tilde{v}}{v\bar{v}^X} + \frac{\tilde{v} \tilde{\phi}_x}{v\bar{v}^X} \bigg)_x + \bar{v}^X \left( 1 + \tilde{\phi} - e^{\tilde{\phi}} \right) + \tilde{v} \left( 1 - e^{\tilde{\phi}} \right) \right]_x \\
& \quad - \bar{v}^X_x \tilde{\phi} - \bar{v}^X \tilde{\phi}_x, 
\end{split}
\]
which is obtained by expanding $e^{\tilde{\phi}}$ in \eqref{1c'} and differentiating the resultant equation with respect to $x$. Using this, and applying integration by parts, we have
\begin{equation} \label{Phi4}
\begin{split}
\int_\mathbb{R} a^X \frac{\tilde{v}_x\tilde{\phi}_{xt}}{( \bar{v}^X )^2} \, dx & = - \int_\mathbb{R} a^X \frac{e^{-\bar{\phi}^X}}{(\bar{v}^X)^3} \bigg( \frac{\tilde{\phi}_{xx}^2}{2} \bigg)_t \, dx - \int_\mathbb{R} a^X_x \frac{e^{-\bar{\phi}^X}\tilde{\phi}_{xt} \tilde{\phi}_{xx}}{(\bar{v}^X)^3} \, dx + \int_\mathbb{R} a^X \frac{2 e^{-\bar{\phi}^X}\bar{v}^X_x \tilde{\phi}_{xt} \tilde{\phi}_{xx}}{(\bar{v}^X)^4} \, dx \\
& \quad + \int_\mathbb{R} \left( a^X \frac{\tilde{\phi}_{xt}}{(\bar{v}^X)^2} \right)_x \left( \frac{e^{-\bar{\phi}^X}\bar{v}^X_x \tilde{\phi}_x}{(\bar{v}^X)^2} + e^{-\bar{\phi}^X} \bigg( \frac{\bar{\phi}^X_x \tilde{v}}{v\bar{v}^X} + \frac{\tilde{v} \tilde{\phi}_x}{v\bar{v}^X} \bigg)_x  \right) \, dx \\
& \quad + \int_\mathbb{R} a^X \frac{\tilde{\phi}_{xt}}{(\bar{v}^X)^2} \left( \bar{v}^X (1 +\tilde{\phi} - e^{\tilde{\phi}}) + \tilde{v} (1 - e^{\tilde{\phi}}) \right)_x \, dx \\
& \quad - \int_\mathbb{R} a^X \frac{\bar{v}^X_x \tilde{\phi}_{xt} \tilde{\phi}}{(\bar{v}^X)^2} \, dx - \int_\mathbb{R} a^X \frac{1}{\bar{v}^X} \bigg( \frac{\tilde{\phi}_x^2}{2} \bigg)_t \, dx.
\end{split}
\end{equation}
Here, we can rewrite the last term on the right-hand side again as
\begin{equation} \label{Phi5}
- \int_\mathbb{R} a^X \frac{1}{\bar{v}^X} \bigg( \frac{\tilde{\phi}_x^2}{2} \bigg)_t \, dx = - \int_\mathbb{R} a^X \frac{e^{-\bar{\phi}^X}}{(\bar{v}^X)^2} \bigg( \frac{\tilde{\phi}_x^2}{2} \bigg)_t \, dx  - \int_\mathbb{R} a^X \bigg( \frac{\bar{\phi}^X_x}{\bar{v}^X} \bigg)_x \frac{e^{-\bar{\phi}^X}\tilde{\phi}_x \tilde{\phi}_{xt}}{(\bar{v}^X)^2} \, dx
\end{equation}
by using the Poisson equation of \eqref{shODE}. By substituting \eqref{Phi5} into \eqref{Phi4} and expanding the derivatives using the product rule for time differentiation, we obtain
\[
\begin{split}
\int_\mathbb{R} a^X \frac{\tilde{v}_x\tilde{\phi}_{xt}}{( \bar{v}^X )^2} \, dx & = - \frac{d}{dt} \left( \int_\mathbb{R} a^X \frac{e^{-\bar{\phi}^X} \tilde{\phi}_{xx}^2}{2(\bar{v}^X)^3} \, dx + \int_\mathbb{R} a^X \frac{e^{-\bar{\phi}^X}\tilde{\phi}_x^2}{2 (\bar{v}^X)^2} \, dx \right) \\
& \quad - \frac{\sigma}{2}  \int_\mathbb{R} a^X_x \left( \frac{ e^{-\bar{\phi}^X}\tilde{\phi}_{xx}^2}{(\bar{v}^X)^3} + \frac{e^{-\bar{\phi}^X}\tilde{\phi}_x^2}{(\bar{v}^X)^2} \right) \, dx - \dot{X}(t) \int_\mathbb{R} a^X_x \left( \frac{ e^{-\bar{\phi}^X}\tilde{\phi}_{xx}^2}{2(\bar{v}^X)^3} + \frac{e^{-\bar{\phi}^X}\tilde{\phi}_x^2}{2(\bar{v}^X)^2} \right) \, dx  \\ 
& \quad + \sum_{j=1}^6 \mathcal{P}_j - \sum_{j=1}^5 \mathcal{P}_{1j} - \sum_{j=1}^3 \mathcal{P}_{2j},
\end{split}
\]
where $\mathcal{P}_{j}$ are defined as in this lemma, and $\mathcal{P}_{1j},\mathcal{P}_{2j}$ are as above. The second and third terms on the right-hand side are absorbed in $\mathcal{J}^{good}(W)$ and $\dot{X}(t)Y(W)$, respectively. Therefore, collecting all the terms obtained above, we complete the proof.
\end{proof}

\subsection{\texorpdfstring{Maximization on $\tilde{p}(v) - \tilde{p}(\bar{v}^X)$}{Maximization on the perturbation of the modified pressure}}
Note that the term in $\mathcal{J}^{bad}(W)$,
\[
\int_\mathbb{R} a^X_x \left( \tilde{p}(v) - \tilde{p}(\bar{v}^X) \right) \tilde{u} \, dx,
\]
is in the form of a coupling of the zeroth-order perturbations of $\tilde{p}(\bar{v}^X)$ and $\bar{u}^X$. This primary bad term cannot be controlled directly by good terms in $\mathcal{J}^{good}$. To resolve this, we separate $\tilde{u}$ from $\tilde{p}(v) - \tilde{p}(\bar{v}^X)$ by employing the following lemma. This lemma is proved by using the bounds on the relative quantities $\tilde{p}(v|\bar{v}^X)$ and $Q(v|\bar{v}^X)$ in Lemma~\ref{Lemma.2.1}. As the proof is identical to that in \cite{HKKL}, we omit it here.

\begin{lemma} [\cite{HKKL}, Lemma~4.3] \label{Lemma.3.3}
For sufficiently small $\delta_S>0$, it holds that
\begin{equation} \label{Max}
\begin{split}
& - \sigma \int_\mathbb{R} a^X_x Q(v|\bar{v}^X) \, dx -  \int_\mathbb{R} a^X \bar{u}^X_x \tilde{p}(v|\bar{v}^X) \, dx \\
& \quad \leq - C_* \int_\mathbb{R} a^X_x \lvert \tilde{p}(v) - \tilde{p}(\bar{v}^X) \rvert^2 \, dx  \\
& \qquad + C \delta_S \int_\mathbb{R} a^X_x \lvert \tilde{p}(v) - \tilde{p}(\bar{v}^X) \rvert^2 \, dx + C \int_\mathbb{R} a^X_x \lvert \tilde{p}(v) - \tilde{p}(\bar{v}^X) \rvert^3 \, dx,
\end{split}
\end{equation}
where the constant $C_*$ is given by
\begin{equation} \label{C_*}
C_* = \left( 1 - \frac{ \sqrt{\delta_S}}{2} \right) v_- >0.
\end{equation}
\end{lemma}

By the result of this lemma, we have a consequent inequality for the weighted relative functional $\eta(W|\bar{W}^X)$. Precisely, we obtain the following corollary. For detailed computations for the decomposition of $Y(W)$ and the application of the above lemma, we refer to Section 4.3 in \cite{HKKL}.

\begin{corollary}
The following inequality holds:
\begin{equation} \label{RFineq}
\begin{split}
\frac{d}{dt} \int_\mathbb{R} a^X \eta(W|\bar{W}^X) \, dx & \leq \dot{X}(t) \sum_{j=1}^9 Y_j(W) + \sum_{j=1}^6 \mathcal{B}_j(W) + \sum_{j=1}^6 \mathcal{P}_{j}(W) - \sum_{j=1}^4 \mathcal{G}_j(W) - \mathcal{D}(W),
\end{split}
\end{equation}
where the terms $\mathcal{P}_j$ are as in Lemma~\ref{Id}, and $Y_j$, $\mathcal{B}_j$, $\mathcal{G}_j$, $\mathcal{D}$ are defined by
\[
\begin{split}
Y_1 (W) &:= \int_\mathbb{R} a^X \bar{u}^X_x \tilde{u} \, dx, \quad Y_2 (W) = \frac{1}{\sigma} \int_\mathbb{R} a^X p'(\bar{v}^X) \bar{v}^X_x \tilde{u} \, dx, \\
Y_3 (W) & := - \frac{1}{2} \int_\mathbb{R} a^X_x \left( \tilde{u} - 2C_* \left( p(v) - p(\bar{v}^X) \right) \right)\left( \tilde{u} + 2C_* \left( p(v) - p(\bar{v}^X) \right) \right) \, dx, \\
Y_4 (W) &:= - \frac{1}{2} \int_\mathbb{R} a^X_x 4C_*^2 \left(p(v) - p(\bar{v}^X) \right)^2 \, dx - \int_\mathbb{R} a^X_x Q(v|\bar{v}^X) \, dx, \\
Y_5 (W) & := - \int_\mathbb{R} a^X p'(\bar{v}^X) \bar{v}^X_x \left( \tilde{v} + \frac{2C_*}{\sigma} \left( p(v) - p(\bar{v}^X) \right) \right) \, dx, \\
Y_6 (W) & := \int_\mathbb{R} a^X p'(\bar{v}^X) \bar{v}^X_x \frac{2C_*}{\sigma} \left( p(v)-p(\bar{v}^X) -\frac{\tilde{u}}{2C_*} \right) \, dx, \\
Y_7 (W) & := - \int_\mathbb{R} a^X_x \frac{\tilde{v}\tilde{\phi}_{xx}}{(\bar{v}^X)^2} \, dx, \quad Y_8 (W) := \frac{1}{2} \int_\mathbb{R} a^X_x \left( \frac{e^{-\bar{\phi}^X}\tilde{\phi}_{xx}^2}{(\bar{v}^X)^3} + \frac{e^{-\bar{\phi}^X}\tilde{\phi}_x^2}{(\bar{v}^X)^2} \right) \, dx, \\
Y_9 (W)& := \int_\mathbb{R} a^X \frac{\bar{v}^X_x \tilde{\phi}_{xx}}{(\bar{v}^X)^2} \, dx,
\end{split}
\]

\begin{align*}
\mathcal{B}_1 (W) & := \frac{1}{4C_*} \int_\mathbb{R} a^X_x \tilde{u}^2 \, dx, & 
\mathcal{B}_2 (W) & := - \int_\mathbb{R} a^X_x \frac{\tilde{u} \tilde{u}_x}{v} \, dx, \\
\mathcal{B}_3 (W) & := \int_\mathbb{R} a^X_x \frac{\bar{u}^X_x\tilde{u} \tilde{v}}{v\bar{v}^X} \, dx, & 
\mathcal{B}_4 (W) & := \int_\mathbb{R} a^X \frac{\bar{u}^X_x \tilde{u}_x \tilde{v}}{v\bar{v}^X } \, dx, \\
\mathcal{B}_5 (W) & := C \delta_S \int_\mathbb{R} a^X_x \lvert \tilde{p}(v) - \tilde{p}(\bar{v}^X) \rvert^2 \, dx, & 
\mathcal{B}_6 (W) & := C \int_\mathbb{R} a^X_x \lvert \tilde{p}(v) - \tilde{p}(\bar{v}^X) \rvert^3 \, dx,
\end{align*}
and
\begin{align*}
\mathcal{G}_1 (W) & := C_* \int_\mathbb{R} a^X_x \bigg\lvert \tilde{p}(v) - \tilde{p}(\bar{v}^X) - \frac{\tilde{u}}{2C_*} \bigg\rvert^2 \, dx, & \mathcal{G}_2 (W) & := \frac{\sigma}{2} \int_\mathbb{R} a^X_x \tilde{u}^2 \, dx, \\
\mathcal{G}_3 (W) & := \frac{\sigma}{2} \int_\mathbb{R} a^X_x \frac{e^{-\bar{\phi}^X}\tilde{\phi}_x^2}{(\bar{v}^X)^3} \, dx, & \mathcal{G}_4 (W) & := \frac{\sigma}{2} \int_\mathbb{R} a^X_x \frac{e^{-\bar{\phi}^X}\tilde{\phi}_{xx}^2}{(\bar{v}^X)^2} \, dx, \\
\mathcal{D} (W) & : = \int_\mathbb{R} a^X \frac{\tilde{u}_x^2}{v} \, dx.
\end{align*}
\end{corollary}

\subsection{Preliminary estimates}
Before proceeding with the estimates of the terms on the right-hand side of \eqref{RFineq}, we provide some preliminary estimates that will be frequently used throughout the rest of this paper.

\begin{lemma} \label{lemma:4.5}
The shock profile $(\bar{v},\bar{u},\bar{\phi})$ described in proposition \ref{Prop.1.1} satisfies the following bounds:
\begin{equation} \label{shbounds}
\bigg\lvert \frac{d^k \bar{v}}{d\xi^k} \bigg \rvert, \bigg\lvert \frac{d^k\bar{u}}{d\xi^k} \bigg \rvert, \bigg\lvert \frac{d^k\bar{\phi}}{d\xi^k} \bigg \rvert \leq C \bar{v}',
\end{equation}
where $k \in \mathbb{N}$ with $k \geq 1$ and $C>0$ is a generic constant.
\end{lemma}

\begin{proof}
From the proof of Theorem 1.1 in \cite{DLZ}, we directly obtain the following bounds:
\[
\lvert \bar{u}' \rvert, \lvert \bar{\phi}' \rvert, \lvert \bar{\phi}'' \rvert \leq C \bar{v}'.
\]
With these bounds, we can  recursively obtain the desired bounds for higher-order derivatives by using the shock ODE \eqref{shODE}. For brevity, the details of the proof are omitted.
\end{proof}

\begin{lemma} \label{lemma:4.6}
Under the assumptions in Proposition \ref{Apriori}, the solution $(v,u,\phi)$ and the shift $X(t)$ satisfy the bounds
\begin{equation} \label{infty}
\lVert (v-\bar{v}^X, u - \bar{u}^X)(t,\cdot) \rVert_{W^{1,\infty}} + \lVert (\phi - \bar{\phi}^X) (t,\cdot) \rVert_{W^{2,\infty}} \leq C \varepsilon_1
\end{equation}
and
\begin{equation} \label{X}
\lvert \dot{X}(t) \rvert \leq C \lVert (u-\bar{u}^X)(t,\cdot) \rVert_{L^\infty}
\end{equation}
for all $t \in [0,T]$.
\end{lemma}

\begin{proof}
We use the Sobolev inequality to obtain
\begin{equation*}
\lVert ( v-\bar{v}^X, u - \bar{u}^X) \rVert_{W^{1,\infty}} + \lVert (\phi - \bar{\phi}^X) \rVert_{W^{2,\infty}} \leq C \left( \lVert ( v-\bar{v}^X, u - \bar{u}^X) \rVert_{H^2} + \lVert (\phi - \bar{\phi}^X) \rVert_{H^3} \right).
\end{equation*}
The bound \eqref{infty} then follows from the assumption \eqref{eps}. To show \eqref{X}, we estimate the ODE \eqref{shiftODE} as
\begin{equation} \label{Xest}
\begin{split}
\lvert \dot{X}(t) \rvert & \leq \frac{C}{\delta_S} \int_\mathbb{R} a^X \left( \lvert \bar{u}^X_x \rvert \lvert u - \bar{u}^X \rvert + \lvert \tilde{p}'(\bar{v}^X) \rvert \lvert \bar{v}^X_x \rvert \lvert  u-\bar{u}^X \rvert \right) \, dx \\
& \leq \frac{C}{\delta_S} \lVert (u - \bar{u}^X)(t,\cdot) \rVert_{L^\infty} \int_\mathbb{R} \left( \lvert \bar{u}^X_x \rvert + \lvert \bar{v}^X_x \rvert \right) \, dx \\
& \leq C \lVert (u-\bar{u}^X)(t,\cdot) \rVert_{L^\infty},
\end{split}
\end{equation}
which completes the proof.
\end{proof}

\begin{lemma} \label{Lv^2}
Under the assumptions in Proposition \ref{Apriori}, there exists a positive constant $C>0$ such that
\begin{equation} \label{v^2}
\int_\mathbb{R} \bar{v}^X_x \tilde{v}^2 \, dx \leq C \left( G_1 + G^S \right),
\end{equation}
where $G_1$ and $G^S$ are as in Lemma~\ref{RE}.
\end{lemma}

\begin{proof}
Since $C_*$ is given by \eqref{C_*}, there exists a positive constant $c$ such that $C_* > c >0$ for sufficiently small $\delta_S < \delta_1$. Thus, by the definition of $\tilde{p}(\cdot)$, we obtain
\[
\begin{split}
\int_\mathbb{R} \bar{v}^X_x \tilde{v}^2 \, dx & \leq C \int_\mathbb{R} \bar{v}^X_x \lvert \tilde{p}(v) - \tilde{p}(\bar{v}^X) \rvert^2 \, dx \\
& \leq C \left( \int_\mathbb{R} \bar{v}^X_x \bigg\lvert \tilde{p}(v) - \tilde{p}(\bar{v}^X) - \frac{\tilde{u}}{2C_*} \bigg\rvert^2 \, dx + \frac{1}{4C_*^2} \int_\mathbb{R} \bar{v}^X_x \tilde{u}^2 \, dx \right) \\
& \leq C \left( \sqrt{\delta_S} \int_\mathbb{R} a^X_x \bigg\lvert \tilde{p}(v) - \tilde{p}(\bar{v}^X) - \frac{\tilde{u}}{2C_*} \bigg\rvert^2 \, dx + \int_\mathbb{R} \bar{v}^X_x \tilde{u}^2 \, dx \right).
\end{split}
\]
\end{proof}

\subsection{Main estimates}
We define the shift function $X(t)$ by the ODE \eqref{shiftODE}:
\[
\dot{X}(t) = - \frac{M}{\delta_S} (Y_1 + Y_2), \quad X(0) =0.
\]
Then, we can rewrite the term $\dot{X}(t)Y(W)$ in \eqref{RFineq} as
\[
\dot{X}(t) Y(U) = - \frac{\delta_S}{M} \lvert \dot{X} \rvert^2 + \dot{X} \sum_{j=3}^9 Y_j,
\]
which allows us to decompose the right-hand side of \eqref{RFineq} as
\begin{equation} \label{deceta}
\begin{split}
\frac{d}{dt} \int_\mathbb{R} \eta(W|\bar{W}^X) \, dx & \leq \underbrace{ - \frac{\delta_S}{2M}\lvert \dot{X} \rvert^2 + \mathcal{B}_1 - \mathcal{G}_2 - \frac{3}{4} \mathcal{D} }_{=: \mathcal{R}_1} \\
& \quad \underbrace{ - \frac{\delta_S}{2M} \lvert \dot{X} \rvert^2 + \dot{X} \sum_{j=3}^6 Y_j + \sum_{j=2}^6 \mathcal{B}_j - \mathcal{G}_1 - \frac{1}{4} \mathcal{D} }_{=: \mathcal{R}_2} \\
& \quad \underbrace{ + \dot{X} \sum_{j=7}^9 Y_j  + \sum_{j=1}^6 \mathcal{P}_{j} - \mathcal{G}_3 - \mathcal{G}_4  }_{=: \mathcal{R}_3}.
\end{split}
\end{equation}
Now we estimate the terms $\mathcal{R}_1$, $\mathcal{R}_2$, and $\mathcal{R}_3$.

\subsubsection{Estimates of $\mathcal{R}_1$ and $\mathcal{R}_2$}
First, we obtain the bounds on $\mathcal{R}_1 + \mathcal{R}_2$.
\begin{lemma} [\cite{HKKL}]
Under the assumptions in Proposition \ref{Apriori}, there exists a positive constant $C>0$ such that
\begin{equation} \label{R1R2}
\begin{split}
\mathcal{R}_1 + \mathcal{R}_2 \leq - C \left( \delta_S \lvert \dot{X} \rvert^2 + G_1 + G^S + D \right),
\end{split}
\end{equation}
where $G_1$, $G^S$, and $D$ are defined as in Lemma~\ref{RE}.
\end{lemma}

In the proof of this lemma, the estimates of $\mathcal{R}_1$ require a delicate analysis, in which a sharp Poincar\'e-type inequality plays a crucial role. However, the analysis and detailed computations for this part are similar to those for the case of viscous-dispersive shocks in the Navier-Stokes-Korteweg equations, as treated in \cite{HKKL}. Indeed, all the bounds and estimates used in the analysis in \cite{HKKL} can be obtained by using the bounds on the shock profile, \eqref{shderiv} and \eqref{shbounds}. Furthermore, all the terms consisting of $\mathcal{R}_2$ are also handled in \cite{HKKL}. Therefore, we omit the proof of this lemma and refer to Sections~4.4-4.5 of \cite{HKKL}.

\subsubsection{Estimate of $\mathcal{R}_3$}
We next estimate the remaining term $\mathcal{R}_3$ which consists of the terms from the electric force.
\begin{lemma} \label{LR3}
Under the assumptions in Proposition \ref{Apriori}, there exists a positive constant $C>0$ such that
\begin{equation} \label{R3}
\begin{split}
\mathcal{R}_3 & \leq C \left( \sqrt{\delta_1} + \varepsilon_1 \right) \bigg(  \delta_S \lvert \dot{X} \rvert^2 + G_1 + G^S + D + \int_\mathbb{R} \left( \tilde{v}_x^2 + \bar{v}^X_x \tilde{\phi}^2 + \tilde{\phi}_x^2 + \tilde{\phi}_{xx}^2 + \tilde{\phi}_{xt}^2 + \tilde{\phi}_{xxt}^2 \right) \, dx \bigg),
\end{split}
\end{equation}
where $G_1$, $G^S$, and $D$ are defined as in Lemma~\ref{RE}.
\end{lemma}

\begin{proof}
We estimate each term of $\mathcal{R}_3$ by using the results of Lemmas~\ref{lemma:4.5} and \ref{lemma:4.6}.\\
\noindent $\bullet$ Estimate of the term $-\mathcal{G}_3 - \mathcal{G}_4$: Using the bound $\lvert a^X_x \rvert \leq C \delta_S^{3/2}$, we have
\[
\lvert - \mathcal{G}_3 - \mathcal{G}_4 \rvert \leq C \delta_S^{3/2} \int_\mathbb{R} \left( \tilde{\phi}_x^2 + \tilde{\phi}_{xx}^2 \right) \, dx.
\]

\noindent $\bullet$ Estimates of the terms $ \dot{X} Y_j$ for $j=7,8,9$: First, we use the bounds $\lvert \dot{X} \rvert \leq C \lVert \tilde{u} \rVert_{L^\infty}$, $ \lvert a^X_x \rvert \leq C \delta_S^{-1/2} \bar{v}^X_x$, $ \lvert \bar{v}^X_x \rvert \leq C \delta_S^2$ and apply Young's inequality to obtain
\[
\begin{split}
\lvert \dot{X} Y_7 \rvert & \leq \lvert \dot{X} \rvert \int_\mathbb{R} \lvert a^X_x \rvert \lvert \tilde{v} \rvert \lvert \tilde{\phi}_{xx} \rvert \, dx \\
& \leq C \lVert \tilde{u} \rVert_{L^\infty} \left( \int_\mathbb{R} \lvert \tilde{\phi}_{xx} \rvert^2 \, dx + \int_\mathbb{R} \lvert a^X_x \rvert^2 \lvert \tilde{v} \rvert^2 \, dx \right) \\
& \leq C \lVert \tilde{u} \rVert_{L^\infty} \int_\mathbb{R} \left( \bar{v}^X_x \tilde{v}^2 + \tilde{\phi}_{xx}^2 \right) \, dx.
\end{split}
\]
Similarly, using $\lvert \dot{X} \rvert \leq C \lVert \tilde{u} \rVert_{L^\infty}$, we have
\[
\lvert \dot{X} Y_8 \rvert \leq \lvert \dot{X} \rvert \int_\mathbb{R} \left( \lvert \tilde{\phi}_{xx} \rvert^2 + \lvert \tilde{\phi}_x \rvert^2 \right) \, dx \leq C \lVert \tilde{u} \rVert_{L^\infty} \int_\mathbb{R} \left( \tilde{\phi}_{xx}^2 + \tilde{\phi}_x^2 \right) \, dx.
\]
The term $\dot{X} Y_9$ is estimated by using $\bar{v}^X_x \leq C \delta_S^2$ as
\[
\begin{split}
\lvert \dot{X} Y_9 \rvert & \leq C \lvert \dot{X} \rvert \int_\mathbb{R} \lvert \bar{v}^X_x \rvert \lvert \tilde{\phi}_{xx} \rvert \, dx \\
& \leq C \lvert \dot{X} \rvert \sqrt{\int_\mathbb{R} \bar{v}^X_x \, dx } \sqrt{\int_\mathbb{R} \bar{v}^X_x \tilde{\phi}_{xx}^2 \, dx} \\
& \leq C \delta_S^2 \lvert \dot{X} \rvert^2 + \frac{C}{\delta_S^2} \left( \int_\mathbb{R} \bar{v}^X_x \, dx \right) \left( \int_\mathbb{R} \bar{v}^X_x \tilde{\phi}_{xx}^2 \, dx \right) \\
& \leq C \delta_S \left( \delta_S \lvert \dot{X} \rvert^2 + \int_\mathbb{R} \tilde{\phi}_{xx}^2 \, dx \right),
\end{split}
\]
where in the second inequality, we used the H\"older inequality.

\noindent $\bullet$ Estimates of the terms $\mathcal{P}_j$: By using $\lvert a^X_x \rvert \leq C \delta_S^{-1/2} \bar{v}^X_x \leq C \delta_S^{3/2}$ and applying Young's inequality, we obtain
\[
\begin{split}
\lvert \mathcal{P}_1 \rvert & \leq C \int_\mathbb{R} \lvert a^X_x \rvert \left( \lvert \tilde{u} \rvert \lvert \tilde{\phi}_{xx} \rvert + \lvert \tilde{v} \rvert \lvert \tilde{\phi}_{xx} \rvert + \lvert \tilde{v} \rvert \lvert \tilde{\phi}_{xt} \rvert \right) \, dx \\
& \leq C \left( \int_\mathbb{R} \lvert a^X_x \rvert^{1/3} \left( \lvert \tilde{\phi}_{xx} \rvert^2 + \lvert \tilde{\phi}_{xt} \rvert^2 \right) \, dx + \int_\mathbb{R} \lvert a^X_x \rvert^{5/3} \left( \lvert \tilde{v} \rvert^2 + \lvert \tilde{u} \rvert^2 \right) \, dx \right) \\
& \leq C \sqrt{\delta_S}  \int_\mathbb{R} \left( \bar{v}^X_x \tilde{v}^2 + \bar{v}^X_x \tilde{u}^2 + \tilde{\phi}_{xx}^2 + \tilde{\phi}_{xt}^2 \right) \, dx.
\end{split}
\]
The terms $\mathcal{P}_2$, $\mathcal{P}_3$, and $\mathcal{P}_4$ are estimated in a similar way, with the additional use of the bounds $\lvert \bar{\phi}^X_x \rvert, \lvert \bar{\phi}^X_{xx} \rvert \leq C \lvert \bar{v}^X_x \rvert \leq C \delta_S^2$:
\[
\begin{split}
\lvert \mathcal{P}_2 \rvert & \leq C \int_\mathbb{R} \lvert \bar{v}^X_x \rvert  \lvert \tilde{v} \rvert \left( \lvert \tilde{u}_x \rvert + \lvert \tilde{\phi}_{xx} \rvert + \lvert \tilde{\phi}_{xt} \rvert + \lvert \tilde{\phi}_{xxt} \rvert \right) \, dx + C \int_\mathbb{R} \lvert \bar{v}^X_x \rvert \lvert \tilde{\phi} \rvert \lvert \tilde{\phi}_{xt} \rvert \, dx \\
& \leq C \left( \int_\mathbb{R} \lvert \bar{v}^X_x \rvert^{1/2} \left( \lvert \tilde{u}_x \rvert^2 + \lvert \tilde{\phi}_{xx} \rvert^2 + \lvert \tilde{\phi}_{xt} \rvert^2 + \lvert \tilde{\phi}_{xxt} \rvert^2 \right) \, dx + \int_\mathbb{R} \lvert \bar{v}^X_x \rvert^{3/2} \left( \lvert \tilde{v} \rvert^2 + \lvert \tilde{\phi} \rvert^2 \right) \, dx \right) \\
& \leq C \delta_S \int_\mathbb{R} \left( \bar{v}^X_x \tilde{v}^2 + \bar{v}^X_x \tilde{\phi}^2 + \tilde{u}_x^2 + \tilde{\phi}_{xx}^2 + \tilde{\phi}_{xt}^2 + \tilde{\phi}_{xxt}^2 \right) \, dx, \\
\lvert \mathcal{P}_3 \rvert & \leq C \int_\mathbb{R} \lvert a^X_x \rvert \lvert \bar{v}^X_x \rvert \lvert \tilde{u} \rvert \left( \lvert \tilde{v} \rvert + \lvert \tilde{\phi}_x \rvert + \lvert \tilde{v}_x \rvert \right) \, dx + C \int_\mathbb{R} \lvert a^X_x \rvert \lvert \tilde{\phi}_{xt} \rvert \left( \lvert \tilde{\phi}_{xx} \rvert + \lvert \bar{v}^X_x \rvert \lvert \tilde{v} \rvert \right) \, dx \\
& \leq C \delta_S^{3/2} \int_\mathbb{R} \left( \bar{v}^X_x \tilde{v}^2 + \bar{v}^X_x \tilde{u}^2+ \tilde{v}_x^2 + \tilde{\phi}_x^2 + \tilde{\phi}_{xx}^2 + \tilde{\phi}_{xt}^2 \right) \, dx,
\end{split}
\]
and
\[
\begin{split}
\lvert \mathcal{P}_4 \rvert & \leq C \int_\mathbb{R} \lvert \bar{v}^X_x \rvert \left( \lvert \tilde{u}_x \rvert + \lvert \tilde{\phi}_{xt} \rvert + \lvert \tilde{\phi}_{xxt} \rvert \right) \left( \lvert \tilde{\phi}_x \rvert + \lvert \tilde{v}_x \rvert + \lvert \bar{v}^X_x \rvert \lvert \tilde{v} \rvert \right) \, dx \\
& \quad + C \int_\mathbb{R} \lvert \bar{v}^X_x \rvert \lvert \tilde{\phi}_{xx} \rvert \left( \lvert \tilde{\phi}_{xx} \rvert + \lvert \tilde{\phi}_{xt} \rvert \right) \, dx + C \int_\mathbb{R} \lvert \bar{v}^X_x \rvert \lvert \tilde{\phi}_x \rvert^2 \, dx  \\
& \leq C \delta_S^2 \int_\mathbb{R} \left( \bar{v}^X_x \tilde{v}^2 + \tilde{v}_x^2 + \tilde{u}_x^2 + \tilde{\phi}_x^2 + \tilde{\phi}_{xx}^2 + \tilde{\phi}_{xt}^2 + \tilde{\phi}_{xxt}^2 \right) \, dx.
\end{split}
\]
For the term $\mathcal{P}_5$, we use $\lvert \dot{X} \rvert \leq C \lVert \tilde{u} \rVert_{L^\infty}$, as in the estimate of $\dot{X}Y_8$, to obtain
\[
\lvert \mathcal{P}_5 \rvert \leq C \lvert \dot{X} \rvert \int_\mathbb{R} \lvert \bar{v}^X_x \rvert \lvert \tilde{v} \rvert \lvert \tilde{\phi}_{xx} \rvert + \lvert \tilde{\phi}_{xx} \rvert^2 + \lvert \tilde{\phi}_x \rvert^2 \, dx  \leq C \lVert \tilde{u} \rVert_{L^\infty} \int_\mathbb{R} \left( \bar{v}^X_x \tilde{v}^2 + \tilde{\phi}_x^2 + \tilde{\phi}_{xx}^2 \right) \, dx.
\]
Now it remains only to estimate the term $\mathcal{P}_6$. For convenience, we decompose it into three parts as $\textstyle \mathcal{P}_6 = \sum_{j=1}^3 \mathcal{P}_{6j}$ with
\[
\begin{split}
\mathcal{P}_{61} & := - \int_\mathbb{R} \left( a^X_x \tilde{u} + a^X \tilde{u}_x \right) \left( \frac{\tilde{\phi}_x^2}{2v^2} - \frac{(\bar{\phi}^X_x)^2 \tilde{v}^2}{2v^2(\bar{v}^X)^2} + \frac{\bar{\phi}^X_{xx} \tilde{v}^2}{v^2 (\bar{v}^X)^2} + \frac{ \tilde{v}^2 \tilde{\phi}_{xx}}{v^2 (\bar{v}^X)^2} + \frac{2 \tilde{v} \tilde{\phi}_{xx}}{v^2 \bar{v}^X} \right)  \, dx \\
& \quad - \int_\mathbb{R} \left( a^X_x \tilde{u} + a^X \tilde{u}_x \right) \left( \frac{\tilde{v}_x \tilde{\phi}_x}{v^3} - \frac{\bar{v}^X_x \bar{\phi}^X_x \tilde{v}^3}{v^3 (\bar{v}^X)^3} - \frac{3\bar{v}^X_x \bar{\phi}^X_x \tilde{v}^2}{v^3(\bar{v}^X)^2} \right) \, dx, \\
\mathcal{P}_{62} & :=  \int_\mathbb{R} \left( a^X_x \frac{e^{-\bar{\phi}^X}\tilde{\phi}_{xt}}{(\bar{v}^X)^2} + a^X \frac{e^{-\bar{\phi}^X}\tilde{\phi}_{xxt}}{(\bar{v}^X)^2} - a^X \frac{2 e^{-\bar{\phi}^X} \bar{v}^X_x \tilde{\phi}_{xt}}{(\bar{v}^X)^3} \right) \\
& \qquad \quad \times \left( \frac{\tilde{v}\tilde{\phi}_{xx}}{v\bar{v}^X} + \frac{\tilde{v}_x \tilde{\phi}_x}{v\bar{v}^X} - \frac{\tilde{v}_x \tilde{v} \tilde{\phi}_x}{v^2 \bar{v}^X} - \frac{\bar{v}^X_x \tilde{v} \tilde{\phi}_x}{v^2 \bar{v}^X} - \frac{\bar{v}^X_x \tilde{v} \tilde{\phi}_x}{v(\bar{v}^X)^2} - \frac{\bar{\phi}^X_x \tilde{v} \tilde{v}_x}{v^2 \bar{v}^X} \right) \, dx, \\
\mathcal{P}_{63} & := \int_\mathbb{R} a^X \frac{\tilde{\phi}_{xt}}{(\bar{v}^X)^2} \left( \bar{v}^X_x (1+\tilde{\phi} - e^{\tilde{\phi}}) + \bar{v}^X (1+\tilde{\phi} - e^{\tilde{\phi}})_x \right) \, dx \\
& \quad + \int_\mathbb{R} a^X \frac{\tilde{\phi}_{xt}}{(\bar{v}^X)^2} \left( \tilde{v}_x (1 -e^{\tilde{\phi}}) + \tilde{v} (1 -e^{\tilde{\phi}})_x \right) \, dx,
\end{split}
\]
where each part will be estimated separately. Note that $\mathcal{P}_{6j}$ consists of nonlinear terms, specifically involving cubic and higher-order products of perturbations, as explained in Remark \ref{RemP}. Thus, we use the boundedness of $\lVert (\tilde{v},\tilde{u},\tilde{\phi}) \rVert_{W^{1,\infty}}$, together with Young's inequality, to obtain 
\[
\begin{split}
\lvert \mathcal{P}_{61} \rvert & \leq C \int_\mathbb{R} \left( \lvert \tilde{u} \rvert + \lvert \tilde{u}_x \rvert \right) \left( \lvert \tilde{\phi}_x \rvert^2 + \lvert \bar{v}^X_x \rvert \lvert \tilde{v} \rvert^2 + \lvert \bar{v}^X_x \rvert \lvert \tilde{v} \rvert^3 + \lvert \tilde{v}_x \rvert \lvert \tilde{\phi}_x \rvert \right) \, dx \\
& \quad + C \int_\mathbb{R} \left( \lvert \tilde{v} \rvert + \lvert \tilde{v} \rvert^2 \right) \left( \lvert a^X_x \rvert \lvert \tilde{u} \rvert \lvert \tilde{\phi}_{xx} \rvert + \lvert \tilde{u}_x \rvert \lvert \tilde{\phi}_{xx} \rvert \right) \, dx \\
& \leq C \left( \lVert \tilde{v} \rVert_{L^\infty} + \lVert \tilde{u} \rVert_{W^{1,\infty}} \right) \int_\mathbb{R} \left( \bar{v}^X_x\tilde{v}^2 + \bar{v}^X_x \tilde{u}^2 +\tilde{v}_x^2 + \tilde{u}_x^2 + \tilde{\phi}_x^2 + \tilde{\phi}_{xx}^2 \right) \, dx,
\end{split}
\]
\[
\begin{split}
\lvert \mathcal{P}_{62} \rvert & \leq C \int_\mathbb{R} \lvert \tilde{v} \rvert \left( \lvert \tilde{\phi}_{xt} \rvert + \lvert \tilde{\phi}_{xxt} \rvert \right) \left( \lvert \tilde{\phi}_{xx} \rvert + \lvert \tilde{\phi}_x \rvert + \lvert \tilde{v}_x \rvert + \lvert \tilde{v}_x \rvert \lvert \tilde{\phi}_x \rvert \right) \, dx \\
& \quad + C \int_\mathbb{R} \lvert \tilde{v}_x \rvert \lvert \tilde{\phi}_x \rvert \left( \lvert \tilde{\phi}_{xt} \rvert + \lvert \tilde{\phi}_{xxt} \rvert \right) \, dx \\
& \leq C \left( \lVert \tilde{v} \rVert_{L^\infty} + \lVert \tilde{\phi}_x \rVert_{L^\infty} \right) \int_\mathbb{R} \left( \tilde{\phi}_{xxt}^2 + \tilde{\phi}_{xt}^2 + \tilde{\phi}_{xx}^2 + \tilde{\phi}_x^2 + \tilde{v}_x^2 \right) \, dx,
\end{split}
\]
and
\[
\begin{split}
\lvert \mathcal{P}_{63} \rvert & \leq C \int_\mathbb{R} \lvert \tilde{\phi}_{xt} \rvert \left( \lvert \bar{v}^X_x \rvert \lvert \tilde{\phi} \rvert \mathcal{O}(\lvert \tilde{\phi} \rvert) + \lvert \tilde{\phi}_x \rvert \mathcal{O}(\lvert \tilde{\phi} \rvert ) + \lvert \tilde{v}_x \rvert \mathcal{O}(\lvert \tilde{\phi} \rvert) + \lvert \tilde{v} \rvert \lvert \tilde{\phi}_x \rvert + \lvert \tilde{v} \rvert \lvert \tilde{\phi}_x \rvert \mathcal{O}(\lvert \tilde{\phi} \rvert ) \right) \, dx \\
& \leq C \left( \lVert \tilde{v} \rVert_{L^\infty} + \lVert \tilde{\phi} \rVert_{L^\infty} \right) \int_\mathbb{R} \left( \bar{v}^X_x \tilde{\phi}^2 + \tilde{\phi}_{xt}^2 + \tilde{\phi}_x^2 + \tilde{v}_x^2 \right) \, dx.
\end{split}
\]
In the estimate of $\mathcal{P}_{63}$, we used Taylor's expansion of $e^{\tilde{\phi}}$ near $\tilde{\phi}=0$.

Combining all the above estimates, we have
\[
\begin{split}
\lvert \mathcal{R}_3 \rvert & \leq C \left( \sqrt{\delta_S} + \lVert \tilde{v} \rVert_{L^\infty} + \lVert \tilde{u} \rVert_{W^{1,\infty}} + \lVert \tilde{\phi} \rVert_{W^{1,\infty}} \right) \\
& \qquad \times \int_\mathbb{R} \left( \bar{v}^X_x \tilde{v}^2 + \bar{v}^X_x \tilde{u}^2 + \bar{v}^X_x \tilde{\phi}^2 + \tilde{v}_x^2 + \tilde{u}_x^2 + \tilde{\phi}_x^2 + \tilde{\phi}_{xx}^2 + \tilde{\phi}_{xt}^2 + \tilde{\phi}_{xxt}^2 \right) \, dx.
\end{split}
\]
Therefore, by using $\delta_S \leq \delta_1$, \eqref{infty}, and \eqref{v^2}, we obtain the bound \eqref{R3}.
\end{proof}

\subsubsection{Combining the estimates}
We prove Lemma~\ref{RE} by collecting all the estimates. Combining \eqref{deceta} with \eqref{R1R2}, \eqref{R3}, \eqref{v^2} and integrating with respect to $t$, we have
\[
\begin{split}
& \int_\mathbb{R} a^X \eta(W|\bar{W}^X) (t,x) \, dx + \int_0^t \left( \delta_S \lvert \dot{X} \rvert^2 + G_1 + G^S + D \right) \, d\tau \\
& \leq \int_\mathbb{R} a(0,x) \eta(W|\bar{W})(0,x) \, dx + C \left( \sqrt{\delta_1} + \varepsilon_1 \right) \int_0^t \int_\mathbb{R} \left( \tilde{v}_x^2 + \bar{v}^X_x \tilde{\phi}^2 + \tilde{\phi}_x^2 + \tilde{\phi}_{xx}^2 + \tilde{\phi}_{xt}^2 + \tilde{\phi}_{xxt}^2 \right) \, dx d\tau
\end{split}
\]
for sufficiently small $\delta_1$ and $\varepsilon_1$. Thanks to the result of Lemma~\ref{entsim}, we have
\begin{equation}
\begin{split}
& \lVert (\tilde{v},\tilde{u}) (t,\cdot) \rVert_{L^2}^2 + \lVert \tilde{\phi}_x (t,\cdot) \rVert_{H^1}^2 + \int_0^t \left( \delta_S \lvert \dot{X} \rvert^2 + G_1 + G^S + D \right) \, d\tau \\
& \leq C \left( \lVert (\tilde{v}_0,\tilde{u}_0) \rVert_{L^2}^2 + \lVert \tilde{\phi}_x (0,\cdot) \rVert_{H^1}^2 \right) \\
& \quad + C \left( \sqrt{\delta_1} + \varepsilon_1 \right) \int_0^t \int_\mathbb{R} \left( \tilde{v}_x^2 + \bar{v}^X_x \tilde{\phi}^2 + \tilde{\phi}_x^2 + \tilde{\phi}_{xx}^2 + \tilde{\phi}_{xt}^2 + \tilde{\phi}_{xxt}^2 \right)  \, dx d\tau .
\end{split}
\end{equation}
Finally, by using \eqref{phiest}, we obtain the desired estimate \eqref{Ree}.

\section{Elliptic estimates}
This section is devoted to providing estimates for $\tilde{\phi}$ and its derivatives that appear on the right-hand side of \eqref{Ree}. For this purpose, we first rewrite \eqref{1c'} as follows:
\begin{equation} \label{temP}
\begin{split}
- \left[ \frac{\tilde{\phi}_x}{v} + \left( \frac{1}{v} - \frac{1}{\bar{v}^X} \right) \bar{\phi}^X_x \right]_x = - e^{\bar{\phi}^X}\tilde{v} - \bar{v}^X e^{\bar{\phi}^X} \tilde{\phi} + e^{\bar{\phi}^X} \mathcal{N}_p,
\end{split}
\end{equation}
where the nonlinear term $\mathcal{N}_p$ is defined by
\begin{equation} \label{Np}
\mathcal{N}_p := \bar{v}^X \left( 1+\tilde{\phi} - e^{\tilde{\phi}} \right) + \tilde{v} \left( 1 - e^{\tilde{\phi}} \right).
\end{equation}
With \eqref{temP}-\eqref{Np}, we obtain first the estimates for the perturbation around $\bar{\phi}^X$ and its $x$-derivatives. 

\begin{lemma} \label{Lemma:phi}
Under the assumptions in Proposition~\ref{Apriori}, there exists a positive constant $C>0$ such that
\begin{equation} \label{derivphi}
\int_\mathbb{R} \left( \bar{v}^X_x \tilde{\phi}^2 + \tilde{\phi}_x^2 + \tilde{\phi}_{xx}^2 \right)  \, dx  \leq C \left( G_1 + G^S + \int_\mathbb{R} \tilde{v}_x^2 \, dx \right)
\end{equation}
for all $t \in [0,T]$, where $G_1$ and $G^S$ are as defined in Lemma~\ref{RE}.
\end{lemma}

\begin{proof}
Multiplying \eqref{temP} by $\bar{v}^X_x \tilde{\phi}$ and integrating the resultant equation with respect to $x$, we have after rearrangement
\begin{equation} \label{Pest0}
\begin{split}
& \int_\mathbb{R} \bar{v}^X e^{\bar{\phi}^X} \bar{v}^X_x \tilde{\phi}^2 \, dx + \int_\mathbb{R} \frac{ \bar{v}^X_x \tilde{\phi}_x^2}{v} \, dx \\
& \quad = - \int_\mathbb{R} \bar{v}^X_x \bar{\phi}^X_x \tilde{\phi}_x  \left( \frac{1}{v} - \frac{1}{\bar{v}^X} \right) \, dx  - \int_\mathbb{R} \frac{\bar{v}^X_{xx}\tilde{\phi}\tilde{\phi}_x}{v} \, dx \\
& \qquad - \int_\mathbb{R} \bar{v}^X_{xx} \bar{\phi}^X_x \tilde{\phi}  \left( \frac{1}{v} - \frac{1}{\bar{v}^X} \right) \, dx - \int_\mathbb{R} e^{\bar{\phi}^X}\bar{v}^X_x \tilde{\phi} \tilde{v}  \, dx + \int_\mathbb{R} e^{\bar{\phi}^X}\bar{v}^X_x\tilde{\phi} \mathcal{N}_p \, dx \\
& \quad =: P_{11} + P_{12} + P_{13} + P_{14} + P_{15}.
\end{split}
\end{equation}
Here, by using the bounds $\lvert \bar{v}^X_{xx} \rvert,\lvert \bar{\phi}^X_x \rvert \leq C \bar{v}^X_x \leq C \delta_S^2$ and Young's inequality, the terms $P_{1j}$ are estimated as
\begin{equation} \label{Pest1}
\begin{split}
\lvert P_{11} \rvert & \leq C \int_\mathbb{R} \lvert \bar{v}^X_x \rvert^2 \lvert \tilde{\phi}_x \rvert \lvert \tilde{v} \rvert \, dx \leq C \delta_S^2 \int_\mathbb{R} \left(  \bar{v}^X_x \tilde{v}^2 + \tilde{\phi}_x^2 \right) \, dx, \\
\lvert P_{12} \rvert & \leq C \int_\mathbb{R} \lvert \bar{v}^X_x \rvert \lvert \tilde{\phi} \rvert \lvert \tilde{\phi}_x \rvert \, dx \leq C \left(\int_\mathbb{R} \lvert \bar{v}^X_x \rvert^{1/2} \tilde{\phi}_x^2 \, dx + \int_\mathbb{R} \lvert \bar{v}^X \rvert^{3/2} \tilde{\phi}^2 \, dx \right) \\
&  \leq C \delta_S \int_\mathbb{R} \left( \bar{v}^X_x \tilde{\phi}^2 + \tilde{\phi}_x^2 \right) \, dx, \\
\lvert P_{13} \rvert & \leq C \int_\mathbb{R} \lvert \bar{v}^X_x \rvert^2 \lvert \tilde{\phi} \rvert \lvert \tilde{v} \rvert \, dx \leq C \delta_S^2 \int_\mathbb{R} \bar{v}^X_x \left( \tilde{v}^2 + \tilde{\phi}^2 \right) \, dx, \\
\lvert P_{14} \rvert & \leq C \int_\mathbb{R} \bar{v}^X_x \lvert \tilde{\phi} \rvert \lvert \tilde{v} \rvert \, dx \leq \eta \int_\mathbb{R} \bar{v}^X_x \tilde{\phi}^2 \, dx + \frac{C}{\eta} \int_\mathbb{R} \bar{v}^X_x \tilde{v}^2 \, dx,
\end{split}
\end{equation}
and
\begin{equation} \label{Pest2}
\begin{split}
\lvert P_{15} \rvert & \leq C \int_\mathbb{R} \bar{v}^X_x \lvert \tilde{\phi} \rvert \left( \mathcal{O}(\lvert \tilde{\phi} \rvert^2) + \lvert \tilde{v} \rvert \mathcal{O}(\lvert \tilde{\phi} \rvert) \right) \, dx \leq C \lVert \tilde{\phi} \rVert_{L^\infty} \int_\mathbb{R} \bar{v}^X_x \left( \tilde{v}^2 + \tilde{\phi}^2 \right) \, dx,
\end{split}
\end{equation}
where $0 < \eta <1$ is an arbitrary constant and in the estimate of $P_{15}$, we used the definition \eqref{Np}, together with the Taylor expansion of $e^{\tilde{\phi}}$ around $\tilde{\phi}=0$. Combining \eqref{Pest0} with \eqref{Pest1} and \eqref{Pest2}, and taking $\eta$ sufficiently small, we obtain
\begin{equation} \label{barvp}
\begin{split}
\int_\mathbb{R} \bar{v}^X_x  \left( \tilde{\phi}^2+ \tilde{\phi}_x^2 \right) \, dx \leq C \left(  \int_\mathbb{R} \bar{v}^X_x \tilde{v}^2 \, dx + \delta_S \int_\mathbb{R} \tilde{\phi}_x^2 \, dx \right)
\end{split}
\end{equation}
for sufficiently small $\delta_S < \delta_1$ and $\lVert \tilde{\phi} \rVert_{L^\infty} \leq C \varepsilon_1$.

Next, we differentiate \eqref{temP} with respect to $x$. Then, we have
\[
\begin{split}
& - \left[ \frac{\tilde{\phi}_{xx}}{v} - \frac{(\tilde{v}_x + \bar{v}^X_x)\tilde{\phi}_x}{v^2} +  \left( \frac{1}{v} - \frac{1}{\bar{v}^X} \right)_x \bar{\phi}^X_x +  \left( \frac{1}{v} - \frac{1}{\bar{v}^X} \right) \bar{\phi}^X_{xx} \right]_x \\
& \quad = - e^{\bar{\phi}^X}\tilde{v}_x - \bar{v}^X e^{\bar{\phi}^X} \tilde{\phi}_x - \big( e^{\bar{\phi}^X} \big)_x \tilde{v} - \big( \bar{v}^X e^{\bar{\phi}^X} \big)_x \tilde{\phi} + \big( e^{\bar{\phi}^X} \mathcal{N}_p \big)_x .
\end{split}
\]
Multiplying it by $\tilde{\phi}_x$ and integrating the resultant equation with respect to $x$, we have after rearrangement
\[
\begin{split}
\int_\mathbb{R} \bigg( \bar{v}^X e^{\bar{\phi}^X} \tilde{\phi}_x^2 + \frac{ \tilde{\phi}_{xx}^2}{v} \bigg) \, dx & = \int_\mathbb{R} \tilde{\phi}_{xx} \left( \frac{(\tilde{v}_x + \bar{v}^X_x)\tilde{\phi}_x}{v^2} + \left( \frac{\tilde{v}}{v\bar{v}^X} \right)_x \bar{\phi}^X_x + \frac{\bar{\phi}^X_{xx}\tilde{v}}{v\bar{v}^X} \right) \, dx \\
& \quad  + \int_\mathbb{R} \tilde{\phi}_x \left( - e^{\bar{\phi}^X}\tilde{v}_x  - \big( e^{\bar{\phi}^X} \big)_x \tilde{v} - \big( \bar{v}^X e^{\bar{\phi}^X} \big)_x \tilde{\phi} \right) \, dx \\
& \quad + \int_\mathbb{R} \big( e^{\bar{\phi}^X} \big)_x \tilde{\phi}_x \mathcal{N}_p \, dx + \int_\mathbb{R} e^{\bar{\phi}} \tilde{\phi}_x \left( \mathcal{N}_{p} \right)_x \, dx \\
& =: P_{21} + P_{22} + P_{23} + P_{24}.
\end{split}
\]
Similarly to \eqref{Pest1}, we obtain the following estimates of $P_{2j}$ by using the bounds on the shock and applying Young's inequality:
\[
\begin{split}
\lvert P_{21} \rvert & \leq C \int_\mathbb{R} \lvert \bar{v}^X_x \rvert \lvert \tilde{\phi}_{xx} \rvert \lvert \tilde{v} \rvert \, dx + C \int_\mathbb{R} \lvert \bar{v}^X_x \rvert \lvert \tilde{\phi}_{xx} \rvert \left( \lvert \tilde{v}_x \rvert + \lvert \tilde{\phi}_x \rvert \right) \, dx \\
& \quad + C \int_\mathbb{R} \lvert \tilde{v}_x \rvert \lvert \tilde{\phi}_x \rvert \lvert \tilde{\phi}_{xx} \rvert + \lvert \tilde{v} \rvert \lvert \tilde{v}_x \rvert \lvert \tilde{\phi}_{xx} \rvert \, dx \\
& \leq C \left( \delta_S + \lVert \tilde{v} \rVert_{W^{1,\infty}} \right) \int_\mathbb{R} \left( \bar{v}^X_x \tilde{v}^2 + \tilde{v}_x^2 + \tilde{\phi}_x^2 + \tilde{\phi}_{xx}^2 \right) \, dx, 
\end{split}
\]
\[
\begin{split}
\lvert P_{22} \rvert & \leq C \int_\mathbb{R} \lvert \tilde{\phi}_x \rvert \lvert \tilde{v}_x \rvert \, dx + C \int_\mathbb{R} \lvert \bar{v}^X_x \rvert \lvert \tilde{\phi}_x \rvert \left( \lvert \tilde{v} \rvert + \lvert \tilde{\phi} \rvert \right) \, dx \\
& \leq \eta \int_\mathbb{R} \tilde{\phi}_x^2 \, dx + \frac{C}{\eta} \int_\mathbb{R} \tilde{v}_x^2 \, dx  + C \delta_S \int_\mathbb{R} \left( \bar{v}^X_x \tilde{v}^2 + \bar{v}^X_x \tilde{\phi}^2 + \tilde{\phi}_x^2 \right) \, dx,
\end{split}
\]
\[
\begin{split}
\lvert P_{23} \rvert & \leq C \int_\mathbb{R} \lvert \bar{v}^X_x \rvert \lvert \tilde{\phi}_x \rvert \lvert \mathcal{N}_p \rvert \, dx  \leq C \int_\mathbb{R} \lvert \bar{v}^X_x \rvert \lvert \tilde{\phi}_x \rvert \left( \mathcal{O}(\lvert \tilde{\phi} \rvert^2) + \lvert \tilde{v} \rvert \mathcal{O}( \lvert \tilde{\phi} \rvert )  \right) \, dx \\
& \leq C \lVert \tilde{\phi} \rVert_{L^\infty} \int_\mathbb{R} \left( \bar{v}^X_x \tilde{v}^2 + \bar{v}^X_x \tilde{\phi}^2 + \tilde{\phi}_x^2 \right) \, dx,
\end{split}
\]
and
\[
\begin{split}
\lvert P_{24} \rvert & \leq \int_\mathbb{R} \lvert \tilde{\phi}_x \rvert \left( \lvert \bar{v}^X_x \rvert \mathcal{O}( \lvert \tilde{\phi} \rvert^2) + \lvert \tilde{\phi}_x \rvert \mathcal{O}(\lvert \tilde{\phi} \rvert) + \lvert \tilde{v}_x \rvert \mathcal{O}(\lvert \tilde{\phi} \rvert) + \lvert \tilde{v} \rvert \lvert \tilde{\phi}_x \rvert (1 + \mathcal{O}(\lvert \tilde{\phi} \rvert ) ) \right) \, dx \\
& \leq C \lVert ( \tilde{v}, \tilde{\phi}) \rVert_{L^\infty} \int_\mathbb{R} \left( \bar{v}^X_x \tilde{\phi}^2 + \tilde{\phi}_x^2 + \tilde{v}_x^2 \right) \, dx.
\end{split}
\]

Collecting all the estimates of $P_{2j}$ and taking $\eta$ small enough in the estimate of $ P_{22} $, we have 
\[
\begin{split}
&\int_\mathbb{R} \left( \tilde{\phi}_x^2 + \tilde{\phi}_{xx}^2 \right) \, dx \leq C \left( \delta_S^2 + \lVert \tilde{\phi} \rVert_{L^\infty} \right) \int_\mathbb{R} \bar{v}^X_x \tilde{\phi}^2 \, dx + C \int_\mathbb{R} \bar{v}^X_x \tilde{v}^2 \, dx + C \int_\mathbb{R} \tilde{v}_x^2 \, dx
\end{split}
\]
for sufficiently small $\delta_S < \delta_1$ and $\lVert (\tilde{v}, \tilde{\phi}, \tilde{\phi}_x ) \rVert_{L^\infty} \leq C \varepsilon_1$. Combining this with \eqref{barvp} and using \eqref{v^2}, we obtain the desired bound \eqref{derivphi}.
\end{proof}

The estimates for $\tilde{\phi}_t$ and its $x$-derivatives are presented in the following two lemmas.
\begin{lemma} \label{Lemma:phit}
Under the assumptions in Proposition \ref{Apriori}, there exists a constant $C>0$ such that
\begin{equation} \label{phit}
\begin{split}
\int_\mathbb{R} \left( \tilde{\phi}_t^2 + \tilde{\phi}_{xt}^2 \right) \, dx & \leq C \left( \delta_1 + \varepsilon_1 \right) \left( \delta_S \lvert \dot{X} \rvert^2 + G_1 + G^S + \int_\mathbb{R} \tilde{v}_x^2  \, dx \right) + C D
\end{split}
\end{equation}
for all $t \in [0,T]$, where $G_1$, $G^S$, and $D$ are as defined in Lemma~\ref{RE}.
\end{lemma}

\begin{proof}
Differentiating \eqref{temP} with respect to $t$ and multiplying the resultant equation by $\tilde{\phi}_t$, we obtain after integrating with respect to $x$
\[
\begin{split}
& - \int_\mathbb{R} \tilde{\phi}_t \left[ \frac{\tilde{\phi}_x}{v} + \left( \frac{1}{v} - \frac{1}{\bar{v}^X} \right)\bar{\phi}^X_x \right]_{xt} \, dx  = - \int_\mathbb{R} \tilde{\phi}_t \left( e^{\bar{\phi}^X} \tilde{v} + \bar{v}^X e^{\bar{\phi}^X} \tilde{\phi} \right)_t \, dx + \int_\mathbb{R} \tilde{\phi}_t \left( e^{\bar{\phi}^X} \mathcal{N}_p \right)_t \, dx.
\end{split}
\]
Then, by integration by parts and rearrangement, we obtain
\[
\begin{split}
\int_\mathbb{R} \bigg( \bar{v}^X e^{\bar{\phi}^X}\tilde{\phi}_t^2 + \frac{ \tilde{\phi}_{xt}^2}{v} \bigg) \, dx & = \int_\mathbb{R} \tilde{\phi}_{xt}  \left( \frac{v_t \tilde{\phi}_x}{v^2} + \left( \frac{\bar{\phi}^X_x}{\bar{v}^X} \right)_t \frac{\tilde{v}}{v} + \frac{\bar{\phi}^X_x \tilde{v}_t}{v\bar{v}^X} - \frac{\bar{\phi}^X_x v_t \tilde{v}}{v^2 \bar{v}^X} \right) \, dx \\
& \quad - \int_\mathbb{R} \tilde{\phi}_t \left( (e^{\bar{\phi}^X})_t \tilde{v} + e^{\bar{\phi}^X}\tilde{v}_t + ( \bar{v}^X e^{\bar{\phi}^X} )_t \tilde{\phi}  \right) \, dx \\
& \quad + \int_\mathbb{R} \tilde{\phi}_t e^{\bar{\phi}^X} \left[ \left( \bar{v}^X ( 1 + \tilde{\phi} - e^{\tilde{\phi}} ) \right)_t + \tilde{v}_t ( 1 - e^{\tilde{\phi}} ) + \tilde{v} (1-e^{\tilde{\phi}})_t \right] \, dx \\
& \quad + \int_\mathbb{R} \tilde{\phi}_t (e^{\bar{\phi}^X})_t \mathcal{N}_p \, dx.
\end{split}
\]
Here we use \eqref{NSP1} and \eqref{1a'} to obtain
\[
\int_\mathbb{R} \bigg( \bar{v}^X e^{\bar{\phi}^X}\tilde{\phi}_t^2 + \frac{\tilde{\phi}_{xt}^2}{v} \bigg) \, dx = \sum_{j=1}^5 \RNum{1}_j,
\]
where
\[
\begin{split}
\RNum{1}_1 & := - \int_\mathbb{R} e^{\bar{\phi}^X} \tilde{\phi}_t \tilde{u}_x \, dx, \quad \RNum{1}_2 := \int_\mathbb{R} \tilde{\phi}_{xt}  \left( \frac{\tilde{u}_x \tilde{\phi}_x}{v^2} - \frac{\bar{\phi}^X_x \tilde{u}_x \tilde{v}}{v^2\bar{v}^X}  \right) \, dx, \\
\RNum{1}_3 & := \int_\mathbb{R} \tilde{\phi}_{xt} \left( \frac{\bar{u}^X_x \tilde{\phi}_x}{v^2} + \frac{\bar{\phi}^X_x \tilde{u}_x}{v\bar{v}^X} -  \left( \sigma + \dot{X}(t) \right) \left( \frac{\bar{\phi}^X_x}{\bar{v}^X} \right)_x \frac{\tilde{v}}{v} - \frac{\bar{\phi}^X_x \bar{u}^X_x \tilde{v}}{v^2\bar{v}^X} \right) \, dx \\
& \quad + \left( \sigma + \dot{X}(t) \right) \int_\mathbb{R} \tilde{\phi}_t \left( (e^{\bar{\phi}^X})_x \tilde{v} + ( \bar{v}^X e^{\bar{\phi}^X} )_x \tilde{\phi} \right) \, dx,
\end{split}
\]
\[
\begin{split}
\RNum{1}_4 & = \dot{X}(t) \left(  \int_\mathbb{R} \tilde{\phi}_{xt} \frac{\bar{\phi}^X_x \bar{v}^X_x}{v\bar{v}^X} \, dx - \int_\mathbb{R} e^{\bar{\phi}^X}\bar{v}^X_x \tilde{\phi}_t \, dx \right),
\end{split}
\]
and
\[
\begin{split}
\RNum{1}_5 & = \int_\mathbb{R} \tilde{\phi}_t e^{\bar{\phi}^X} \left( -\sigma \bar{v}^X_x (1+\tilde{\phi} - e^{\tilde{\phi}})  + \bar{v}^X (1 + \tilde{\phi} - e^{\tilde{\phi}})_t +  \tilde{u}_x (1-e^{\tilde{\phi}}) + \tilde{v} (1-e^{\tilde{\phi}})_t \right) \, dx \\
& \quad - \dot{X}(t) \int_\mathbb{R} e^{\bar{\phi}^X} \bar{v}^X_x \tilde{\phi}_t \tilde{\phi} \, dx  - \left( \sigma + \dot{X}(t) \right) \int_\mathbb{R} \tilde{\phi}_t ( e^{\bar{\phi}^X} )_x \left( \bar{v}^X (1+\tilde{\phi}-e^{\tilde{\phi}}) + \tilde{v} (1-e^{\tilde{\phi}}) \right) \, dx.
\end{split}
\]

We estimate the terms $\RNum{1}_j$ for $j=1,\dots,5$. First, applying Young's inequality, we obtain the bounds on $\RNum{1}_1$
\[
\lvert \RNum{1}_1 \rvert \leq C \int_\mathbb{R} \lvert \tilde{\phi}_t \rvert \lvert \tilde{u}_x \rvert \, dx \leq \eta \int_\mathbb{R} \tilde{\phi}_t^2 \, dx + \frac{C}{\eta} \int_\mathbb{R} \tilde{u}_x^2 \, dx
\]
for any $ 0 < \eta < 1$. Similarly, the term $\RNum{1}_2$ is estimated by using the boundedness of $\lVert ( \tilde{v},\tilde{\phi}_x ) \rVert_{L^\infty}$ as
\[
\begin{split}
\lvert \RNum{1}_2 \rvert & \leq C \int_\mathbb{R} \left( \lvert \tilde{\phi}_x \rvert + \lvert \tilde{v} \rvert \right) \lvert \tilde{\phi}_{xt} \rvert \lvert \tilde{u}_x \rvert  \, dx \leq C \left( \lVert \tilde{v} \rVert_{L^\infty} + \lVert \tilde{\phi}_x \rVert_{L^\infty} \right) \int_\mathbb{R} \left( \tilde{\phi}_{xt}^2 + \tilde{u}_x^2 \right) \, dx.
\end{split}
\]
Next we use $\lvert \bar{u}^X_x \rvert \sim \lvert \bar{\phi}^X_x \rvert \sim \lvert \bar{v}^X_x \rvert \leq C \delta_S^2$ and $\lvert \dot{X} \rvert \leq C \lVert \tilde{u} \rVert_{L^\infty}$ to obtain
\[
\begin{split}
\lvert \RNum{1}_3 \rvert & \leq  C \int_\mathbb{R} \lvert \bar{v}^X_x \rvert \lvert \tilde{\phi}_{xt} \rvert \left( \lvert \tilde{\phi}_x \rvert + \lvert \tilde{u}_x \rvert + \lvert \bar{v}^X_x \rvert \lvert \tilde{v} \rvert \right) \, dx \\
& \quad + C \left( 1 + \lvert \dot{X} \rvert \right) \int_\mathbb{R} \lvert \bar{v}^X_x \rvert \left( \lvert \tilde{\phi}_t \rvert \lvert \tilde{v} \rvert + \lvert \tilde{\phi}_t \rvert \lvert \tilde{\phi} \rvert + \lvert \tilde{\phi}_{xt} \rvert \lvert \tilde{v} \rvert \right) \, dx \\
& \leq C \left( \int_\mathbb{R} \lvert \bar{v}^X_x \rvert \left( \tilde{\phi}_{xt}^2 + \tilde{\phi}_x^2 + \tilde{u}_x^2 \right) \, dx + \int_\mathbb{R} \lvert \bar{v}^X_x \rvert^2 \tilde{v}^2 \, dx \right)  \\
& \quad + C \left( 1 + \lVert \tilde{u} \rVert_{L^\infty} \right) \left( \int_\mathbb{R} \lvert \bar{v}^X_x \rvert^{1/2} \left( \tilde{\phi}_t^2 + \tilde{\phi}_{xt}^2 \right) \, dx + \int_\mathbb{R} \lvert \bar{v}^X_x \rvert^{3/2} \left( \tilde{v}^2 + \tilde{\phi}^2 \right) \, dx \right) \\
& \leq C \left( \delta_S + \lVert \tilde{u} \rVert_{L^\infty} \right) \int_\mathbb{R} \left( \bar{v}^X_x \tilde{v}^2 + \bar{v}^X_x \tilde{\phi}^2 + \tilde{\phi}_t^2 + \tilde{\phi}_{xt}^2 + \tilde{\phi}_x^2 + \tilde{u}_x^2 \right) \, dx.
\end{split}
\]
For the term $\RNum{1}_4$, we apply the H\"older inequality and Young's inequality to obtain
\[
\begin{split}
\lvert \RNum{1}_{4} \rvert & \leq C \lvert \dot{X} \rvert  \int_\mathbb{R} \bar{v}^X_x \left( \lvert \tilde{\phi}_{xt} \rvert + \lvert \tilde{\phi}_t \rvert \right) \, dx \\
& \leq C \lvert \dot{X} \rvert \sqrt{\int_\mathbb{R} \bar{v}^X_x \, dx} \left( \sqrt{\int_\mathbb{R} \bar{v}^X_x \tilde{\phi}_{xt}^2 \, dx} + \sqrt{\int_\mathbb{R} \bar{v}^X_x \tilde{\phi}_t^2 \, dx} \right) \\
& \leq C \delta_S^2 \lvert \dot{X} \rvert^2 + \frac{C}{\delta_S^2} \left( \int_\mathbb{R} \bar{v}^X_x \, dx \right) \left( \int_\mathbb{R} \bar{v}^X_x \tilde{\phi}_{xt}^2 \, dx + \int_\mathbb{R} \bar{v}^X_x \tilde{\phi}_t^2 \, dx \right) \\
& \leq C \delta_S \left( \delta_S \lvert \dot{X} \rvert^2 + \int_\mathbb{R} \left( \tilde{\phi}_{xt}^2 + \tilde{\phi}_t^2 \right) \, dx \right).
\end{split}
\]
Similar to $\RNum{1}_3$, the term $\RNum{1}_5$ is estimated as
\[
\begin{split}
\lvert \RNum{1}_{5} \rvert & \leq C \int_\mathbb{R} \lvert \bar{v}^X_x \rvert \lvert \tilde{\phi}_t \rvert \lvert \tilde{\phi} \rvert \mathcal{O}(\lvert \tilde{\phi} \rvert) \, dx + C \int_\mathbb{R} \lvert \tilde{\phi}_t \rvert \left( \lvert \tilde{\phi}_t \rvert \mathcal{O}(\lvert \tilde{\phi} \rvert) + \lvert \tilde{u}_x \rvert \mathcal{O}(\lvert \tilde{\phi} \rvert ) + \lvert \tilde{v} \rvert \lvert \tilde{\phi}_t \rvert + \lvert \tilde{v} \rvert \lvert \tilde{\phi}_t \rvert \mathcal{O}(\lvert \tilde{\phi} \rvert \right) \, dx \\
& \quad + C \lvert \dot{X} \rvert \int_\mathbb{R} \lvert \bar{v}^X_x \rvert \lvert \tilde{\phi}_t \rvert \lvert \tilde{\phi} \rvert \, dx + C \left( 1 + \lvert \dot{X} \rvert \right) \int_\mathbb{R} \lvert \bar{v}^X_x \rvert \lvert \tilde{\phi}_t \rvert \left( \lvert \tilde{\phi} \rvert + \lvert \tilde{v} \rvert \right)\mathcal{O}(\lvert \tilde{\phi} \rvert) \, dx \\
& \leq C \left( \lVert \tilde{v} \rVert_{L^\infty} + \lVert \tilde{u} \rVert_{L^\infty} + \lVert \tilde{\phi} \rVert_{L^\infty} \right) \int_\mathbb{R} \left( \bar{v}^X_x \tilde{v}^2 + \bar{v}^X_x \tilde{\phi}^2 + \tilde{\phi}_t^2 + \tilde{u}_x^2 \right) \, dx.
\end{split}
\]

Collecting all the above estimates of $\RNum{1}_j$ and taking $\eta$ sufficiently small, we obtain
\[
\begin{split}
\int_\mathbb{R} \left( \tilde{\phi}_t^2 + \tilde{\phi}_{xt}^2 \right) \, dx & \leq C \left( \delta_S + \lVert ( \tilde{v},\tilde{u},\tilde{\phi},\tilde{\phi}_x) \rVert_{L^\infty} \right) \\
& \qquad \times \left( \int_\mathbb{R} \left( \bar{v}^X_x \tilde{v}^2 + \bar{v}^X_x \tilde{\phi}^2 + \tilde{\phi}_x^2 \right) \, dx + \delta_S \lvert \dot{X} \rvert^2 \right) + C \int_\mathbb{R} \tilde{u}_x^2 \, dx
\end{split}
\]
for sufficiently small $\delta_S < \delta_1$ and $\lVert ( \tilde{v},\tilde{u},\tilde{\phi},\tilde{\phi}_x) \rVert_{L^\infty} \leq C \varepsilon_1$. Finally, we have the desired inequality by using the results of Lemmas~\ref{Lv^2} and \ref{Lemma:phi}.
\end{proof}

\begin{lemma} \label{Lemma:phixt}
Under the assumptions in Proposition \ref{Apriori}, there exists a constant $C>0$ such that
\begin{equation} \label{phixt}
\begin{split}
\int_\mathbb{R} \left( \tilde{\phi}_{xt}^2 + \tilde{\phi}_{xxt}^2 \right) \, dx & \leq C \left( \delta_1 + \varepsilon_1 \right)  \left( \delta_S \lvert \dot{X} \rvert^2 + G_1 + G^S + \int_\mathbb{R} \left( \tilde{v}_x^2 + \tilde{u}_{xx}^2 \right) \, dx \right) + C D
\end{split}
\end{equation}
for all $t \in [0,T]$, where $G_1$, $G^S$, and $D$ are as defined in Lemma~\ref{RE}.
\end{lemma}

\begin{proof}
Differentiating \eqref{temP} once with respect to $x$ and $t$ and multiplying the resultant equation by $\tilde{\phi}_{xt}$, we obtain after integrating with respect to $x$
\[
\begin{split}
& - \int_\mathbb{R} \tilde{\phi}_{xt} \left[ \frac{\tilde{\phi}_x}{v} + \left( \frac{1}{v} - \frac{1}{\bar{v}^X} \right)\bar{\phi}^X_x \right]_{xxt} \, dx \\
& \quad = - \int_\mathbb{R} \tilde{\phi}_{xt} \left( e^{\bar{\phi}^X} \tilde{v} + \bar{v}^X e^{\bar{\phi}^X} \tilde{\phi} \right)_{xt} \, dx + \int_\mathbb{R} \tilde{\phi}_{xt} \left( e^{\bar{\phi}^X} \mathcal{N}_p \right)_{xt} \, dx.
\end{split}
\]
By integration by parts, we have after rearrangement
\[
\begin{split}
& \int_\mathbb{R} \bigg( \bar{v}^X e^{\bar{\phi}^X}\tilde{\phi}_{xt}^2 + \frac{ \tilde{\phi}_{xxt}^2}{v} \bigg) \, dx \\
& \quad = \int_\mathbb{R} \tilde{\phi}_{xxt} \frac{v_x \tilde{\phi}_{xt}}{v^2} \, dx + \int_\mathbb{R} \tilde{\phi}_{xxt} \left[ \frac{v_t\tilde{\phi}_x}{v^2} + \frac{\bar{\phi}^X_x \tilde{v}_t}{v\bar{v}^X} - \frac{\bar{\phi}^X_x v_t \tilde{v}}{v^2 \bar{v}^X} + \left( \frac{\bar{\phi}^X_x}{\bar{v}^X} \right)_t \frac{\tilde{v}}{v} \right]_x \, dx \\
& \qquad + \int_\mathbb{R} \tilde{\phi}_{xxt} \left( (e^{\bar{\phi}^X})_t\tilde{v} + e^{\bar{\phi}^X}\tilde{v}_t + (\bar{v}^X e^{\bar{\phi}^X})_t \tilde{\phi} \right) \, dx - \int_\mathbb{R} ( \bar{v}^X e^{\bar{\phi}^X} )_x \tilde{\phi}_{xt}\tilde{\phi}_t \, dx \\
& \qquad - \int_\mathbb{R} \tilde{\phi}_{xxt}  \left( (e^{\bar{\phi}^X})_t \tilde{v} (1-e^{\tilde{\phi}}) + e^{\bar{\phi}^X} \tilde{v}_t (1-e^{\tilde{\phi}}) + e^{\bar{\phi}^X} \tilde{v} (1-e^{\tilde{\phi}})_t \right) \, dx \\
& \qquad - \int_\mathbb{R} \tilde{\phi}_{xxt} \left( \bar{v}^X e^{\bar{\phi}^X}(1 + \tilde{\phi} - e^{\tilde{\phi}}) \right)_t \, dx.
\end{split}
\]
We use \eqref{NSP1} and \eqref{1a'} in the right-hand side to obtain
\[
\int_\mathbb{R} \bigg( \bar{v}^X e^{\bar{\phi}^X}\tilde{\phi}_{xt}^2 + \frac{\tilde{\phi}_{xxt}^2}{v} \bigg) \, dx  = \sum_{j=1}^6 \RNum{2}_j,
\]
where
\[
\begin{split}
\RNum{2}_1 & := \int_\mathbb{R}  e^{\bar{\phi}^X} \tilde{\phi}_{xxt} \tilde{u}_x \, dx, \\
\RNum{2}_2 & := \int_\mathbb{R} \tilde{\phi}_{xxt} \left( \frac{\tilde{v}_x\tilde{\phi}_{xt}}{v^2} + \frac{\tilde{u}_{xx}\tilde{\phi}_x}{v^2} + \frac{\tilde{u}_x \tilde{\phi}_{xx}}{v^2} - \frac{2 ( \tilde{v}_x + \bar{v}^X_x) \tilde{u}_x\tilde{\phi}_x}{v^3} - \frac{2 \bar{u}^X_x \tilde{v}_x \tilde{\phi}_x}{v^3} \right) \, dx \\
& \quad + \int_\mathbb{R} \tilde{\phi}_{xxt} \left( - \frac{2 \bar{\phi}^X_x \tilde{v}_x \tilde{u}_x}{v^2\bar{v}^X} - \dot{X}(t) \frac{\bar{\phi}^X_x \bar{v}^X_x \tilde{v}_x}{v^2\bar{v}^X} - \left( \frac{\bar{\phi}^X_x}{\bar{v}^X} \right)_x \frac{\tilde{u}_x \tilde{v}}{v^2}  - \frac{\bar{\phi}^X_x \tilde{u}_{xx}\tilde{v}}{v^2\bar{v}^X} \right) \, dx \\
& \quad + \int_\mathbb{R} \tilde{\phi}_{xxt} \left( + \frac{2\bar{\phi}^X_x ( \tilde{v}_x + \bar{v}^X_x )\tilde{u}_x\tilde{v}}{v^3 \bar{v}^X} + \frac{2 \bar{\phi}^X_x \bar{u}^X_x \tilde{v}_x \tilde{v}}{v^3 \bar{v}^X} + (\sigma + \dot{X}(t)) \left( \frac{\bar{\phi}^X_x}{\bar{v}^X} \right)_x \frac{\tilde{v}_x^2}{v^2} \right) \, dx, \\
\RNum{2}_3 & := \int_\mathbb{R} \tilde{\phi}_{xxt} \left( \frac{\bar{v}^X_x \tilde{\phi}_{xt}}{v^2} + \frac{\bar{u}^X_{xx} \tilde{\phi}_x}{v^2} + \frac{\bar{u}^X_x \tilde{\phi}_{xx}}{v^2} - \frac{2\bar{u}^X_x\bar{v}^X_x \tilde{\phi}_x}{v^3} + \left( \frac{\bar{\phi}^X_x}{\bar{v}^X} \right)_x \frac{\tilde{u}_x}{v}  +  \frac{\bar{\phi}^X_x\tilde{u}_{xx}}{v\bar{v}^X}  \right) \, dx \\
& \quad +  \int_\mathbb{R} \tilde{\phi}_{xxt} \left( - \frac{\bar{\phi}^X_x\bar{v}^X_x \tilde{u}_x}{v^2\bar{v}^X} - \left( \frac{\bar{\phi}^X_x\bar{u}^X_x}{\bar{v}^X} \right)_x \frac{\tilde{v}}{v^2} - \frac{\bar{\phi}^X_x \bar{u}^X_x \tilde{v}_x}{v^2 \bar{v}^X} + \frac{2\bar{\phi}^X_x\bar{u}^X_x\bar{v}^X_x\tilde{v}}{v^3\bar{v}^X} \right) \, dx \\
& \quad - (\sigma + \dot{X}(t)) \int_\mathbb{R}\tilde{\phi}_{xxt}  \left[ \left( \frac{\bar{\phi}^X_x}{\bar{v}^X} \right)_{xx} \frac{\tilde{v}}{v} + \left( \frac{\bar{\phi}^X_x}{\bar{v}^X} \right)_x \left( \frac{\tilde{v}_x}{v}- \frac{\bar{v}^X_x \tilde{v}_x}{v^2} \right) \right] \, dx \\
& \quad - (\sigma + \dot{X}(t)) \int_\mathbb{R} \tilde{\phi}_{xxt} \left( (e^{\bar{\phi}^X})_x\tilde{v} + (\bar{v}^Xe^{\bar{\phi}^X})_x \tilde{\phi} \right) \, dx  - \int_\mathbb{R} (\bar{v}^Xe^{\bar{\phi}^X})_x \tilde{\phi}_{xt}\tilde{\phi}_t \, dx,
\end{split}
\]
\[
\begin{split}
\RNum{2}_4 & := \dot{X}(t) \int_\mathbb{R}  \tilde{\phi}_{xxt}  \left( \frac{1}{v} \left(  \frac{\bar{\phi}^X_x  \bar{v}^X_x}{\bar{v}^X} \right)_x - \frac{\bar{\phi}^X_x ( \bar{v}^X_x )^2}{v^2\bar{v}^X} + e^{\bar{\phi}^X}\bar{v}^X_x \right) \, dx, \\
\RNum{2}_5 & := \int_\mathbb{R} \tilde{\phi}_{xxt}  \left( (\sigma + \dot{X}(t)) (e^{\bar{\phi}^X})_x \tilde{v} - e^{\bar{\phi}^X} \tilde{u}_x - \dot{X}(t) e^{\bar{\phi}^X} \bar{v}^X_x  \right) (1-e^{\tilde{\phi}}) \, dx \\
& \quad - \int_\mathbb{R} \tilde{\phi}_{xxt}  e^{\bar{\phi}^X}\tilde{v} (1-e^{\tilde{\phi}})_t \, dx,
\end{split}
\]
and
\[
\RNum{2}_6 := \int_\mathbb{R} \tilde{\phi}_{xxt} \left( (\sigma + \dot{X}(t) ) (\bar{v}^X e^{\bar{\phi}^X})_x (1+\tilde{\phi}-e^{\tilde{\phi}}) - \bar{v}^Xe^{\bar{\phi}^X} (1+\tilde{\phi}-e^{\tilde{\phi}})_t \right) \, dx.
\]

We estimate the terms $\RNum{2}_j$ for $j=1,\dots,6$. First, by Young's inequality, we have
\[
\lvert \RNum{2}_1 \rvert \leq C \int_\mathbb{R} \lvert \tilde{\phi}_{xxt} \rvert \lvert \tilde{u}_x \rvert \, dx \leq \eta \int_\mathbb{R} \tilde{\phi}_{xxt}^2  \, dx + \frac{C}{\eta} \int_\mathbb{R} \tilde{u}_x^2 \, dx
\]
for any $0 < \eta < 1$. The nonlinear term $\RNum{2}_2$ is estimated by using $\lvert \dot{X} \rvert \leq C \lVert \tilde{u} \rVert_{L^\infty}$ as
\[
\begin{split}
\lvert \RNum{2}_2 \rvert & \leq C \int_\mathbb{R} \left( \lvert \tilde{v} \rvert + \lvert \tilde{\phi}_x \rvert \right)  \lvert \tilde{\phi}_{xxt} \rvert  \left( \lvert \tilde{v}_x \rvert + \lvert \tilde{u}_x \rvert + \lvert \tilde{u}_{xx} \rvert  \right) \, dx \\
& \quad + C \int_\mathbb{R} \lvert \tilde{v}_x \rvert  \lvert \tilde{\phi}_{xxt} \rvert  \left( \lvert \tilde{\phi}_{xt} \rvert + \lvert \tilde{u}_x \rvert + \lvert \tilde{v}_x \rvert + \lvert \tilde{v} \rvert \lvert \tilde{u}_x \rvert + \lvert \tilde{u}_x \rvert \lvert \tilde{\phi}_x \rvert \right) \, dx \\
& \quad + C \int_\mathbb{R} \lvert \tilde{u}_x \rvert \lvert \tilde{\phi}_{xxt} \rvert  \lvert \tilde{\phi}_{xx} \rvert  \, dx + C \lvert \dot{X} \rvert \int_\mathbb{R} \lvert \tilde{\phi}_{xxt} \rvert \left( \lvert \tilde{v}_x \rvert + \lvert \tilde{v}_x \rvert^2 \right) \, dx \\
& \leq C \lVert (\tilde{v},\tilde{v}_x,\tilde{u}_x,\tilde{\phi}_x ) \rVert_{L^\infty}  \int_\mathbb{R} \left( \tilde{\phi}_{xt}^2 + \tilde{\phi}_{xxt}^2 + \tilde{u}_x^2 + \tilde{u}_{xx}^2 + \tilde{v}_x^2 + \tilde{\phi}_x^2 + \tilde{\phi}_{xx}^2 \right) \, dx. 
\end{split}
\]
We use the bounds $ \lvert \bar{\phi}^X_{xxx} \rvert, \lvert \bar{v}^X_{xx} \rvert, \lvert \bar{u}^X_{xx} \rvert, \lvert \bar{\phi}^X_{xx} \rvert \leq C \bar{v}^X_x \leq C \delta_S^2$ and $ \lvert \bar{u}^X_x \rvert \sim \lvert \bar{\phi}^X_x \rvert \sim \lvert \bar{v}^X_x \rvert$ to obtain
\[
\begin{split}
\lvert \RNum{2}_3 \rvert & \leq C \int_\mathbb{R} \lvert \bar{v}^X_x \rvert \lvert \tilde{\phi}_{xxt} \rvert  \left( \lvert \tilde{\phi}_t \rvert +  \lvert \tilde{\phi}_{xt} \rvert + \lvert \tilde{\phi}_x \rvert + \lvert \tilde{\phi}_{xx} \rvert + \lvert \tilde{u}_x \rvert + \lvert \tilde{u}_{xx} \rvert + \lvert \tilde{v}_x \rvert \right) \, dx \\
& \quad + C \left( 1 + \lvert \dot{X} \rvert \right)  \int_\mathbb{R} \bar{v}^X_x \lvert \tilde{\phi}_{xxt} \rvert \left( \lvert \tilde{v} \rvert + \lvert \tilde{v}_x \rvert + \lvert \tilde{\phi} \rvert \right) \, dx \\
& \leq C \int_\mathbb{R} \lvert \bar{v}^X_x \rvert \left( \tilde{\phi}_t^2 + \tilde{\phi}_{xt}^2 + \tilde{\phi}_{xxt}^2 + \tilde{\phi}_x^2 + \tilde{\phi}_{xx}^2 + \tilde{u}_x^2 + \tilde{u}_{xx}^2 + \tilde{v}_x^2 \right) \, dx \\
& \quad + C \left( 1 + \lVert \tilde{u} \rVert_{L^\infty} \right) \left( \int_\mathbb{R} \lvert \bar{v}^X_x \rvert^{1/2} \left( \tilde{\phi}_{xxt}^2 + \tilde{v}_x^2 \right) \, dx + \int_\mathbb{R} \lvert \bar{v}^X_x \rvert^{3/2} \left( \tilde{v}^2 + \tilde{\phi}^2 \right) \, dx \right) \\
& \leq C \delta_S^2 \int_\mathbb{R} \left( \tilde{\phi}_t^2 + \tilde{\phi}_{xt}^2 + \tilde{\phi}_{xxt}^2 + \tilde{\phi}_x^2 + \tilde{\phi}_{xx}^2 + \tilde{u}_x^2 + \tilde{u}_{xx}^2 + \tilde{v}_x^2 \right) \, dx \\
& \quad + C \delta_S \left( 1 + \lVert \tilde{u} \rVert_{L^\infty} \right) \int_\mathbb{R} \left( \bar{v}^X_x \tilde{v}^2 + \bar{v}^X_x \tilde{\phi}^2 + \tilde{\phi}_{xxt}^2 + \tilde{v}_x^2 \right) \, dx.
\end{split}
\]
For $\RNum{2}_4$, we apply the H\"older inequality and Young's inequality to obtain
\[
\begin{split}
\lvert \RNum{2}_4 \rvert & \leq C \lvert \dot{X} \rvert \int_\mathbb{R} \lvert \bar{v}^X_x \rvert \lvert \tilde{\phi}_{xxt} \rvert  \, dx \\
& \leq C \lvert \dot{X} \rvert \sqrt{\int_\mathbb{R} \bar{v}^X_x \, dx} \sqrt{ \int_\mathbb{R} \bar{v}^X_x \tilde{\phi}_{xxt}^2 \, dx} \\
& \leq C \delta_S^2 \lvert \dot{X} \rvert^2 + \frac{C}{\delta_S^2} \left( \int_\mathbb{R} \bar{v}^X_x \, dx \right) \left( \int_\mathbb{R} \bar{v}^X_x \tilde{\phi}_{xxt}^2  \, dx \right) \\
& \leq C \delta_S \left( \delta_S \lvert \dot{X} \rvert^2 + \int_\mathbb{R} \tilde{\phi}_{xxt}^2  \, dx \right).
\end{split}
\]
Similar to $\RNum{2}_2$, the nonlinear terms $\RNum{2}_5$ and $\RNum{2}_6$ are estimated as 
\[
\begin{split}
\lvert \RNum{2}_5 \rvert & \leq C \int_\mathbb{R} \lvert \tilde{\phi}_{xxt} \rvert \left( \bar{v}^X_x \lvert \tilde{v} \rvert + \lvert \tilde{u}_x \rvert \right) \mathcal{O}(\lvert \tilde{\phi} \rvert) \, dx  + C \lvert \dot{X} \rvert \int_\mathbb{R} \lvert \tilde{\phi}_{xxt} \rvert \left( \bar{v}^X_x + \bar{v}^X_x \lvert \tilde{v} \rvert \right) \mathcal{O}(\lvert \tilde{\phi} \rvert) \, dx \\
& \quad + C \int_\mathbb{R} \lvert \tilde{v} \rvert \lvert \tilde{\phi}_{xxt} \rvert  \left( \lvert \tilde{\phi}_t \rvert + \lvert \tilde{\phi}_t \rvert \mathcal{O}(\lvert \tilde{\phi} \rvert) \right) \, dx\\
& \leq C \lVert ( \tilde{v}, \tilde{u}, \tilde{\phi} ) \rVert_{L^\infty} \int_\mathbb{R} \left( \bar{v}^X_x \tilde{v}^2 + \bar{v}^X_x \tilde{\phi}^2 + \tilde{\phi}_t^2 + \tilde{\phi}_{xxt}^2 + \tilde{u}_x^2 \right) \, dx
\end{split}
\]
and
\[
\begin{split}
\lvert \RNum{2}_6 \rvert & \leq C (1 + \lvert \dot{X} \rvert) \int_\mathbb{R} \bar{v}^X_x \lvert \tilde{\phi}_{xxt} \rvert \mathcal{O}(\lvert \tilde{\phi} \rvert^2 )  \, dx + C \int_\mathbb{R}  \lvert \tilde{\phi}_{xxt} \rvert \lvert \tilde{\phi}_t \rvert \mathcal{O}(\lvert \tilde{\phi} \rvert) \, dx \\
& \leq C \lVert ( \tilde{u}, \tilde{\phi} ) \rVert_{L^\infty} \int_\mathbb{R} \left( \bar{v}^X_x \tilde{\phi}^2 + \tilde{\phi}_t^2 + \tilde{\phi}_{xxt}^2 \right) \, dx,
\end{split}
\]
respectively. 

Collecting all the estimates for $\RNum{2}_j$ and taking $\eta$ sufficiently small, we obtain
\[
\begin{split}
\int_\mathbb{R} \left( \tilde{\phi}_{xt}^2 +\tilde{\phi}_{xxt}^2 \right) \, dx & \leq C \left( \delta_S + \lVert ( \tilde{v}, \tilde{v}_x, \tilde{u}_x, \tilde{\phi}, \tilde{\phi}_x ) \rVert_{L^\infty} \right) \\
& \qquad \times \int_\mathbb{R} \left( \bar{v}^X_x \tilde{v}^2 + \bar{v}^X_x \tilde{\phi}^2 + \tilde{v}_x^2 + \tilde{u}_{xx}^2 + \tilde{\phi}_t^2 + \tilde{\phi}_x^2 + \tilde{\phi}_{xx}^2 \right) \, dx + C \int_\mathbb{R} \tilde{u}_x^2 \, dx. 
\end{split}
\]
Therefore, the estimate \eqref{phixt} follows from the results of Lemmas~\ref{Lv^2}, \ref{Lemma:phi}, and \ref{Lemma:phit}.
\end{proof}

\section{Higher-order estimates and proof of Proposition \ref{Apriori}}

We provide the estimates of the first and second derivatives of the perturbation $(\tilde{v},\tilde{u})$ and prove Proposition \ref{Apriori} by combining these results with those obtained in previous sections. First, we observe that the estimate \eqref{Ree}, together with \eqref{derivphi} and \eqref{phixt}, gives
\begin{equation} \label{Ree'}
\begin{split}
& \lVert (\tilde{v},\tilde{u}) (t,\cdot) \rVert_{L^2}^2 + \lVert \tilde{\phi} (t,\cdot) \rVert_{H^2}^2 + \int_0^t \left( \delta_S \lvert \dot{X} \rvert^2 + G_1 + G^S + D \right) \, d\tau \\
& \leq C \left( \lVert ( \tilde{v}_0,\tilde{u}_0 ) \rVert_{L^2}^2 + \lVert \tilde{v}_{0x} \rVert_{L^2}^2 \right) + C \left( \sqrt{\delta_1} + \varepsilon_1 \right) \int_0^t \left( \lVert \tilde{v}_x \rVert_{L^2}^2 + \lVert \tilde{u}_{xx} \rVert_{L^2}^2 \right) \, d\tau
\end{split}
\end{equation}
for $t \in [0,T]$. To close this a priori estimate, we need $H^1$-estimates for $(\tilde{v},\tilde{u})$ which yield good terms absorbing the last term on the right-hand side of the above inequality. However, as a consequence of Lemma~\ref{Lemma:v1}, one can find that further estimates up to $H^2$ are necessary. Accordingly, the next two subsections focus on deriving the higher-order estimates for the perturbation $(\tilde{v},\tilde{u})$.

\subsection{\texorpdfstring{$H^1$-estimates}{H1-estimates}}

\begin{lemma} \label{Lemma:v1}
Under the assumptions in Proposition \ref{Apriori}, there exists a constant $C>0$ such that
\begin{equation} \label{derivv1}
\begin{split}
& \lVert \tilde{v}_x (t,\cdot) \rVert_{L^2}^2 + \int_0^t \left( \lVert \tilde{v}_x \rVert_{L^2}^2 + \lVert \tilde{\phi}_{xx} \rVert_{H^1}^2 \right) \, d\tau \\
& \quad \leq C \left( \lVert ( \tilde{v}_0,\tilde{u}_0) \rVert_{L^2}^2 + \lVert \tilde{v}_{0x} \rVert_{L^2}^2 \right) \\
& \qquad + C \left( \delta_1 + \varepsilon_1 \right) \int_0^t \left( \delta_S \lvert \dot{X} \rvert^2 + G_1 + G^S + \lVert \tilde{v}_{xx} \rVert_{L^2}^2 + \lVert \tilde{u}_{xx}  \rVert_{L^2}^2 \right) \, d\tau
\end{split}
\end{equation}
for all $t \in [0,T]$, where $G_1$ and $G^S$ are as defined in Lemma~\ref{RE}.
\end{lemma}

\begin{proof}
Differentiating \eqref{1a'} with respect to $x$ and multiplying the resultant equation by $\tilde{v}_x$, and multiplying \eqref{1b} by $- v \tilde{v}_x$, we have
\begin{equation} \label{'1}
\left( \frac{\tilde{v}_x^2}{2}  \right)_t - \tilde{v}_x \tilde{u}_{xx} = \dot{X}(t) \bar{v}^X_{xx} \tilde{v}_x, 
\end{equation}
and
\begin{equation} \label{'2}
\begin{split}
& - v\tilde{v}_x \tilde{u}_t - v \tilde{v}_x \left( \tilde{p}(v) - \tilde{p}(\bar{v}^X) \right)_x + v \tilde{v}_x \left(\frac{u_x}{v}-\frac{\bar{u}^X_x}{\bar{v}^X} \right)_x \\
& \quad = - \frac{1}{2} v \tilde{v}_x \left[ \left(\frac{\phi_x}{v} \right)^2 - \left( \frac{\bar{\phi}^X_x}{\bar{v}^X} \right)^2 \right]_x + v \tilde{v}_x \left[\frac{1}{v} \left( \frac{\phi_x}{v} \right)_x - \frac{1}{\bar{v}^X} \left( \frac{\bar{\phi}^X_x}{\bar{v}^X} \right)_x \right]_x - \dot{X}(t) v \tilde{v}_x \bar{u}^X_x,
\end{split}
\end{equation}
respectively. For the first term on the left-hand side of \eqref{'2}, we use \eqref{NSP11} and \eqref{1a'} to obtain
\[
\begin{split}
- v\tilde{v}_x \tilde{u}_t & = \left( - v \tilde{v}_x \tilde{u} \right)_t + v_t \tilde{v}_x \tilde{u} + v \tilde{v}_{xt} \tilde{u} \\
& = \left( - v \tilde{v}_x \tilde{u} \right)_t + \left( \tilde{u}_x + \bar{u}^X_x \right) \tilde{v}_x \tilde{u}+ v \left( \tilde{u}_{xx} + \dot{X}(t) \bar{v}^X_{xx} \right) \tilde{u} \\
& = \left( - v \tilde{v}_x \tilde{u} \right)_t + \bar{u}^X_x \tilde{v}_x \tilde{u} + (\cdots)_x - \bar{v}^X_x \tilde{u}_x \tilde{u} - v \tilde{u}_x^2 + \dot{X}(t) v \bar{v}^X_{xx} \tilde{u}.
\end{split}
\]
Then, by summing \eqref{'1} and \eqref{'2}, we have 
\begin{equation} \label{'3}
\begin{split}
& \left( \frac{\tilde{v}_x^2}{2} - v \tilde{v}_x \tilde{u} \right)_t - v \tilde{v}_x \left( \tilde{p}(v) - \tilde{p}(\bar{v}^X) \right)_x - v \tilde{v}_x \left[\frac{1}{v} \left( \frac{\phi_x}{v} \right)_x - \frac{1}{\bar{v}^X} \left( \frac{\bar{\phi}^X_x}{\bar{v}^X} \right)_x \right]_x \\
& \quad = (\cdots)_x - \bar{u}^X_x \tilde{v}_x \tilde{u} + \bar{v}^X_x \tilde{u}_x \tilde{u} + v \tilde{u}_x^2  + \frac{(\tilde{v}_x + \bar{v}^X_x)\tilde{v}_x \tilde{u}_x}{v} + v \tilde{v}_x \left( \frac{\bar{u}^X_x \tilde{v}}{v\bar{v}^X} \right)_x \\
& \qquad - \frac{1}{2} v \tilde{v}_x \left[ \left(\frac{\phi_x}{v} \right)^2 - \left( \frac{\bar{\phi}^X_x}{\bar{v}^X} \right)^2 \right]_x + \dot{X}(t) \left( - \bar{u}^X_x v \tilde{v}_x + \bar{v}^X_{xx}\tilde{v}_x - v \bar{v}^X_{xx}\tilde{u} \right).
\end{split} 
\end{equation}
Here, the second and third terms on the left-hand side are expanded as
\begin{equation} \label{'4}
\begin{split}
- v \tilde{v}_x \left( \tilde{p}(v) - \tilde{p}(\bar{v}^X) \right)_x = \frac{2\tilde{v}_x^2}{\bar{v}^X} - \frac{2\tilde{v}_x^2 \tilde{v}}{v\bar{v}^X} - \frac{2\bar{v}^X_x \tilde{v}_x \tilde{v}}{v\bar{v}^X} - \frac{2\bar{v}^X_x \tilde{v}_x \tilde{v}}{(\bar{v}^X)^2}
\end{split}
\end{equation}
and
\begin{equation} \label{'5}
\begin{split}
& - v \tilde{v}_x \left[\frac{1}{v} \left( \frac{\phi_x}{v} \right)_x - \frac{1}{\bar{v}^X} \left( \frac{\bar{\phi}^X_x}{\bar{v}^X} \right)_x \right]_x \\
& \quad = - \frac{\tilde{v}_x\tilde{\phi}_{xxx}}{\bar{v}^X} + \frac{\tilde{v}_x^2 \tilde{\phi}_{xx}}{v\bar{v}^X} + \frac{\bar{v}^X_x \tilde{v}_x \tilde{\phi}_{xx}}{v\bar{v}^X} + \frac{\bar{v}^X_x \tilde{v}_x \tilde{\phi}_{xx}}{(\bar{v}^X)^2}  + v\tilde{v}_x \left( \frac{\tilde{\phi}_{xx}\tilde{v}}{v^2\bar{v}^X} + \frac{2\bar{\phi}^X_{xx}\tilde{v}}{v^2\bar{v}^X} + \frac{\bar{\phi}^X_{xx}\tilde{v}^2}{v^2(\bar{v}^X)^2} \right)_x \\
& \qquad - \left( v \tilde{v}_x \right)_x \left( \frac{\tilde{v}_x\tilde{\phi}_x}{v^3} + \frac{\bar{v}^X_x \tilde{\phi}_x}{v^3} + \frac{\bar{\phi}^X_x \tilde{v}_x}{v^3} + \left( \frac{1}{v^3} - \frac{1}{(\bar{v}^X)^3} \right) \bar{\phi}^X_{xx} \right) + (\cdots)_x,
\end{split}
\end{equation}
respectively. To handle the first term on the right-hand side of \eqref{'5}, we use the Poisson equation for the perturbation to obtain a quadratic form consisting of $\tilde{v}_x$ and $\tilde{\phi}_{xxx}$. We rewrite \eqref{1c'} as
\[
\begin{split}
\tilde{v}_x & = \left[ e^{-\bar{\phi}^X} \left( \frac{\phi_x}{v} - \frac{\bar{\phi}^X_x}{\bar{v}^X} \right)_x \right]_x + \tilde{v}_x (1-e^{\tilde{\phi}}) + \tilde{v} (1-e^{\tilde{\phi}})_x \\
& \quad - \bar{v}^X \tilde{\phi}_x + \bar{v}^X_x (1-e^{\tilde{\phi}}) + \bar{v}^X (1+\tilde{\phi}-e^{\tilde{\phi}})_x.
\end{split}
\]
Multiplying it by $-\tilde{\phi}_{xxx}/\bar{v}^X$, we have after rearrangement
\begin{equation} \label{'7}
\begin{split}
& -\frac{ \tilde{v}_x \tilde{\phi}_{xxx}}{\bar{v}^X} + \frac{e^{-\bar{\phi}^X}\tilde{\phi}_{xxx}^2}{(\bar{v}^X)^2} + \tilde{\phi}_{xx}^2 \\
& \quad = \frac{e^{-\bar{\phi}^X}\tilde{\phi}_{xxx}}{\bar{v}^X} \left( \frac{2 \bar{v}^X_x \tilde{\phi}_{xx}}{(\bar{v}^X)^2} + \frac{\bar{v}^X_{xx}\tilde{\phi}_x}{(\bar{v}^X)^2} - \frac{2 (\bar{v}^X_x)^2 \tilde{\phi}_x}{(\bar{v}^X)^3} \right) + \frac{e^{\bar{\phi}^X}\tilde{\phi}_{xxx}}{\bar{v}^X} \left( \frac{\tilde{\phi}_x \tilde{v}}{v\bar{v}^X} + \frac{\bar{\phi}^X_x\tilde{v}}{v\bar{v}^X} \right)_{xx} \\
& \qquad + \frac{e^{-\bar{\phi}^X }\bar{\phi}^X_x \tilde{\phi}_{xxx}}{\bar{v}^X} \left( \frac{\tilde{\phi}_x}{v} - \frac{\bar{\phi}^X_x \tilde{v}}{v\bar{v}^X} \right)_x  - \frac{\tilde{v}_x \tilde{\phi}_{xxx}}{\bar{v}^X} (1-e^{\tilde{\phi}})  \\
& \qquad - \frac{\tilde{v} \tilde{\phi}_{xxx}}{\bar{v}^X} (1-e^{\tilde{\phi}})_x + (\cdots)_x - \frac{\bar{v}^X_x \tilde{\phi}_{xxx}}{\bar{v}^X} (1-e^{\tilde{\phi}}) - \tilde{\phi}_{xxx} (1+\tilde{\phi}-e^{\tilde{\phi}})_x.
\end{split}
\end{equation}
Collecting \eqref{'3}-\eqref{'7}, and using the expansion
\[
\left[ \left( \frac{\phi_x}{v} \right)^2 - \left( \frac{\bar{\phi}^X_x}{\bar{v}^X} \right)^2 \right] = \frac{\tilde{\phi}_x^2}{v^2} + \frac{2 \bar{\phi}^X_x \tilde{\phi}_x}{v^2} - \frac{(\bar{\phi}^X_x)^2 \tilde{v}^2}{v^2 (\bar{v}^X)^2} - \frac{2(\bar{\phi}^X_x)^2 \tilde{v}}{v^2 \bar{v}^X},
\] 
we obtain, after integration by parts,
\[
\begin{split}
& \left( \frac{\tilde{v}_x^2}{2} - v\tilde{v}_x\tilde{u} \right)_t + \frac{2\tilde{v}_x^2}{\bar{v}^X} - \frac{2 \tilde{v}_x \tilde{\phi}_{xxx}}{\bar{v}^X} + \frac{e^{-\bar{\phi}^X}\tilde{\phi}_{xxx}^2}{(\bar{v}^X)^2} + \tilde{\phi}_{xx}^2  \\
& \quad = \frac{2\tilde{v}_x^2 \tilde{v}}{v\bar{v}^X} + \frac{2\bar{v}^X_x \tilde{v}_x \tilde{v}}{v\bar{v}^X} + \frac{2\bar{v}^X_x \tilde{v}_x \tilde{v}}{(\bar{v}^X)^2} - \frac{\tilde{v}_x^2 \tilde{\phi}_{xx}}{v\bar{v}^X} - \frac{\bar{v}^X_x \tilde{v}_x \tilde{\phi}_{xx}}{v\bar{v}^X} - \frac{\bar{v}^X_x \tilde{v}_x \tilde{\phi}_{xx}}{(\bar{v}^X)^2} \\
& \qquad + \left( v\tilde{v}_x \right)_x \left( \frac{\tilde{\phi}_{xx}\tilde{v}}{v^2\bar{v}^X} + \frac{2\bar{\phi}^X_{xx}\tilde{v}}{v^2\bar{v}^X} + \frac{\bar{\phi}^X_{xx}\tilde{v}^2}{v^2(\bar{v}^X)^2} \right) \\
& \qquad + \left( v \tilde{v}_x \right)_x \left( \frac{\tilde{v}_x\tilde{\phi}_x}{v^3} + \frac{\bar{v}^X_x \tilde{\phi}_x}{v^3} + \frac{\bar{\phi}^X_x \tilde{v}_x}{v^3} + \left( \frac{1}{v^3} - \frac{1}{(\bar{v}^X)^3} \right) \bar{\phi}^X_{xx} \right) \\
& \qquad - \bar{u}^X_x \tilde{v}_x \tilde{u} + \bar{v}^X_x \tilde{u}_x \tilde{u} + v \tilde{u}_x^2 + \frac{(\tilde{v}_x + \bar{v}^X_x)\tilde{v}_x \tilde{u}_x}{v} - \left( v \tilde{v}_x \right)_x \left( \frac{\bar{u}^X_x \tilde{v}}{v\bar{v}^X} \right) \\
& \qquad + \frac{1}{2} \left( v \tilde{v}_x \right)_x \left( \frac{\tilde{\phi}_x^2}{v^2} + \frac{2\bar{\phi}^X_x \tilde{\phi}_x}{v^2} - \frac{(\bar{\phi}^X_x)^2 \tilde{v}^2}{v^2 (\bar{v}^X)^2} - \frac{2(\bar{\phi}^X_x)^2 \tilde{v}}{v^2 \bar{v}^X} \right) + \dot{X}(t) \left( - \bar{u}^X_x v \tilde{v}_x + \bar{v}^X_{xx}\tilde{v}_x - v \bar{v}^X_{xx}\tilde{u} \right) \\
& \qquad + \frac{e^{-\bar{\phi}^X}\tilde{\phi}_{xxx}}{\bar{v}^X} \left( \frac{2 \bar{v}^X_x \tilde{\phi}_{xx}}{(\bar{v}^X)^2} + \frac{\bar{v}^X_{xx}\tilde{\phi}_x}{(\bar{v}^X)^2} - \frac{2 (\bar{v}^X_x)^2 \tilde{\phi}_x}{(\bar{v}^X)^3} \right) + \frac{e^{-\bar{\phi}^X}\tilde{\phi}_{xxx}}{\bar{v}^X} \left( \frac{\tilde{\phi}_x \tilde{v}}{v\bar{v}^X} + \frac{\bar{\phi}^X_x\tilde{v}}{v\bar{v}^X} \right)_{xx} \\
& \qquad + \frac{e^{-\bar{\phi}^X }\bar{\phi}^X_x \tilde{\phi}_{xxx}}{\bar{v}^X} \left( \frac{\tilde{\phi}_x}{v} - \frac{\bar{\phi}^X_x \tilde{v}}{v\bar{v}^X} \right)_x  - \frac{\tilde{v}_x \tilde{\phi}_{xxx}}{\bar{v}^X} (1-e^{\tilde{\phi}}) - \frac{\tilde{v} \tilde{\phi}_{xxx}}{\bar{v}^X} (1-e^{\tilde{\phi}})_x  \\
& \qquad + (\cdots)_x - \frac{\bar{v}^X_x \tilde{\phi}_{xxx}}{\bar{v}^X} (1-e^{\tilde{\phi}}) - \tilde{\phi}_{xxx} (1+\tilde{\phi}-e^{\tilde{\phi}})_x.
\end{split}
\]
By integrating this equation with respect to $x$, we have
\begin{equation} \label{'8}
\begin{split}
\frac{d}{dt} \int_\mathbb{R} \bigg( \frac{\tilde{v}_x^2}{2} - v\tilde{v}_x\tilde{u} \bigg) \, dx + \int_\mathbb{R} \bigg( \frac{2\tilde{v}_x^2}{\bar{v}^X} - \frac{2 \tilde{v}_x \tilde{\phi}_{xxx}}{\bar{v}^X} + \frac{e^{-\bar{\phi}^X}\tilde{\phi}_{xxx}^2}{(\bar{v}^X)^2} + \tilde{\phi}_{xx}^2 \bigg) \, dx = \sum_{j=1}^8 \mathcal{V}^{(1)}_j,
\end{split}
\end{equation}
where
\[
\begin{split}
\mathcal{V}^{(1)}_1 & := \int_\mathbb{R} v \tilde{u}_x^2 \, dx, \quad \mathcal{V}^{(1)}_2 := - \dot{X}(t) \int_\mathbb{R} \left( \bar{u}^X_x v \tilde{v}_x - \bar{v}^X_{xx}\tilde{v}_x + v \bar{v}^X_{xx} \tilde{u} \right) \, dx,
\end{split}
\]

\[
\begin{split}
\mathcal{V}^{(1)}_3 & := \int_\mathbb{R} \bigg( \frac{2 \tilde{v}_x^2 \tilde{v}}{v\bar{v}^X} - \frac{\tilde{v}_x^2 \tilde{\phi}_{xx}}{v\bar{v}^X} \bigg) \, dx + \int_\mathbb{R} \left( \tilde{v}_x^2 + \bar{v}^X_x \tilde{v}_x + v \tilde{v}_{xx} \right) \left( \frac{\tilde{\phi}_{xx} \tilde{v}}{v^2 \bar{v}^X} + \frac{\bar{\phi}^X_x \tilde{v}^2}{v^2 (\bar{v}^X)^2} \right) \, dx \\
& \quad + \int_\mathbb{R} \frac{2 \bar{\phi}^X_{xx} \tilde{v}_x^2 \tilde{v}}{v^2 \bar{v}^X} \, dx + \int_\mathbb{R} \left( \tilde{v}_x^2 + \bar{v}^X_x \tilde{v}_x + v \tilde{v}_{xx} \right) \left( \frac{\tilde{v}_x \tilde{\phi}_x}{v^3} - \frac{\bar{v}^X_x \bar{\phi}^X_x \tilde{v}^3}{v^3(\bar{v}^X)^3} - \frac{3 \bar{v}^X_x \bar{\phi}^X_x\tilde{v}^2}{v^3(\bar{v}^X)^3} \right) \, dx \\
& \quad + \int_\mathbb{R} \tilde{v}_x^2 \left( \frac{\bar{v}^X_x \tilde{\phi}_x}{v^3} + \frac{\bar{\phi}^X_x \tilde{v}_x}{v^3} - \frac{3\bar{v}^X_x\bar{\phi}^X_x \tilde{v}}{v^3\bar{v}^X} \right) \, dx + \int_\mathbb{R} \frac{\tilde{v}_x^2 \tilde{u}_x}{v} \, dx - \int_\mathbb{R} \frac{\bar{u}^X_x \tilde{v}_x^2 \tilde{v}}{v\bar{v}^X} \, dx \\
& \quad + \frac{1}{2} \int_\mathbb{R} \left( \tilde{v}_x^2 + \bar{v}^X_x \tilde{v}_x + v\tilde{v}_{xx} \right) \left( \frac{\tilde{\phi}_x^2}{v^2} - \frac{(\bar{\phi}^X_x)^2 \tilde{v}^2}{v^2(\bar{v}^X)^2} \right) \, dx + \int_\mathbb{R} \tilde{v}_x^2 \left( \frac{\bar{\phi}^X_x \tilde{\phi}_x}{v^2} - \frac{(\bar{\phi}^X_x)^2 \tilde{v}}{v^2 \bar{v}^X} \right) \, dx,
\end{split}
\]

\[
\begin{split}
\mathcal{V}^{(1)}_4 & := \int_\mathbb{R} \frac{e^{-\bar{\phi}^X}\tilde{\phi}_{xxx}}{\bar{v}^X} \left( \frac{\tilde{\phi}_{xxx}\tilde{v}}{v\bar{v}^X} + \frac{2\tilde{\phi}_{xx}\tilde{v}_x}{v\bar{v}^X} - \frac{2\tilde{v}_x \tilde{\phi}_{xx} \tilde{v}}{v^2\bar{v}^X} - \frac{2\bar{v}^X_x \tilde{\phi}_{xx} \tilde{v}}{v^2\bar{v}^X} - \frac{2\bar{v}^X_x \tilde{\phi}_{xx}\tilde{v}}{v(\bar{v}^X)^2} \right) \, dx \\
& \quad + \int_\mathbb{R} \frac{e^{-\bar{\phi}^X}\tilde{\phi}_{xxx}}{\bar{v}^X} \left( \frac{\tilde{\phi}_x \tilde{v}_{xx}}{v\bar{v}^X} - \frac{2\tilde{\phi}_x \tilde{v}_x^2}{v^2\bar{v}^X} - \frac{2\bar{v}^X_x \tilde{\phi}_x \tilde{v}_x}{v^2 \bar{v}^X} - \frac{2\bar{v}^X_x \tilde{\phi}_x \tilde{v}_x}{v(\bar{v}^X)^2} - \frac{\tilde{v}_{xx}\tilde{\phi}_x \tilde{v}}{v^2\bar{v}^X} \right) \, dx \\
& \quad + \int_\mathbb{R} \frac{e^{-\bar{\phi}^X}\tilde{\phi}_{xxx}}{\bar{v}^X} \left( \frac{2\tilde{v}_x^2 \tilde{\phi}_x \tilde{v}}{v^3 \bar{v}^X} + \frac{4\bar{v}^X_x \tilde{v}_x \tilde{\phi}_x \tilde{v}}{v^3\bar{v}^X} + \frac{2\bar{v}^X_x \tilde{v}_x \tilde{\phi}_x \tilde{v}}{v^2(\bar{v}^X)^2} + \frac{2(\bar{v}^X_x)^2 \tilde{\phi}_x \tilde{v}}{v^3\bar{v}^X} \right) \, dx \\
& \quad - \int_\mathbb{R} \frac{e^{-\bar{\phi}^X}\tilde{\phi}_{xxx}}{\bar{v}^X} \left( \left( \frac{\bar{v}^X_x}{\bar{v}^X} \right)_x \frac{\tilde{\phi}_x \tilde{v}}{v^2} - \frac{(\bar{v}^X_x)^2 \tilde{\phi}_x \tilde{v}}{v^2 (\bar{v}^X)^2} + \left( \frac{\bar{v}^X_x}{(\bar{v}^X)^2} \right)_x \frac{\tilde{\phi}_x \tilde{v}}{v} + \frac{2 \bar{\phi}^X_x \tilde{v}_x^2}{v^2 \bar{v}^X} \right) \, dx \\
& \quad - \int_\mathbb{R} \frac{e^{-\bar{\phi}^X}\tilde{\phi}_{xxx}}{\bar{v}^X} \left( \frac{\bar{\phi}^X_x \tilde{v}_{xx}\tilde{v}}{v^2 \bar{v}^X} + \frac{2\bar{\phi}^X_x \tilde{v}_x^2 \tilde{v}}{v^3\bar{v}^X} - \frac{4\bar{\phi}^X_x \bar{v}^X_x \tilde{v}_x \tilde{v}}{v^3\bar{v}^X} + \left( \frac{\bar{\phi}^X_x}{\bar{v}^X} \right)_x \frac{2\tilde{v}_x \tilde{v}}{v^2}  \right) \, dx \\
& \quad - \int_\mathbb{R} \frac{e^{-\bar{\phi}^X} \bar{\phi}^X_x \tilde{\phi}_{xxx}}{\bar{v}^X} \left( \frac{\tilde{v}_x \tilde{\phi}_x}{v^2} + \frac{\bar{\phi}^X_x \tilde{v}_x \tilde{v}}{v^2 \bar{v}^X} \right) \, dx,
\end{split}
\]

\[
\begin{split}
\mathcal{V}^{(1)}_5 & := \int_\mathbb{R} \bigg( \frac{2 \bar{v}^X_x \tilde{v}_x \tilde{v}}{v \bar{v}^X} + \frac{2 \bar{v}^X_x \tilde{v}_x \tilde{v}}{(\bar{v}^X)^2} \bigg) \, dx + \int_\mathbb{R} \frac{2\bar{\phi}^X_{xx} \tilde{v}_{xx} \tilde{v}}{v\bar{v}^X} \, dx - \int_\mathbb{R} \left( \bar{u}^X_x \tilde{v}_x \tilde{u} - \bar{v}^X_x \tilde{u}_x \tilde{u} \right) \, dx \\
& \quad - \int_\mathbb{R} \frac{\bar{u}^X_x \tilde{v}_{xx} \tilde{v}}{\bar{v}^X} \, dx + \int_\mathbb{R} \left( \frac{\bar{\phi}^X_x}{\bar{v}^X} \right)_x \frac{e^{-\bar{\phi}^X}\tilde{\phi}_{xxx}\tilde{v}}{v\bar{v}^X} \, dx,
\end{split}
\]

\[
\begin{split}
\mathcal{V}^{(1)}_6 & := - \int_\mathbb{R} \frac{\bar{v}^X_x \tilde{v}_x \tilde{\phi}_{xx}}{v\bar{v}^X} \, dx - \int_\mathbb{R} \frac{\bar{v}^X_x \tilde{v}_x \tilde{\phi}_{xx}}{(\bar{v}^X)^2} \, dx + \int_\mathbb{R} \frac{2 \bar{\phi}^X_{xx}\bar{v}^X_x \tilde{v}_x \tilde{v}}{v^2 \bar{v}^X} \, dx \\
& \quad + \int_\mathbb{R} \left( \bar{v}^X_x \tilde{v}_x + v\tilde{v}_{xx} \right) \left( \frac{\bar{v}^X_x \tilde{\phi}_x}{v^3} + \frac{\bar{\phi}^X_x \tilde{v}_x}{v^3} - \frac{3\bar{v}^X_x\bar{\phi}^X_x \tilde{v}}{v^3\bar{v}^X} \right) \, dx + \int_\mathbb{R} \frac{\bar{v}^X_x \tilde{v}_x \tilde{u}_x}{v} \, dx \\
& \quad - \int_\mathbb{R} \frac{\bar{v}^X_x \bar{u}^X_x \tilde{v}_x \tilde{v}}{v\bar{v}^X} \, dx + \int_\mathbb{R} \left( \bar{v}^X_x \tilde{v}_x + v\tilde{v}_{xx} \right) \left( \frac{\bar{\phi}^X_x \tilde{\phi}_x}{v^2} - \frac{(\bar{\phi}^X_x)^2 \tilde{v}}{v^2\bar{v}^X} \right) \, dx,
\end{split}
\]

\[
\begin{split}
\mathcal{V}^{(1)}_7 & := \int_\mathbb{R} \frac{e^{-\bar{\phi}^X}\tilde{\phi}_{xxx}}{\bar{v}^X} \left( \frac{2\bar{v}^X_x \tilde{\phi}_{xx}}{(\bar{v}^X)^2} + \frac{\bar{v}^X_{xx} \tilde{\phi}_x}{(\bar{v}^X)^2} - \frac{2(\bar{v}^X_x)^2 \tilde{\phi}_x}{(\bar{v}^X)^3} \right) \, dx \\
& \quad + \int_\mathbb{R} \frac{e^{-\bar{\phi}^X}\tilde{\phi}_{xxx}}{\bar{v}^X} \left( \frac{\bar{\phi}^X_x \tilde{v}_{xx}}{v\bar{v}^X} - \frac{2\bar{v}^X_x \bar{\phi}^X_x \tilde{v}_x}{v^2\bar{v}^X} + \left( \frac{\bar{\phi}^X_x}{\bar{v}^X} \right)_x \frac{2\tilde{v}_x}{v} \right) \, dx \\
& \quad + \int_\mathbb{R}  \frac{e^{-\bar{\phi}^X}\tilde{\phi}_{xxx}}{\bar{v}^X} \left( \frac{2\bar{\phi}^X_x (\bar{v}^X_x)^2 \tilde{v}}{v^3\bar{v}^X} - \left( \frac{\bar{\phi}^X_x \bar{v}^X_x}{\bar{v}^X} \right)_x \frac{\tilde{v}}{v^2} - \left( \frac{\bar{\phi}^X_x}{\bar{v}^X} \right)_x \frac{\bar{v}^X_x \tilde{v}}{v^2} \right) \, dx \\
& \quad + \int_\mathbb{R} \frac{e^{-\bar{\phi}^X}\bar{\phi}^X_x \tilde{\phi}_{xxx}}{\bar{v}^X} \left( \frac{\tilde{\phi}_{xx}}{v} - \frac{\bar{v}^X_x \tilde{\phi}_x}{v^2} - \frac{\bar{\phi}^X_x \tilde{v}_x}{v\bar{v}^X} + \frac{\bar{\phi}^X_x \bar{v}^X_x \tilde{v}}{v^2 \bar{v}^X} - \left( \frac{\bar{\phi}^X_x}{\bar{v}^X} \right)_x \frac{\tilde{v}}{v} \right) \, dx,
\end{split}
\]
and
\[
\begin{split}
\mathcal{V}^{(1)}_8 & := - \int_\mathbb{R} \bigg( \frac{\tilde{v}_x \tilde{\phi}_{xxx}}{\bar{v}^X} (1 - e^{\tilde{\phi}}) - \frac{\tilde{v}\tilde{\phi}_{xxx}}{\bar{v}^X} (1-e^{\tilde{\phi}})_x  - \frac{\bar{v}^X_x \tilde{\phi}_{xxx}}{\bar{v}^X} (1-e^{\tilde{\phi}}) - \tilde{\phi}_{xxx} (1+\tilde{\phi}-e^{\tilde{\phi}})_x \bigg) \, dx.
\end{split}
\]

We estimate the terms $\mathcal{V}^{(1)}_j$ for $j=1,\dots,8$. First, we directly have
\[
\begin{split}
\lvert \mathcal{V}^{(1)}_1 \rvert & \leq C \int_\mathbb{R} \tilde{u}_x^2 \, dx.
\end{split}
\]
For the term $\mathcal{V}^{(1)}_2$, we have by integration by parts
\[
\begin{split}
\mathcal{V}^{(1)}_2 & = - \dot{X}(t) \int_\mathbb{R} \left( \bar{u}^X_x v \tilde{v}_x - \bar{v}^X_{xx} \tilde{v}_x - (\bar{v}^X_x)^2 \tilde{u} - \bar{v}^X_x v \tilde{u}_x \right) \, dx + \dot{X}(t) \int_\mathbb{R} \bar{v}^X_x \tilde{v}_x \tilde{u} \, dx.
\end{split}
\]
The first term on the right-hand side is bounded by using $ \lvert \bar{u}^X_x \rvert, \lvert \bar{v}^X_{xx} \rvert \leq C \lvert \bar{v}^X_x \rvert \leq C \delta_S^2 $ as follows:
\[
\begin{split}
& \bigg\lvert - \dot{X}(t) \int_\mathbb{R} \bar{u}^X_x v \tilde{v}_x - \bar{v}^X_{xx} \tilde{v}_x - (\bar{v}^X_x)^2 \tilde{u} - \bar{v}^X_x v \tilde{u}_x \, dx \bigg\rvert \\
& \quad \leq C \lvert \dot{X} \rvert \int_\mathbb{R} \bar{v}^X_x \left( \lvert \tilde{v}_x \rvert + \lvert \tilde{u}_x \rvert + \lvert \bar{v}^X_x \rvert \lvert \tilde{u} \rvert \right) \, dx \\
& \quad \leq C \lvert \dot{X} \rvert \sqrt{ \int_\mathbb{R} \bar{v}^X_x \, dx }  \left( \sqrt{ \int_\mathbb{R} \bar{v}^X_x \tilde{v}_x^2 } + \sqrt{ \int_\mathbb{R} \bar{v}^X_x \tilde{u}_x^2 \, dx} + \sqrt{\int_\mathbb{R} \lvert \bar{v}^X_x \rvert^3 \tilde{u}^2 \, dx}  \right) \\
& \quad \leq C \delta_S^2 \lvert \dot{X} \rvert^2 + C \delta_S \int_\mathbb{R} \left( \tilde{v}_x^2 + \tilde{u}_x^2 + \bar{v}^X_x \tilde{u}^2 \right) \, dx, 
\end{split}
\]
where the H\"older inequality and Young's inequality are applied in the second and third inequalities, respectively. The second term is estimated by using $\lvert \dot{X} \rvert \leq C \lVert \tilde{u} \rVert_{L^\infty}$ as
\[
\begin{split}
\bigg \lvert \dot{X}(t) \int_\mathbb{R} \bar{v}^X_x \tilde{v}_x \tilde{u} \, dx \bigg\rvert  \leq C \lvert \dot{X} \rvert \int_\mathbb{R} \bar{v}^X_x \lvert \tilde{v}_x \rvert \lvert \tilde{u} \rvert \, dx  \leq C \lVert \tilde{u} \rVert_{L^\infty} \int_\mathbb{R} \left( \bar{v}^X_x \tilde{u}^2 +\tilde{v}_x^2 \right) \, dx.
\end{split}
\]
Thus, we obtain
\[
\lvert \mathcal{V}^{(1)}_2 \rvert \leq C \left( \delta_S + \lVert \tilde{u} \rVert_{L^\infty} \right) \left( \delta_S \lvert \dot{X} \rvert^2 + \int_\mathbb{R} \left( \bar{v}^X_x \tilde{u}^2  + \tilde{v}_x^2 + \tilde{u}_x^2 \right) \, dx \right).
\]

Note that $\mathcal{V}^{(1)}_3$ and $\mathcal{V}^{(1)}_4$ are nonlinear terms, as is $\mathcal{P}_6$ described in Remark \ref{RemP}, involving cubic and higher-order products of perturbations. First, the term $\mathcal{V}^{(1)}_3$ is estimated by using the bound $\lvert \bar{\phi}^X_{x} \rvert \leq C \lvert \bar{v}^X_x \rvert$ and Young's inequality as
\[
\begin{split}
\lvert \mathcal{V}^{(1)}_3 \rvert & \leq C \int_\mathbb{R} \lvert \tilde{v} \rvert \left( \lvert \tilde{v}_x \rvert^2 + \lvert \tilde{v}_x \rvert \lvert \tilde{\phi}_{xx} \rvert + \lvert \tilde{v}_x \rvert^2 \lvert \tilde{\phi}_{xx} \rvert + \lvert \tilde{v}_{xx} \rvert \lvert \tilde{\phi}_{xx} \rvert + \lvert \bar{v}^X_x \rvert \lvert \tilde{v} \rvert \lvert \tilde{v}_x \rvert + \lvert \bar{v}^X_x \rvert \lvert \tilde{v} \rvert^2 \lvert \tilde{v}_x \rvert \right) \, dx \\
& \quad + C \int_\mathbb{R} \lvert \tilde{v} \rvert \left( \lvert \bar{v}^X_x \rvert \lvert \tilde{v} \rvert \lvert \tilde{v}_{xx} \rvert + \lvert \bar{v}^X_x \rvert \lvert \tilde{v} \rvert^2 \lvert \tilde{v}_{xx} \rvert + \lvert \tilde{v} \rvert \lvert \tilde{v}_x \rvert^2 + \lvert \tilde{v} \rvert^2 \lvert \tilde{v}_x \rvert^2 \right) \, dx \\
& \quad + C \int_\mathbb{R} \lvert \tilde{v}_x \rvert \left( \lvert \tilde{v}_x \rvert \lvert \tilde{\phi}_{xx} \rvert + \lvert \tilde{v}_x \rvert \lvert \tilde{\phi}_x \rvert + \lvert \tilde{v}_{xx} \rvert \lvert \tilde{\phi}_x \rvert + \lvert \tilde{v}_x \rvert^2 + \lvert \tilde{v}_x \rvert \lvert \tilde{u}_x \rvert \right) \, dx \\
& \quad + C \int_\mathbb{R} \lvert \tilde{\phi}_x \rvert \left( \lvert \tilde{v}_x \rvert \lvert \tilde{\phi}_x \rvert + \lvert \tilde{v}_{xx} \rvert \lvert \tilde{\phi}_x \rvert + \lvert \tilde{\phi}_x \rvert^2 \lvert \tilde{v}_x \rvert + \lvert \tilde{v}_x \rvert^3 \right) \, dx \\
& \leq C \left( \lVert \tilde{v} \rVert_{W^{1,\infty}} + \lVert \tilde{\phi}_x \rVert_{L^\infty} \right) \int_\mathbb{R} \left( \bar{v}^X_x \tilde{v}^2 + \tilde{v}_x^2 + \tilde{v}_{xx}^2 + \tilde{u}_x^2 + \tilde{\phi}_x^2 + \tilde{\phi}_{xx}^2 \right) \, dx. 
\end{split}
\]
Similarly, we obtain the estimates of $\mathcal{V}^{(1)}_4$ by applying Young's inequality:
\[
\begin{split}
\lvert \mathcal{V}^{(1)}_4 \rvert & \leq C \int_\mathbb{R} \lvert \tilde{v} \rvert \lvert \tilde{\phi}_{xxx} \rvert \left( \lvert \tilde{\phi}_{xxx} \rvert + \lvert \tilde{\phi}_{xx} \rvert + \lvert \tilde{\phi}_x \rvert + \lvert \tilde{v}_{xx} \rvert + \lvert \tilde{v}_x \rvert \right) \, dx \\
& \quad + C \int_\mathbb{R} \left( \lvert \tilde{v} \rvert \lvert \tilde{v}_x \rvert + \lvert \tilde{v}_x \rvert \right) \lvert \tilde{\phi}_{xxx} \rvert \left( \lvert \tilde{\phi}_{xx} \rvert + \lvert \tilde{\phi}_x \rvert + \lvert \tilde{v}_x \rvert \right) \, dx \\
& \quad + C \int_\mathbb{R} \lvert \tilde{\phi}_x \rvert \lvert \tilde{\phi}_{xxx} \rvert \left( \lvert \tilde{v}_{xx} \rvert + \lvert \tilde{v}_x \rvert^2 + \lvert \tilde{v} \rvert \lvert \tilde{v}_{xx} \rvert + \lvert \tilde{v} \rvert \lvert \tilde{\phi}_x \rvert^2 \right) \, dx \\
& \leq C \left( \lVert \tilde{v} \rVert_{W^{1,\infty}} + \lVert \tilde{\phi}_x \rVert_{L^\infty} \right) \int_\mathbb{R} \left( \tilde{\phi}_{xxx}^2 + \tilde{\phi}_{xx}^2 + \tilde{\phi}_x^2 + \tilde{v}_{xx}^2 + \tilde{v}_x^2 \right) \, dx.
\end{split}
\]
For the terms $\mathcal{V}^{(1)}_5$, $\mathcal{V}^{(1)}_6$ and $\mathcal{V}^{(1)}_7$, we use the bounds \eqref{shbounds} with $ \lvert \bar{v}^X_x \rvert \leq C \delta_S^2$ to obtain
\[
\begin{split}
\lvert \mathcal{V}^{(1)}_5 \rvert & \leq C \int_\mathbb{R} \lvert \bar{v}^X_x \rvert \lvert \tilde{v} \rvert \left( \lvert \tilde{v}_x \lvert + \lvert \tilde{v}_{xx} \rvert + \lvert \tilde{\phi}_{xxx} \rvert \right) \, dx + C \int_\mathbb{R} \lvert \bar{v}^X_x \rvert \lvert \tilde{u} \rvert \left( \lvert \tilde{v}_x \rvert + \lvert \tilde{u}_x \rvert \right) \, dx \\
& \leq C \left( \int_\mathbb{R} \lvert \bar{v}^X_x \rvert^{1/2} \left( \lvert \tilde{v}_x \rvert^2 + \lvert \tilde{v}_{xx} \rvert^2 + \lvert \tilde{\phi}_{xxx} \rvert^2 + \lvert \tilde{u}_x \rvert^2 \right) \, dx + \int_\mathbb{R} \lvert \bar{v}^X_x \rvert^{3/2} \left( \lvert \tilde{v} \rvert^2 + \lvert \tilde{u} \rvert^2 \right) \, dx \right) \\
& \leq C \delta_S \int_\mathbb{R} \left( \bar{v}^X_x \tilde{v}^2 + \bar{v}^X_x \tilde{u}^2 +\tilde{v}_x^2 + \tilde{v}_{xx}^2 + \tilde{u}_x^2 + \tilde{\phi}_{xxx}^2 \right) \, dx,
\end{split}
\]

\[
\begin{split}
\lvert \mathcal{V}^{(1)}_6 \rvert & \leq C \int_\mathbb{R} \lvert \bar{v}^X_x \rvert \lvert \tilde{v}_x \rvert \left( \lvert \tilde{\phi}_{xx} \rvert + \lvert \tilde{\phi}_x \rvert + \lvert \bar{v}^X_x \rvert \lvert \tilde{v} \rvert + \lvert \tilde{v}_x \rvert + \lvert \tilde{u}_x \rvert \right) \, dx \\
& \quad + C \int_\mathbb{R} \lvert \bar{v}^X_x \rvert \lvert \tilde{v}_{xx} \rvert \left( \lvert \tilde{\phi}_x \rvert + \lvert \tilde{v}_x \rvert + \lvert \bar{v}^X_x \rvert \lvert \tilde{v} \rvert \right) \, dx \\
& \leq C \delta_S^2 \int_\mathbb{R} \left( \bar{v}^X_x \tilde{v}^2 + \tilde{\phi}_{xx}^2 + \tilde{\phi}_x^2 + \tilde{v}_x^2 + \tilde{u}_x^2 \right) \, dx,
\end{split}
\]
and
\[
\begin{split}
\lvert \mathcal{V}^{(1)}_7 \rvert & \leq C \int_\mathbb{R} \lvert \bar{v}^X_x \rvert \lvert \tilde{\phi}_{xxx} \rvert \left( \lvert \tilde{\phi}_{xx} \rvert + \lvert \tilde{\phi}_x \rvert + \lvert \tilde{v}_{xx} \rvert + \lvert \tilde{v}_x \rvert + \lvert \bar{v}^X_x \rvert \lvert \tilde{v} \rvert  \right) \, dx \\
& \leq C \delta_S^2 \int_\mathbb{R} \left( \bar{v}^X_x \tilde{v}^2 +\tilde{\phi}_{xxx}^2 + \tilde{\phi}_x^2 + \tilde{v}_{xx}^2 + \tilde{v}_x^2 \right) \, dx.
\end{split}
\]
Lastly, the term $\mathcal{V}^{(1)}_8$ is bounded by using the expansion $e^{\tilde{\phi}} = 1 + \mathcal{O}(\tilde{\phi})$ around $\tilde{\phi}=0$ as follows
\[
\begin{split}
\lvert \mathcal{V}^{(1)}_8 \rvert & \leq C \int_\mathbb{R} \lvert \bar{v}^X_x \rvert  \lvert \tilde{\phi}_{xxx} \lvert \tilde{\phi} \rvert \left( 1 + \mathcal{O}(\lvert \tilde{\phi} \rvert) \right) \, dx + C \int_\mathbb{R} \lvert \tilde{v} \rvert \lvert \tilde{\phi}_{xxx} \rvert \lvert \tilde{\phi}_x \rvert \left( 1 + \mathcal{O}(\lvert \tilde{\phi} \rvert ) \right)  \, dx \\
& \quad + C \int_\mathbb{R} \lvert \tilde{\phi}_{xxx} \rvert \left( \lvert \tilde{v}_x \rvert + \lvert \tilde{\phi}_x \rvert \right) \mathcal{O}(\lvert \tilde{\phi} \rvert) \, dx \\
& \leq C \left( 1 + \lVert \tilde{\phi} \rVert_{L^\infty} \right) \left( \int_\mathbb{R} \lvert \bar{v}^X_x \rvert^{1/2} \tilde{\phi}_{xxx}^2 \, dx + \int_\mathbb{R} \lvert \bar{v}^X_x \rvert^{3/2} \tilde{\phi}^2 \, dx \right) \\
& \quad + C \left( \lVert \tilde{v} \rVert_{L^\infty} + \lVert \tilde{\phi} \rVert_{L^\infty}  \right) \int_\mathbb{R} \left( \tilde{\phi}_{xxx}^2 + \tilde{\phi}_x^2 + \tilde{v}_x^2 \right) \, dx \\
& \leq C \left( \delta_S + \lVert \tilde{v} \rVert_{L^\infty} + \lVert \tilde{\phi} \rVert_{L^\infty} \right) \int_\mathbb{R} \left(  \bar{v}^X_x \tilde{\phi}^2 + \tilde{\phi}_{xxx}^2 + \tilde{\phi}_x^2 + \tilde{v}_x^2 \right) \, dx.
\end{split}
\]

Combining \eqref{'8} with all the estimates of $\mathcal{V}^{(1)}_j$, we have
\begin{equation} \label{'9}
\begin{split}
& \frac{d}{dt} \int_\mathbb{R} \left( \frac{\tilde{v}_x^2}{2} - v \tilde{v}_x \tilde{u} \right) \, dx + \int_\mathbb{R} \bigg( \frac{2\tilde{v}_x^2}{\bar{v}^X} - \frac{2 \tilde{v}_x \tilde{\phi}_{xxx}}{\bar{v}^X} + \frac{ e^{-\bar{\phi}^X} \tilde{\phi}_{xxx}^2}{(\bar{v}^X)^2} + \tilde{\phi}_{xx}^2 \bigg) \, dx \\
& \quad \leq C \int_\mathbb{R} \tilde{u}_x^2 \, dx + C \left( \delta_S + \lVert \tilde{v} \rVert_{W^{1,\infty}} + \lVert \tilde{u} \rVert_{L^\infty} + \lVert \tilde{\phi} \rVert_{W^{1,\infty}} \right) \\
& \qquad \times \left( \int_\mathbb{R} \left( \bar{v}^X_x \tilde{v}^2 + \bar{v}^X_x \tilde{u}^2 + \bar{v}^X_x \tilde{\phi}^2 + \tilde{v}_x^2 + \tilde{v}_{xx}^2 + \tilde{\phi}_x^2 + \tilde{\phi}_{xx}^2 + \tilde{\phi}_{xxx}^2 \right) \, dx  + \delta_S \lvert \dot{X} \rvert^2 \right).
\end{split}
\end{equation}
Note that there exists a positive constant $c$ such that
\[
\bigg\lvert \frac{e^{-\bar{\phi}^X}}{\bar{v}^X} \bigg\rvert = \bigg\lvert 1 -\frac{e^{-\bar{\phi}^X}}{\bar{v}^X} \left( \frac{\bar{\phi}^X_x}{\bar{v}^X} \right)_x \bigg\rvert \geq 1 - C \delta_S^2 \geq c
\]
for sufficiently small $\delta_S < \delta_1$. Thus, we have
\[
\frac{2\tilde{v}_x^2}{\bar{v}^X} - \frac{2 \tilde{v}_x \tilde{\phi}_{xxx}}{\bar{v}^X} + \frac{ e^{-\bar{\phi}^X} \tilde{\phi}_{xxx}^2}{(\bar{v}^X)^2} \geq c \left( \tilde{v}_x^2 + \tilde{\phi}_{xxx}^2 \right)
\]
for a generic constant $c>0$, provided that $\delta_1$ is small enough. Using this lower bound of the quadratic term on the left-hand side of \eqref{'9}, we obtain after the application of the results of Lemmas \ref{Lv^2} and \ref{Lemma:phi}
\[
\begin{split}
& \frac{d}{dt} \int_\mathbb{R} \bigg( \frac{\tilde{v}_x^2}{2} - v \tilde{v}_x \tilde{u} \bigg) \, dx +  \lVert \tilde{v}_x (t,\cdot) \rVert_{L^2}^2 + \lVert \tilde{\phi}_{xx} (t,\cdot) \rVert_{H^1}^2 \\
& \quad \leq C \left( \delta_1 + \varepsilon_1 \right) \left( G_1 + G^S + \delta_S \lvert \dot{X} \rvert^2 + \lVert \tilde{v}_{xx} (t,\cdot) \rVert_{L^2}^2 \right) + C D
\end{split}
\]
for sufficiently small $\delta_1$, $\varepsilon_1$. We integrate this with respect to $t$ to have
\[
\begin{split}
& \lVert \tilde{v}_x (t,\cdot) \rVert_{L^2}^2  + \int_0^t \left( \lVert \tilde{v}_x \rVert_{L^2}^2 + \lVert \tilde{\phi}_{xx} \rVert_{H^1}^2 \right) \, d\tau \\
& \quad \leq C \lVert \tilde{v}_{0x} \rVert_{L^2}^2 + \bigg[ \int_\mathbb{R} v\tilde{v}_x \tilde{u} \, dx \bigg]^{t=t}_{t=0} + C \left( \delta_1 + \varepsilon_1 \right) \int_0^t \left( G_1 + G^S + \delta_S \lvert \dot{X} \rvert^2 + \lVert \tilde{v}_{xx}\rVert_{L^2}^2 \right) \, d\tau \\
& \qquad + C \int_0^t D \, d\tau.
\end{split}
\]
Then we obtain by applying Young's inequality
\[
\begin{split}
& \lVert \tilde{v}_x (t,\cdot) \rVert_{L^2}^2 + \int_0^t \left( \lVert \tilde{v}_x \rVert_{L^2}^2 + \lVert \tilde{\phi}_{xx} \rVert_{H^1}^2 \right) \, d\tau \\
& \quad \leq C \lVert ( \tilde{v}_{0x},\tilde{u}_0) \rVert_{L^2}^2 + C \left( \delta_1 + \varepsilon_1 \right) \int_0^t \left( G_1 + G^S + \delta_S \lvert \dot{X} \rvert^2 + \lVert \tilde{v}_{xx} \rVert_{L^2}^2 \right) \, d\tau \\
& \qquad + C \left( \lVert \tilde{u} (t,\cdot) \rVert_{L^2}^2 + \int_0^t D \, d\tau \right).
\end{split}
\]
Finally, we use \eqref{RE} to obtain the desired estimate \eqref{derivv1}.

\end{proof}

\begin{lemma} \label{Lemma:u1}
Under the assumptions in Proposition \ref{Apriori}, there exists a constant $C>0$ such that
\begin{equation} \label{derivu1}
\begin{split}
& \lVert \tilde{u}_x (t,\cdot) \rVert_{L^2}^2 + \int_0^t \lVert \tilde{u}_{xx} \rVert_{L^2}^2 \, d\tau \\
& \quad \leq C \lVert \tilde{u}_{0x} \rVert_{L^2}^2 +  C \left( \delta_1 + \varepsilon_1 \right) \int_0^t \left( \delta_S \lvert \dot{X} \rvert^2 + G_1 + G^S + D \right) \, d\tau \\
& \qquad + \eta \int_0^t \lVert \tilde{u}_{x} \rVert_{H^1}^2 d \tau + \frac{C}{\eta} \int_0^t \left( \lVert \tilde{v}_x \rVert_{L^2}^2 + \lVert \tilde{\phi}_{xx} \rVert_{L^2}^2 \right) \, d\tau
\end{split}
\end{equation}
for all $t \in [0,T]$, where $0 < \eta < 1$ is an arbitrary constant and $G_1$, $G^S$, and $D$ are as defined in Lemma~\ref{RE}.
\end{lemma}

\begin{proof}
Differentiating \eqref{1b'} with respect to $x$ and multiplying the resultant equation by $ \tilde{u}_x$, we have after rearrangement
\begin{equation} \label{ux11}
\begin{split}
\left( \frac{\tilde{u}_x^2}{2} \right)_t + \frac{\tilde{u}_{xx}^2}{v} & = (\cdots)_x +  \tilde{u}_x \left( \frac{\tilde{v}}{v\bar{v}^X} \right)_{xx} + \frac{\tilde{v}_x \tilde{u}_{xx}\tilde{u}_x}{v^2} + \frac{\bar{v}^X_x \tilde{u}_{xx}\tilde{u}_x}{v^2} + \tilde{u}_{xx} \left( \frac{\bar{u}^X_x\tilde{v}}{v\bar{v}^X} \right)_x \\
& \quad - \tilde{u}_x \left( \frac{\tilde{\phi}_x}{v} - \frac{\bar{\phi}^X_x \tilde{v}}{v\bar{v}^X} \right)_x + \dot{X}(t) \bar{u}^X_{xx}\tilde{u}_x,
\end{split}
\end{equation}
where we have used
\[
\begin{split}
- \tilde{u}_x \left( \frac{u_x}{v} - \frac{\bar{u}^X_x}{\bar{v}^X} \right)_{xx} & = (\cdots)_x + \frac{ \tilde{u}_{xx}^2}{v} - \frac{\tilde{v}_x \tilde{u}_{xx}\tilde{u}_x}{v^2} - \frac{\bar{v}^X_x \tilde{u}_{xx}\tilde{u}_x}{v^2} - \tilde{u}_{xx} \left( \frac{\bar{u}^X_x\tilde{v}}{v\bar{v}^X} \right)_x.
\end{split}
\]
Integrating \eqref{ux11} over $\mathbb{R}$, we have by integration by parts
\[
\begin{split}
& \frac{d}{dt} \int_\mathbb{R} \frac{\tilde{u}_x^2}{2} \, dx + \int_\mathbb{R} \frac{ \tilde{u}_{xx}^2}{v} \, dx = \sum_{j=1}^4 \mathcal{U}^{(1)}_j,
\end{split}
\]
where
\[
\begin{split}
\mathcal{U}^{(1)}_1 & := - \int_\mathbb{R} \bigg( \frac{\tilde{u}_{xx}\tilde{v}_x}{v\bar{v}^X} + \frac{\tilde{u}_x \tilde{\phi}_{xx}}{v} \bigg) \, dx, \quad \mathcal{U}^{(1)}_2 := \dot{X}(t) \int_\mathbb{R} \bar{u}^X_{xx}\tilde{u}_x  \, dx, \\
\mathcal{U}^{(1)}_3 & := \int_\mathbb{R} \tilde{u}_{xx} \left( \frac{\tilde{v}_x\tilde{v}}{v^2\bar{v}^X} + \frac{\tilde{v}_x\tilde{u}_x}{v^2} - \frac{\bar{u}^X_x\tilde{v}_x\tilde{v}}{v^2\bar{v}^X} \right) \, dx + \int_\mathbb{R} \tilde{u}_x \left( \frac{\tilde{v}_x \tilde{\phi}_x}{v^2} - \frac{\bar{\phi}^X_x \tilde{v}_x\tilde{v}}{v^2\bar{v}^X} \right) \, dx, \\
\mathcal{U}^{(1)}_4 & := \int_\mathbb{R} \tilde{u}_{xx} \left( \frac{\bar{v}^X_x \tilde{v}}{v^2\bar{v}^X} + \frac{\bar{v}^X_x \tilde{v}}{v (\bar{v}^X)^2} + \frac{\bar{v}^X_x \tilde{u}_x}{v^2} + \frac{\bar{u}^X_x \tilde{v}_x}{v\bar{v}^X} - \frac{\bar{v}^X_x \bar{u}^X_x \tilde{v}}{v^2\bar{v}^X} + \left( \frac{\bar{u}^X_x}{\bar{v}^X} \right)_x \frac{\tilde{v}}{v} \right) \, dx \\
& \quad + \int_\mathbb{R} \tilde{u}_x \left( \frac{\bar{v}^X_x \tilde{\phi}_x}{v^2} + \frac{\bar{\phi}^X_x \tilde{v}_x}{v\bar{v}^X} - \frac{\bar{v}^X_x \bar{\phi}^X_x \tilde{v}}{v^2 \bar{v}^X} + \left( \frac{\bar{\phi}^X_x}{\bar{v}^X} \right)_x \frac{\tilde{v}}{v} \right) \, dx.
\end{split}
\]

We estimate the terms $\mathcal{U}^{(1)}_j$ for $j=1,\dots,4$. First, by applying Young's inequality, we obtain
\[
\begin{split}
\lvert \mathcal{U}^{(1)}_1 \rvert & \leq C \int_\mathbb{R} \lvert \tilde{u}_{xx} \rvert \lvert \tilde{v}_x \rvert + \lvert \tilde{u}_x \rvert \lvert \tilde{\phi}_{xx} \rvert \, dx \\
& \leq \eta \int_\mathbb{R} \left( \tilde{u}_{xx}^2 + \tilde{u}_x^2 \right) \, dx + \frac{C}{\eta} \int_\mathbb{R} \left( \tilde{v}_x^2 + \tilde{\phi}_{xx}^2 \right) \, dx
\end{split}
\]
for any constant $0 < \eta < 1$. For the term $\mathcal{U}^{(1)}_2$, we use $\lvert \bar{u}^X_{xx} \rvert \leq C \lvert \bar{v}^X_x \rvert$ and apply the H\"older inequality and Young's inequality to obtain
\[
\begin{split}
\lvert \mathcal{U}^{(1)}_2 \rvert & \leq C \lvert \dot{X} \rvert \sqrt{ \int_\mathbb{R} \bar{v}^X_x \, dx} \sqrt{ \int_\mathbb{R} \bar{v}^X_x \tilde{u}_x^2 \, dx} \\
& \leq C \delta_S^2 \lvert \dot{X} \rvert^2 + \frac{C}{\delta_S^2} \left( \int_\mathbb{R} \bar{v}^X_x \, dx \right) \left( \int_\mathbb{R} \bar{v}^X_x \tilde{u}_x^2 \, dx \right) \\
& \leq C \delta_S \left( \delta_S \lvert \dot{X} \rvert^2 + \int_\mathbb{R} \tilde{u}_x^2 \, dx \right). 
\end{split}
\]
The nonlinear term $\mathcal{U}^{(1)}_3$ satisfies the bounds
\[
\begin{split}
\lvert \mathcal{U}^{(1)}_3 \rvert & \leq C \int_\mathbb{R} \lvert \tilde{v} \rvert \left( \lvert \tilde{v}_x \rvert \lvert \tilde{u}_{xx} \rvert + \lvert \tilde{v}_x \rvert \lvert \tilde{u}_x \rvert \right) \, dx + C \int_\mathbb{R} \lvert \tilde{v}_x \rvert \left( \lvert \tilde{u}_x \rvert \lvert \tilde{u}_{xx} \rvert + \lvert \tilde{u}_x \rvert \lvert \tilde{\phi}_x \rvert \right) \, dx \\
& \leq C \lVert \tilde{v} \rVert_{W^{1,\infty}} \int_\mathbb{R} \left( \tilde{v}_x^2 + \tilde{u}_x^2 + \tilde{u}_{xx}^2 + \tilde{\phi}_x^2 \right) \, dx.
\end{split}
\]
The term $\mathcal{U}^{(1)}_4$ is estimated by using the bounds $\lvert \bar{u}^X_{xx} \rvert, \lvert \bar{\phi}^X_{xx} \rvert \leq C \bar{v}^X_x$ and $\lvert \bar{u}^X_x \rvert \sim \lvert \bar{v}^X_x \rvert \leq C \delta_S^2$ as
\[
\begin{split}
\lvert \mathcal{U}^{(1)}_4 \rvert & \leq C \int_\mathbb{R} \lvert \bar{v}^X_x \rvert \left( \lvert \tilde{u}_{xx} \rvert \lvert \tilde{v} \rvert + \lvert \tilde{u}_x \rvert \lvert \tilde{v} \rvert \right) \, dx + C \int_\mathbb{R} \lvert \bar{v}^X_x \rvert \lvert \tilde{u}_{xx} \rvert \left( \lvert \tilde{v}_x \rvert + \lvert \tilde{u}_x \rvert \right) \, dx \\
& \quad + C \int_\mathbb{R} \lvert \bar{v}^X_x \rvert \lvert \tilde{u}_x \rvert \left( \lvert \tilde{v}_x \rvert + \lvert \tilde{\phi}_x \rvert \right) \, dx \\
& \leq C \int_\mathbb{R} \lvert \bar{v}^X_x \rvert^{1/2} \left( \tilde{u}_{xx}^2 + \tilde{u}_x^2 \right) \, dx + C \int_\mathbb{R} \lvert \bar{v}^X_x \rvert^{3/2} \tilde{v}^2 \, dx  + C \int_\mathbb{R} \lvert \bar{v}^X_x \rvert \left( \tilde{u}_{xx}^2 + \tilde{u}_x^2 + \tilde{v}_x^2 + \tilde{\phi}_x^2 \right) \, dx \\
& \leq C \delta_S \int_\mathbb{R} \left( \bar{v}^X_x \tilde{v}^2 + \tilde{u}_{xx}^2 + \tilde{u}_x^2 + \tilde{v}_x^2 + \tilde{\phi}_x^2 \right) \, dx. 
\end{split}
\]

Collecting all the estimates, together with $\delta_S \leq C \delta_1$ and $\lVert \tilde{v} \rVert_{W^{1,\infty}} \leq C \varepsilon_1$, and integrating the resultant inequality with respect to $t$, we obtain \eqref{derivu1} after the application of the results of Lemmas~\ref{Lv^2} and \ref{Lemma:phi}.

\end{proof}

\subsection{\texorpdfstring{$H^2$-estimates}{H2-estimates}}

\begin{lemma} \label{Lemma:v2}
Under the assumptions in Proposition \ref{Apriori}, there exists a constant $C>0$ such that
\begin{equation} \label{derivv2}
\begin{split}
& \lVert \tilde{v}_{xx} (t,\cdot) \rVert_{L^2}^2 + \int_0^t \lVert \tilde{v}_{xx} \rVert_{L^2}^2 \, d\tau \\
& \quad \leq C \left( \lVert \tilde{v}_{0xx} \rVert_{L^2}^2 + \lVert \tilde{u}_{0x} \rVert_{L^2}^2 \right)  + C \lVert \tilde{u}_x (t,\cdot) \rVert_{L^2}^2 + C \int_0^t \left( \lVert \tilde{u}_{xx} \rVert_{L^2}^2 + \lVert \tilde{\phi}_{xx} \rVert_{L^2}^2 \right) \, d\tau \\
& \qquad + C \left( \delta_1 + \varepsilon_1 \right) \int_0^t \left( \delta_S \lvert \dot{X} \rvert^2 + G_1 + G^S + D + \lVert \tilde{v}_x \rVert_{L^2}^2 \right) \, d\tau
\end{split}
\end{equation}
for all $t \in [0,T]$, where $G_1$, $G^S$, and $D$ are as defined in Lemma~\ref{RE}.
\end{lemma}

\begin{proof}
Differentiating \eqref{1a'} twice with respect to $x$ and multiplying the resultant equation by $\tilde{v}_{xx} $, we have
\begin{equation} \label{derivv21}
\left( \frac{\tilde{v}_{xx}^2}{2} \right)_t - \tilde{v}_{xx} \tilde{u}_{xxx} = \dot{X}(t) \bar{v}^X_{xxx}\tilde{v}_{xx}.
\end{equation}
Differentiating \eqref{1b'} with respect to $x$ and multiplying the resultant equation by $ - v \tilde{v}_{xx} $, we have
\begin{equation} \label{derivv22}
\begin{split}
& - v \tilde{v}_{xx} \tilde{u}_{xt} - v \tilde{v}_{xx} \left( \frac{1}{v} - \frac{1}{\bar{v}^X} \right)_{xx}  + v \tilde{v}_{xx} \left( \frac{u_x}{v} - \frac{\bar{u}^X_x}{\bar{v}^X} \right)_{xx} \\
& \quad = v \tilde{v}_{xx} \left( \frac{\phi_x}{v} - \frac{\bar{\phi}^X_x}{\bar{v}^X} \right)_x - \dot{X}(t) v \bar{u}^X_{xx} \tilde{v}_{xx}.
\end{split}
\end{equation}
Here, the first term on the left-hand side can be written as
\[
\begin{split}
- v \tilde{v}_{xx} \tilde{u}_{xt} & = \left( - v \tilde{v}_{xx}\tilde{u}_x \right)_t + v_t \tilde{v}_{xx}\tilde{u}_x + v \tilde{v}_{xxt} \tilde{u}_x \\
& = \left( - v \tilde{v}_{xx}\tilde{u}_x \right)_t + u_x \tilde{v}_{xx}\tilde{u}_x + v \tilde{u}_x \tilde{u}_{xxx} + \dot{X}(t) v \bar{v}^X_{xxx} \tilde{u}_x \\
& = \left( - v \tilde{v}_{xx}\tilde{u}_x \right)_t + (\cdots)_x  + u_x \tilde{v}_{xx}\tilde{u}_x - v_x \tilde{u}_x \tilde{u}_{xx} - v \tilde{u}_{xx}^2 + \dot{X}(t) v \bar{v}^X_{xxx} \tilde{u}_x, 
\end{split}
\]
where we used \eqref{1a'} in the second equality. The second and third terms are expanded as
\[
\begin{split}
- v \tilde{v}_{xx} \left( \frac{1}{v} - \frac{1}{\bar{v}^X} \right)_{xx} & = \frac{\tilde{v}_{xx}^2}{\bar{v}^X} - \tilde{v}_{xx} \left( \frac{\tilde{v}_x^2}{v \bar{v}^X} + \frac{\bar{v}^X_x \tilde{v}_x}{v\bar{v}^X} + \frac{\bar{v}^X_x \tilde{v}_x}{(\bar{v}^X)^2} \right) \\
& \quad - v \tilde{v}_{xx} \left( \frac{\tilde{v}_x\tilde{v}}{v^2\bar{v}^X} + \frac{\bar{v}^X_x \tilde{v}}{v^2 \bar{v}^X} + \frac{\bar{v}^X_x \tilde{v}}{v(\bar{v}^X)^2} \right)_x
\end{split}
\]
and
\[
\begin{split}
v \tilde{v}_{xx} \left( \frac{u_x}{v} - \frac{\bar{u}^X_x}{\bar{v}^X} \right)_{xx} & = \tilde{v}_{xx}\tilde{u}_{xxx} - \frac{\tilde{v}_{xx}\tilde{v}_x \tilde{u}_{xx}}{v} - \frac{\bar{v}^X_x \tilde{v}_{xx}\tilde{u}_{xx}}{v} \\
& \quad - v \tilde{v}_{xx} \left( \frac{\tilde{v}_x \tilde{u}_x}{v^2} + \frac{\bar{v}^X_x \tilde{u}_x}{v^2} \right)_x - v\tilde{v}_{xx} \left(  \frac{\bar{u}^X_x \tilde{v}}{v\bar{v}^X} \right)_{xx},
\end{split}
\]
respectively. Thus, by summing \eqref{derivv21}, \eqref{derivv22} and integrating the resultant equation with respect to $x$, we have
\begin{equation} \label{''v2}
\begin{split}
\frac{d}{dt} \int_\mathbb{R} \bigg( \frac{\tilde{v}_{xx}^2}{2} - v\tilde{v}_{xx}\tilde{u}_x \bigg) \, dx + \int_\mathbb{R} \frac{\tilde{v}_{xx}^2}{\bar{v}^X} \, dx = \sum_{j=1}^4 \mathcal{V}^{(2)}_j,
\end{split}
\end{equation}
where
\[
\begin{split}
\mathcal{V}^{(2)}_1 & := \int_\mathbb{R} \left( v\tilde{u}_{xx}^2 + \tilde{v}_{xx}\tilde{\phi}_{xx} \right) \, dx, \quad \mathcal{V}^{(2)}_2 := \dot{X}(t) \int_\mathbb{R} \left( \bar{v}^X_{xxx} \tilde{v}_{xx} - v\bar{u}^X_{xx} \tilde{v}_{xx} - v \bar{v}^X_{xxx} \tilde{u}_x \right) \, dx, \\
\mathcal{V}^{(2)}_3 & := \int_\mathbb{R} \left( \tilde{v}_x\tilde{u}_x\tilde{u}_{xx} - \tilde{u}_x^2 \tilde{v}_{xx} \right) \, dx + \int_\mathbb{R} \tilde{v}_{xx} \left( \frac{2\tilde{v}_x^2}{v\bar{v}^X} + \frac{\tilde{v}_{xx}\tilde{v}}{v\bar{v}^X} - \frac{2\tilde{v}_x^2 \tilde{v}}{v^2 \bar{v}^X} - \frac{4\bar{v}^X_x\tilde{v}_x\tilde{v}}{v^2 \bar{v}^X}- \frac{2\bar{v}^X_x\tilde{v}_x\tilde{v}}{v(\bar{v}^X)^2} \right) \, dx \\
& \quad + \int_\mathbb{R} \tilde{v}_{xx} \left( \frac{2\tilde{v}_x\tilde{u}_{xx}}{v} + \frac{\tilde{v}_{xx}\tilde{u}_x}{v} - \frac{2 \tilde{v}_x^2 \tilde{u}_x}{v^2} - \frac{4\bar{v}^X_x \tilde{v}_x \tilde{u}_x}{v^2} - \frac{\tilde{v}_x \tilde{\phi}_x}{v} + \frac{\bar{\phi}^X_x\tilde{v}_x \tilde{v}}{v\bar{v}^X} \right) \, dx \\
& \quad + \int_\mathbb{R} \tilde{v}_{xx} \left( - \frac{2\bar{u}^X_x\tilde{v}_x^2}{v\bar{v}^X} - \frac{\bar{u}^X_x\tilde{v}_{xx}\tilde{v}}{v\bar{v}^X} + \frac{2\bar{u}^X_x\tilde{v}_x^2 \tilde{v}}{v^2 \bar{v}^X} + \frac{4\bar{v}^X_x\bar{u}^X_x\tilde{v}_x \tilde{v}}{v^2 \bar{v}^X} - \left( \frac{2 \bar{u}^X_x}{\bar{v}^X} \right)_x \frac{\tilde{v}_x \tilde{v}}{v} \right) \, dx
\end{split}
\]
and
\[
\begin{split}
\mathcal{V}^{(2)}_4 & := \int_\mathbb{R} \left( \bar{v}^X_x \tilde{u}_x \tilde{u}_{xx} - \bar{u}^X_x \tilde{v}_{xx}\tilde{u}_x \right) \, dx + \int_\mathbb{R} \tilde{v}_{xx} \left( \frac{2\bar{v}^X_x \tilde{u}_{xx}}{v} - \frac{2(\bar{v}^X_x)^2 \tilde{u}_x}{v^2} + \frac{\bar{v}^X_{xx} \tilde{u}_x}{v} \right) \, dx \\
& \quad + \int_\mathbb{R} \tilde{v}_{xx} \left( \frac{2\bar{v}^X_x \tilde{v}_x}{v\bar{v}^X} + \frac{2\bar{v}^X_x\tilde{v}_x}{(\bar{v}^X)^2} - \frac{2(\bar{v}^X_x)^2 \tilde{v}}{v^2 \bar{v}^X} + \left(\frac{\bar{v}^X_x}{\bar{v}^X} \right)_x \frac{\tilde{v}}{v} - \frac{(\bar{v}^X_x)^2\tilde{v}}{v (\bar{v}^X)^2} + \left( \frac{\bar{v}^X_x}{(\bar{v}^X)^2} \right)_x \tilde{v} \right) \, dx \\
& \quad + \int_\mathbb{R} \tilde{v}_{xx} \left( \frac{\bar{u}^X_x \tilde{v}_{xx}}{\bar{v}^X} - \frac{2\bar{v}^X_x\bar{u}^X_x\tilde{v}_x}{v\bar{v}^X} + \left( \frac{2\bar{u}^X_x}{\bar{v}^X} \right)_x \tilde{v}_x + \frac{2(\bar{v}^X_x)^2 \bar{u}^X_x \tilde{v}}{v^2\bar{v}^X} - \left( \frac{\bar{v}^X_x\bar{u}^X_x}{\bar{v}^X} \right)_x \frac{\tilde{v}}{v} \right) \, dx \\
& \quad + \int_\mathbb{R} \tilde{v}_{xx} \left( -\left( \frac{\bar{u}^X_x}{\bar{v}^X} \right)_x \frac{\bar{v}^X_x\tilde{v}}{v} + \left( \frac{\bar{u}^X_x}{\bar{v}^X} \right)_{xx} \tilde{v} - \frac{\bar{v}^X_x \tilde{\phi}_x}{v} - \frac{\bar{\phi}^X_x \tilde{v}_x}{\bar{v}^X} + \frac{\bar{v}^X_x \bar{\phi}^X_x \tilde{v}}{v\bar{v}^X} - \left( \frac{\bar{\phi}^X_x}{\bar{v}^X} \right)_x \tilde{v} \right) \, dx.
\end{split}
\]

We estimate the terms $\mathcal{V}^{(2)}_j$ for $j=1,\dots,4$. By Young's inequality, the term $\mathcal{V}^{(2)}_1$ is bounded as follows
\[
\begin{split}
\lvert \mathcal{V}^{(2)}_1 \rvert & \leq C \int_\mathbb{R} \left( \lvert \tilde{u}_{xx} \rvert^2 +  \lvert \tilde{v}_{xx} \rvert \lvert \tilde{\phi}_{xx} \rvert \right) \, dx \leq \eta \int_\mathbb{R} \tilde{v}_{xx}^2 \, dx + \frac{C}{\eta} \int_\mathbb{R} \tilde{\phi}_{xx}^2 \, dx + C \int_\mathbb{R} \tilde{u}_{xx}^2 \, dx.
\end{split}
\]
For $\mathcal{V}^{(2)}_2$, using the bounds $\lvert \bar{v}^X_{xxx} \rvert, \lvert \bar{u}^X_{xx} \rvert \leq C \lvert \bar{v}^X_x \rvert$, we obtain
\[
\begin{split}
\lvert \mathcal{V}^{(2)}_2 \rvert & \leq C \lvert \dot{X} \rvert \sqrt{\int_\mathbb{R} \bar{v}^X_x \, dx} \left( \sqrt{ \int_\mathbb{R} \bar{v}^X_x \tilde{v}_{xx}^2 \, dx} + \sqrt{ \int_\mathbb{R} \bar{v}^X_x \tilde{u}_x^2 \, dx} \right) \\
& \leq C \delta_S^2 \lvert \dot{X} \rvert^2 + \frac{C}{\delta_S^2} \left( \int_\mathbb{R} \bar{v}^X_x \, dx \right) \left( \int_\mathbb{R} \bar{v}^X_x \tilde{v}_{xx}^2 \, dx + \int_\mathbb{R} \bar{v}^X_x \tilde{u}_x^2 \, dx \right) \\
& \leq C \delta_S \left( \delta_S \lvert \dot{X} \rvert^2 + \int_\mathbb{R} \left( \tilde{v}_{xx}^2 + \tilde{u}_x^2 \right) \, dx \right),
\end{split}
\]
where in the first and second inequality, we have applied the H\"older inequality and Young's inequality, respectively. The nonlinear term $\mathcal{V}^{(2)}_3$ is estimated as
\[
\begin{split}
\lvert \mathcal{V}^{(2)}_3 \rvert & \leq C \int_\mathbb{R} \lvert \tilde{v} \rvert \lvert \tilde{v}_{xx} \rvert \left( \lvert \tilde{v}_{xx} \rvert + \lvert \tilde{v}_x \rvert^2 + \lvert \tilde{v}_x \rvert  \right) \, dx \\
& \quad + C \int_\mathbb{R} \lvert \tilde{v}_x \rvert \lvert \tilde{v}_{xx} \rvert \left( \lvert \tilde{v}_x \rvert + \lvert \tilde{u}_{xx} \rvert + \lvert \tilde{v}_x \rvert \lvert \tilde{u}_x \rvert + \lvert \tilde{u}_x \rvert + \lvert \tilde{\phi}_x \rvert \right) \, dx \\
& \quad + C \int_\mathbb{R} \lvert \tilde{u}_x \rvert \left( \lvert \tilde{v}_x \rvert \lvert \tilde{u}_{xx} \rvert + \lvert \tilde{u}_x \rvert \lvert \tilde{v}_{xx} \rvert + \lvert \tilde{v}_{xx} \rvert^2 \right) \, dx \\
& \leq C \left( \lVert \tilde{v} \rVert_{W^{1,\infty}} + \lVert \tilde{u}_x \rVert_{L^\infty} \right) \int_\mathbb{R} \left( \tilde{v}_{xx}^2 + \tilde{v}_x^2 + \tilde{u}_{xx}^2 + \tilde{u}_x^2 + \tilde{\phi}_x^2 \right) \, dx.
\end{split}
\]
Lastly, we use the bounds $\lvert \bar{v}^X_{xx} \rvert, \lvert \bar{u}^X_{xx} \rvert \leq C \lvert \bar{v}^X_x \rvert$ and $\lvert \bar{\phi}^X_x \rvert \sim \lvert \bar{u}^X_x \rvert \sim \lvert \bar{v}^X_x \rvert \leq C \delta_S^2$ to obtain
\[
\begin{split}
\lvert \mathcal{V}^{(2)}_4 \rvert & \leq C \int_\mathbb{R} \lvert \bar{v}^X_x \rvert \lvert \tilde{v}_{xx} \rvert \lvert \tilde{v} \rvert \, dx + C \int_\mathbb{R} \lvert \bar{v}^X_x \rvert \lvert \tilde{u}_x \rvert \left( \lvert \tilde{u}_{xx} \rvert + \lvert \tilde{v}_{xx} \rvert \right) \, dx \\
& \quad + C \int_\mathbb{R} \lvert \bar{v}^X_x \rvert \lvert \tilde{v}_{xx} \rvert \left( \lvert \tilde{u}_{xx} \rvert + \lvert \tilde{u}_x \rvert + \lvert \tilde{v}_{xx} \rvert + \lvert \tilde{v}_x \rvert + \lvert \tilde{\phi}_x \rvert \right)  \, dx \\
& \leq C \int_\mathbb{R} \lvert \bar{v}^X_x \rvert^{1/2} \tilde{v}_{xx}^2 \, dx + C \int_\mathbb{R} \lvert \bar{v}^X_x \rvert^{3/2} \tilde{v}^2 \, dx  + C \int_\mathbb{R} \lvert \bar{v}^X_x \rvert \left( \tilde{v}_{xx}^2 + \tilde{v}_x^2 + \tilde{u}_{xx}^2 + \tilde{u}_x^2 + \tilde{\phi}_x^2 \right) \, dx \\
& \leq C \delta_S \int_\mathbb{R} \left( \bar{v}^X_x \tilde{v}^2 + \tilde{v}_{xx}^2 + \tilde{v}_x^2 + \tilde{u}_{xx}^2 + \tilde{u}_x^2 + \tilde{\phi}_x^2 \right) \, dx.
\end{split}
\]

Combining \eqref{''v2} with all the estimates of $\mathcal{V}^{(2)}_j$ and taking $\eta$ sufficiently small, we obtain after the application of  the results of Lemmas~\ref{Lv^2} and \ref{Lemma:phi}
\[
\begin{split}
& \frac{d}{dt} \int_\mathbb{R} \bigg( \frac{\tilde{v}_{xx}^2 }{2} - v\tilde{v}_{xx}\tilde{u}_x \bigg) \, dx + \int_\mathbb{R} \tilde{v}_{xx}^2 \, dx \\
& \quad \leq C \left( \delta_1 + \varepsilon_1 \right) \left( G_1 + G^S + D + \delta_S \lvert \dot{X} \rvert^2 +  \int_\mathbb{R} \tilde{v}_x^2 \, dx \right) + C \int_\mathbb{R} \left( \tilde{u}_{xx}^2 + \tilde{\phi}_{xx}^2 \right) \, dx
\end{split}
\]
for sufficiently small $\delta_1$ and $\varepsilon_1$. Integrating this with respect to $t$, we have
\[
\begin{split}
& \lVert \tilde{v}_{xx} (t,\cdot) \rVert_{L^2}^2 + \int_0^t \lVert \tilde{v}_{xx} \rVert_{L^2}^2  \, d\tau  \\
& \quad \leq C \lVert \tilde{v}_{0xx} \rVert_{L^2}^2 + \left[ \int_\mathbb{R} v\tilde{v}_{xx}\tilde{u}_x \, dx \right]^{t=t}_{t=0} + C \int_0^t \left( \lVert \tilde{u}_{xx} \rVert_{L^2}^2 + \lVert \tilde{\phi}_{xx} \rVert_{L^2}^2 \right) \, d\tau \\
& \qquad + C \left( \delta_1 + \varepsilon_1 \right) \int_0^t \left( G_1 + G^S + D + \delta_S \lvert \dot{X} \rvert^2 + \lVert \tilde{v}_x \rVert_{L^2}^2 \right) \, d\tau.
\end{split}
\]
Therefore, the application of Young's inequality to the second term on the right-hand side yields \eqref{derivv2}.

\end{proof}

\begin{lemma} \label{Lemma:u2}
Under the assumptions in Proposition \ref{Apriori}, there exists a constant $C>0$ such that
\begin{equation} \label{derivu2}
\begin{split}
& \lVert \tilde{u}_{xx} (t,\cdot) \rVert_{L^2}^2 + \int_0^t \lVert \tilde{u}_{xxx} \rVert_{L^2}^2 \, d\tau \\
& \quad \leq C \lVert \tilde{u}_{0xx} \rVert_{L^2}^2 + C \int_0^t \left( \lVert \tilde{v}_{xx} \rVert_{L^2}^2 + \lVert \tilde{\phi}_{xx} \rVert_{L^2}^2 \right) \, d\tau \\
& \qquad + C \left( \delta_1 + \varepsilon_1 \right) \int_0^t \left( \delta_S \lvert \dot{X} \rvert^2 + G_1 + G^S + D + \lVert \tilde{v}_x \rVert_{L^2}^2 + \lVert \tilde{u}_{xx} \rVert_{L^2}^2 \right) \, d\tau
\end{split}
\end{equation}
for all $t \in [0,T]$, where $G_1$, $G^S$, and $D$ are as defined in Lemma~\ref{RE}.
\end{lemma}

\begin{proof}
Differentiating \eqref{1b'} twice with respect to $x$, and multiplying the resultant equation by $\tilde{u}_{xx}$, we have
\begin{equation} \label{ux22}
\begin{split}
\left( \frac{\tilde{u}_{xx}^2}{2} \right)_t + \frac{\tilde{u}_{xxx}^2}{v} & = (\cdots)_x + \tilde{u}_{xx} \left( \frac{\tilde{v}}{v\bar{v}^X} \right)_{xxx} + \tilde{u}_{xxx} \left( \frac{\tilde{v}_x \tilde{u}_{xx}}{v^2} + \frac{\bar{v}^X_x \tilde{u}_{xx}}{v^2} \right) \\
& \quad + \tilde{u}_{xxx} \left( \frac{\tilde{v}_x \tilde{u}_x}{v^2} + \frac{\bar{v}^X_x \tilde{u}_x}{v^2} \right)_x + \tilde{u}_{xxx} \left( \frac{\bar{u}^X_x \tilde{v}}{v\bar{v}^X} \right)_{xx} - \tilde{u}_{xx} \left( \frac{\tilde{\phi}_x}{v} - \frac{\bar{\phi}^X_x\tilde{v}}{v\bar{v}^X} \right)_{xx} \\
& \quad + \dot{X} (t) \bar{u}^X_{xxx} \tilde{u}_{xx},
\end{split}
\end{equation}
where we have used
\[
\begin{split}
- \tilde{u}_{xx} \left( \frac{u_x}{v} - \frac{\bar{u}^X_x}{\bar{v}^X} \right)_{xxx} & = (\cdots)_x + \frac{\tilde{u}_{xxx}^2}{v} - \tilde{u}_{xxx} \left( \frac{\tilde{v}_x \tilde{u}_{xx}}{v^2} + \frac{\bar{v}^X_x \tilde{u}_{xx}}{v^2} \right) \\
& \quad - \tilde{u}_{xxx} \left( \frac{\tilde{v}_x \tilde{u}_x}{v^2} + \frac{\bar{v}^X_x \tilde{u}_x}{v^2} \right)_x - \tilde{u}_{xx} \left( \frac{\bar{u}^X_x \tilde{v}}{v \bar{v}^X} \right)_{xx}.
\end{split}
\]
Integrating \eqref{ux22} with respect to $x$ and applying integration by parts, we have
\begin{equation} \label{ux222}
\frac{d}{dt} \int_\mathbb{R} \frac{\tilde{u}_{xx}^2}{2} \, dx + \int_\mathbb{R} \frac{\tilde{u}_{xxx}^2}{v} \, dx  = \sum_{j=1}^4 \mathcal{U}^{(2)}_j,
\end{equation}
where
\[
\begin{split}
\mathcal{U}^{(2)}_1 & := \int_\mathbb{R} \left( - \frac{\tilde{u}_{xxx}\tilde{v}_{xx}}{v\bar{v}^X} + \frac{\tilde{u}_{xxx}\tilde{\phi}_{xx}}{v} \right) \, dx, \quad \mathcal{U}^{(2)}_2 := \dot{X}(t) \int_\mathbb{R} \bar{u}^X_{xxx} \tilde{u}_{xx} \, dx, \\
\mathcal{U}^{(2)}_3 & := \int_\mathbb{R} \tilde{u}_{xxx} \left( \frac{2 \tilde{v}_x^2}{v^2\bar{v}^X} + \frac{\tilde{v}_{xx}\tilde{v}}{v^2\bar{v}^X} - \frac{2\tilde{v}_x^2 \tilde{v}}{v^3 \bar{v}^X} - \frac{4\bar{v}^X_x \tilde{v}_x \tilde{v}}{v^3\bar{v}^X} - \frac{2 \bar{v}^X_x \tilde{v}_x \tilde{v}}{v^2(\bar{v}^X)^2} \right) \, dx \\
& \quad + \int_\mathbb{R} \tilde{u}_{xxx} \left( \frac{2\tilde{v}_x\tilde{u}_{xx}}{v^2} + \frac{\tilde{v}_{xx}\tilde{u}_x}{v^2} - \frac{2 \tilde{v}_x^2 \tilde{u}_x}{v^3} - \frac{4 \bar{v}^X_x \tilde{v}_x\tilde{u}_x}{v^3} \right) \, dx \\
& \quad + \int_\mathbb{R} \tilde{u}_{xxx} \left[ \frac{\bar{u}^X_x}{\bar{v}^X} \left( \frac{2 \tilde{v}_x^2 \tilde{v}}{v^3} + \frac{4\bar{v}^X_x \tilde{v}_x \tilde{v}}{v^3} - \frac{2\tilde{v}_x^2}{v^2} - \frac{\tilde{v}_{xx}\tilde{v}}{v^2} \right) - \left( \frac{\bar{u}^X_x}{\bar{v}^X} \right)_x \frac{\tilde{v}_x\tilde{v}}{v^2} \right] \, dx \\
& \quad + \int_\mathbb{R} \tilde{u}_{xxx} \left( \frac{\bar{\phi}^X_x \tilde{v}_x \tilde{v}}{v^2 \bar{v}^X} - \frac{\tilde{v}_x \tilde{\phi}_x}{v^2} \right) \, dx,
\end{split}
\]
and
\[
\begin{split}
\mathcal{U}^{(2)}_4 & := \int_\mathbb{R} \tilde{u}_{xxx} \left( \frac{2\bar{v}^X_x \tilde{v}_x}{v^2 \bar{v}^X} + \frac{2\bar{v}^X_x \tilde{v}_x}{v(\bar{v}^X)^2} - \frac{2(\bar{v}^X_x)^2 \tilde{v}}{v^3\bar{v}^X} - \frac{\bar{v}^X_{xx} \tilde{v}}{v^2 \bar{v}^X} + \frac{2(\bar{v}^X)^2 \tilde{v}}{v^2(\bar{v}^X)^2} - \left( \frac{\bar{v}^X_x}{(\bar{v}^X)^2} \right)_x \frac{\tilde{v}}{v} \right) \, dx \\
& \quad + \int_\mathbb{R} \tilde{u}_{xxx} \left( \frac{\bar{v}^X_x \tilde{u}_{xx}}{v^2} + \frac{\bar{v}^X_x \tilde{u}_{xx}}{v^2} - \frac{2(\bar{v}^X_x)^2 \tilde{u}_x}{v^3} + \frac{\bar{v}^X_{xx} \tilde{u}_x}{v^2} \right) \, dx \\
& \quad + \int_\mathbb{R} \tilde{u}_{xxx} \left[ \frac{\bar{u}^X_x}{\bar{v}^X} \left( \frac{\tilde{v}_{xx}}{v} - \frac{2\bar{v}^X_x \tilde{v}_x}{v^2} + \frac{2(\bar{v}^X_x)^2 \tilde{v}}{v^3} - \frac{\bar{v}^X_{xx}\tilde{v}}{v^2} \right) - \left( \frac{ \bar{u}^X_x}{\bar{v}^X} \right)_x \frac{2 \bar{v}^X_x \tilde{v}}{v^2} + \left( \frac{\bar{u}^X_x}{\bar{v}^X} \right)_{xx}  \frac{\tilde{v}}{v} \right] \, dx \\
& \quad + \int_\mathbb{R} \tilde{u}_{xxx} \left( - \frac{\bar{v}^X_x \tilde{\phi}_x}{v^2} - \frac{\bar{\phi}^X_x \tilde{v}_x}{v\bar{v}^X} + \frac{\bar{v}^X_x \bar{\phi}^X_x \tilde{v}}{v^2 \bar{v}^X} - \left( \frac{\bar{\phi}^X_x}{\bar{v}^X} \right)_x \frac{\tilde{v}}{v} \right) \, dx.
\end{split}
\]

We estimate the term $\mathcal{U}^{(2)}_j$ in a similar manner to the estimate of $\mathcal{U}^{(1)}_j$ for each $j$ in the proof of Lemma~\ref{Lemma:u1}. The term $\mathcal{U}^{(2)}_1$ is estimated by applying Young's inequality as
\[
\lvert \mathcal{U}^{(2)}_1 \rvert \leq C \int_\mathbb{R} \lvert \tilde{u}_{xxx} \rvert \left( \lvert \tilde{v}_{xx} \rvert + \lvert \tilde{\phi}_{xx} \rvert \right) \, dx  \leq \eta \int_\mathbb{R} \tilde{u}_{xxx}^2 \, dx + \frac{C}{\eta} \int_\mathbb{R} \left( \tilde{v}_{xx}^2 + \tilde{\phi}_{xx}^2 \right) \, dx
\]
for any $0 < \eta < 1$. We obtain the bounds on $\mathcal{U}^{(2)}_2$ by using $\lvert \bar{u}^X_{xxx} \rvert \leq C \lvert \bar{v}^X_x \rvert$, the H\"older inequality, and Young's inequality
\[
\begin{split}
\lvert \mathcal{U}^{(2)}_2 \rvert & \leq C \lvert \dot{X} \rvert \sqrt{ \int_\mathbb{R} \bar{v}^X_x \, dx } \sqrt{\int_\mathbb{R} \bar{v}^X_x \tilde{u}_{xx}^2 \, dx } \\
& \leq C \delta_S^2 \lvert \dot{X} \rvert^2 + \frac{C}{\delta_S^2} \left( \int_\mathbb{R} \bar{v}^X_x \, dx \right) \left( \int_\mathbb{R} \bar{v}^X_x \tilde{u}_{xx}^2 \, dx \right) \\
& \leq C \delta_S \left( \delta_S \lvert \dot{X} \rvert^2 + \int_\mathbb{R} \tilde{u}_{xx}^2 \, dx \right).
\end{split}
\]
The nonlinear term $\mathcal{U}^{(2)}_3$ is estimates as
\[
\begin{split}
\lvert \mathcal{U}^{(2)}_3 \rvert & \leq C \int_\mathbb{R} \lvert \tilde{v} \rvert \lvert \tilde{u}_{xxx} \rvert \left( \lvert \tilde{v}_{xx} \rvert + \lvert \tilde{v}_x \rvert + \lvert \tilde{v}_x \rvert^2 \right) \, dx + C \int_\mathbb{R} \lvert \tilde{u}_x \rvert \lvert \tilde{u}_{xxx} \rvert \lvert \tilde{v}_{xx} \rvert \, dx \\
& \quad + C \int_\mathbb{R} \lvert \tilde{v}_x \rvert \lvert \tilde{u}_{xxx} \rvert \left( \lvert \tilde{v}_x \rvert + \lvert \tilde{u}_{xx} \rvert + \lvert \tilde{u}_x \rvert + \lvert \tilde{v}_x \rvert \lvert \tilde{u}_x \rvert + \lvert \tilde{\phi}_x \rvert \right) \, dx \\
& \leq C \left( \lVert \tilde{v} \rVert_{W^{1,\infty}} + \lVert \tilde{u}_x \rVert_{L^\infty} \right) \int_\mathbb{R} \left( \tilde{u}_{xxx}^2 + \tilde{u}_{xx}^2 + \tilde{u}_x^2 + \tilde{v}_{xx}^2 + \tilde{v}_x^2 + \tilde{\phi}_x^2 \right) \, dx.
\end{split}
\]
Finally, we use the bounds $\lvert \bar{v}^X_{xx} \rvert, \lvert \bar{u}^X_{xx} \rvert, \lVert \bar{\phi}^X_{xx} \rvert, \lvert \bar{u}^X_{xxx} \rvert \leq C \lvert \bar{v}^X_x \rvert$ and $\lvert \bar{\phi}^X_x \rvert \sim \lvert \bar{u}^X_x \rvert \sim \lvert \bar{v}^X_x \rvert \leq C \delta_S^2$ to obtain
\[
\begin{split}
\lvert \mathcal{U}^{(2)}_4 \rvert & \leq C \int_\mathbb{R} \lvert \bar{v}^X_x \rvert \lvert \tilde{u}_{xxx} \rvert \lvert \tilde{v} \rvert \, dx + C \int_\mathbb{R} \lvert \bar{v}^X_x \rvert \lvert \tilde{u}_{xxx} \rvert \left( \lvert \tilde{v}_{xx} \rvert + \lvert \tilde{v}_x \rvert + \lvert \tilde{u}_{xx} \rvert + \lvert \tilde{u}_x \rvert + \lvert \tilde{\phi}_x \rvert \right) \, dx \\
& \leq C \int_\mathbb{R} \lvert \bar{v}^X_x \rvert^{1/2} \tilde{u}_{xxx}^2 \, dx + C \int_\mathbb{R} \lvert \bar{v}^X_x \rvert^{3/2} \tilde{v}^2 \, dx \\
& \quad + C \int_\mathbb{R} \lvert \bar{v}^X_x \rvert \left( \tilde{u}_{xxx}^2 + \tilde{u}_{xx}^2 + \tilde{u}_x^2 + \tilde{v}_{xx}^2 + \tilde{v}_x^2 + \tilde{\phi}_x^2 \right) \, dx \\
& \leq C \delta_S \int_\mathbb{R} \left( \bar{v}^X_x \tilde{v}^2 + \tilde{u}_{xxx}^2 + \tilde{u}_{xx}^2 + \tilde{u}_x^2 + \tilde{v}_{xx}^2 + \tilde{v}_x^2 + \tilde{\phi}_x^2 \right) \, dx.
\end{split}
\]

Combining \eqref{ux222} with the estimates of $\mathcal{U}^{(2)}_j$ and taking $\eta$ sufficiently small, by the smallness of the parameters $\delta_1, \varepsilon_1$ and the results of Lemmas~\ref{Lv^2} and \ref{Lemma:phi}, we obtain 
\[
\begin{split}
\frac{d}{dt} \int_\mathbb{R} \frac{\tilde{u}_{xx}^2}{2} \, dx + \int_\mathbb{R} \tilde{u}_{xxx}^2\, dx & \leq C \left( \delta_1 + \varepsilon_1 \right)  \left( G_1 + G^S  + D + \int_\mathbb{R} \left( \tilde{v}_x^2 + \tilde{u}_{xx}^2 \right) \, dx + \delta_S \lvert \dot{X} \rvert^2 \right) \\
& \quad + C \int_\mathbb{R} \left( \tilde{v}_{xx}^2 + \tilde{\phi}_{xx}^2 \right) \, dx. 
\end{split}
\]
By integrating this with respect to $t$, we obtain \eqref{derivu2}.
\end{proof}

\subsection{Proof of Proposition \ref{Apriori}}
Now, to prove Proposition \ref{Apriori}, we combine \eqref{Ree'} with the higher-order estimates \eqref{derivv1}, \eqref{derivu1}, \eqref{derivv2}, \eqref{derivu2}. Summing \eqref{Ree'}, \eqref{derivv1}, \eqref{derivu1}, \eqref{derivv2}, and \eqref{derivu2}, we have
\[
\begin{split}
& \lVert (\tilde{v},\tilde{u},\tilde{\phi}) (t,\cdot) \rVert_{H^2}^2 + \int_0^t \left( \delta_S \lvert \dot{X} \rvert^2 + G_1 + G^S + D + \lVert \tilde{v}_x \rVert_{H^1}^2 + \lVert \tilde{u}_{xx} \rVert_{H^1}^2 + \lVert \tilde{\phi}_{xx} \rVert_{H^1}^2 \right) \, d\tau \\
& \leq C \lVert (\tilde{v}_0,\tilde{u}_0) \rVert_{H^2}^2 + \eta \int_0^t \lVert \tilde{u}_{x} \rVert_{H^1}^2 (\tau) \, d\tau + \frac{C}{\eta} \int_0^t \left( \lVert \tilde{v}_x \rVert_{L^2}^2 + \lVert \tilde{\phi}_{xx} \rVert_{L^2}^2 \right) \, d\tau \\
& \quad + C \left(\sqrt{\delta_1} + \varepsilon_1 \right) \int_0^t \left( \delta_S \lvert \dot{X} \rvert^2 + G_1 + G^S + D + \lVert \tilde{v}_x \rVert_{H^1}^2 + \lVert \tilde{u}_{xx} \rVert_{L^2}^2 \right) \, d\tau \\
& \quad + C \lVert \tilde{u}_x (t,\cdot) \rVert_{L^2}^2  + C \int_0^t \left( \lVert \tilde{v}_{xx} \rVert_{L^2}^2 + \lVert \tilde{u}_{xx} \rVert_{L^2}^2 + \lVert \tilde{\phi}_{xx} \rVert_{L^2}^2 \right) \, d\tau,
\end{split}
\]
where $ \eta > 0$ is any small constant and $C>0$ is a generic constant. To control $\textstyle \int_0^t \lVert \tilde{v}_{xx} \rVert_{L^2}^2 $ and $\textstyle  \lVert \tilde{u}_x(t,\cdot) \rVert_{L^2}^2 + \int_0^t \lVert \tilde{u}_{xx} \rVert_{L^2}^2$, we use \eqref{derivv2} and \eqref{derivu1} successively. Then we obtain
\[
\begin{split}
& \lVert (\tilde{v},\tilde{u},\tilde{\phi}) (t,\cdot) \rVert_{H^2}^2 + \int_0^t \left( \delta_S \lvert \dot{X} \rvert^2 + G_1 + G^S + D + \lVert \tilde{v}_x \rVert_{H^1}^2 + \lVert \tilde{u}_{xx} \rVert_{H^1}^2 + \lVert \tilde{\phi}_{xx} \rVert_{H^1}^2 \right) \, d\tau \\
& \leq C \lVert (\tilde{v}_0,\tilde{u}_0) \rVert_{H^2}^2 + \tilde{\eta} \int_0^t \lVert \tilde{u}_{x} \rVert_{H^1}^2 (\tau) \, d\tau + \frac{C}{\tilde{\eta}} \int_0^t \left( \lVert \tilde{v}_x \rVert_{L^2}^2 + \lVert \tilde{\phi}_{xx} \rVert_{L^2}^2 \right) \, d\tau \\
& \quad + C \left(\sqrt{\delta_1} + \varepsilon_1 \right) \int_0^t \left( \delta_S \lvert \dot{X} \rvert^2 + G_1 + G^S + D + \lVert \tilde{v}_x \rVert_{H^1}^2 + \lVert \tilde{u}_{xx} \rVert_{L^2}^2 \right) \, d\tau  + C \int_0^t \lVert \tilde{\phi}_{xx} \rVert_{L^2}^2 \, d\tau
\end{split}
\]
for any $\tilde{\eta}$ with $0 < \tilde{\eta} < 1$. Taking $\tilde{\eta}$ sufficiently small, we have
\[
\begin{split}
& \lVert (\tilde{v},\tilde{u},\tilde{\phi}) (t,\cdot) \rVert_{H^2}^2 + \int_0^t \left( \delta_S \lvert \dot{X} \rvert^2 + G_1 + G^S + D + \lVert \tilde{v}_x \rVert_{H^1}^2 + \lVert \tilde{u}_{xx} \rVert_{H^1}^2 + \lVert \tilde{\phi}_{xx} \rVert_{H^1}^2 \right) \, d\tau \\
& \leq C \lVert (\tilde{v}_0,\tilde{u}_0) \rVert_{H^2}^2  + C \left(\sqrt{\delta_1} + \varepsilon_1 \right) \int_0^t \left( \delta_S \lvert \dot{X} \rvert^2 + G_1 + G^S + D + \lVert \tilde{v}_x \rVert_{H^1}^2 + \lVert \tilde{u}_{xx} \rVert_{L^2}^2 \right) \, d\tau \\
& \quad + C \int_0^t \left( \lVert \tilde{v}_x \rVert_{L^2}^2 + \lVert \tilde{\phi}_{xx} \rVert_{L^2}^2 \right) \, d\tau.
\end{split}
\]
Lastly, using \eqref{derivv1}, we obtain \eqref{apriori} by the smallness of the parameters $\delta_1$ and $\varepsilon_1$.

\bigskip

\appendix

\section{Basic elliptic estimates for the Poisson equation}
In this appendix, we provide basic $L^2$-estimates of the perturbation $\tilde{\phi}$ and its $x$-derivatives.

\begin{lemma}
Under the assumptions in Proposition \ref{Apriori}, it holds that
\begin{equation} \label{phiest}
\lVert \tilde{\phi} (t,\cdot) \rVert_{H^k}^2 \leq C \lVert \tilde{v} (t,\cdot) \rVert_{H^{k-1}}^2
\end{equation}
for all $t \in [0,T]$ and $k=1,2,3$, where $\lVert \cdot \rVert_{H^0} = \lVert \cdot \rVert_{L^2}$ and $C>0$ is a generic constant.
\end{lemma}

\begin{proof}
Taking the $L^2$ inner product of \eqref{1c'} against $\tilde{\phi}$, we have
\[
\la - \left( \frac{\tilde{\phi}_x}{v} - \frac{\bar{\phi}^X_x \tilde{v}}{v\bar{v}^X} \right)_x, \tilde{\phi} \ra = \la - e^{\bar{\phi}^X}\tilde{v} - \bar{v}^X e^{\bar{\phi}^X}, \tilde{\phi} \ra + \la \bar{v}^X e^{\bar{\phi}^X}\left( 1 + \tilde{\phi} - e^{\tilde{\phi}} \right) + e^{\bar{\phi}^X}\tilde{v}\left( 1 - e^{\tilde{\phi}} \right), \tilde{\phi} \ra.
\]
By integration by parts, one can obtain
\[
\begin{split}
\la \frac{\tilde{\phi}_x}{v},\tilde{\phi}_x \ra + \la \bar{v}^Xe^{\bar{\phi}^X}\tilde{\phi}, \tilde{\phi} \ra & = \la \frac{\bar{\phi}^X_x \tilde{v}}{v\bar{v}^X}, \tilde{\phi}_x \ra + \la - e^{\bar{\phi}^X} \tilde{v}, \tilde{\phi} \ra \\
& \quad + \la \bar{v}^X e^{\bar{\phi}^X}\left( 1 + \tilde{\phi} - e^{\tilde{\phi}} \right) + e^{\bar{\phi}^X}\tilde{v}\left( 1 - e^{\tilde{\phi}} \right), \tilde{\phi} \ra.
\end{split}
\]
Using the bound $\lvert \bar{\phi}^X_x \rvert \leq C \delta_S^2 $ and the expansion $e^{\tilde{\phi}} = 1 + \mathcal{O}(\tilde{\phi})$ near $\tilde{\phi}=0$, we have after several applications of Young's inequality
\[
\lVert \tilde{\phi} \rVert_{L^2}^2 + \lVert \tilde{\phi}_x \rVert_{L^2}^2  \leq \eta \left( \lVert \tilde{\phi} \rVert_{L^2}^2 + \lVert \tilde{\phi}_x \rVert_{L^2}^2 \right) + \frac{C}{\eta} \lVert \tilde{v} \rVert_{L^2}^2 + C \left( \delta_S^2 + \lVert \tilde{v} \rVert_{L^\infty} + \lVert \tilde{\phi} \rVert_{L^\infty} \right) \lVert \tilde{\phi} \rVert_{L^2}^2
\]
for any constant $0 < \eta < 1$. Taking $\eta$ small enough, we obtain
\begin{equation} \label{Apphi1}
\lVert \tilde{\phi} \rVert_{H^1}^2 \leq C \lVert \tilde{v} \rVert_{L^2}^2
\end{equation}
for sufficiently small $\delta_1$, $\varepsilon_1$.

Next we estimate $\lVert \tilde{\phi}_x \rVert_{H^1}$. Differentiating \eqref{1c'} and taking the $L^2$ inner product of the resultant equation against $\tilde{\phi}_x$, we have
\[
\begin{split}
\la - \left( \frac{\tilde{\phi}_x}{v} - \frac{\bar{\phi}^X_x \tilde{v}}{v\bar{v}^X} \right)_{xx}, \tilde{\phi}_x \ra & = - \la ( e^{\bar{\phi}^X}\tilde{v} + \bar{v}^X e^{\bar{\phi}^X} )_x, \tilde{\phi}_x \ra \\
& \quad + \la \left( \bar{v}^X e^{\bar{\phi}^X}\left( 1 + \tilde{\phi} - e^{\tilde{\phi}} \right) + e^{\bar{\phi}^X}\tilde{v}\left( 1 - e^{\tilde{\phi}} \right) \right)_x, \tilde{\phi}_x \ra.
\end{split}
\]
By integration by parts, we have after rearrangement
\[
\begin{split}
\la \frac{\tilde{\phi}_{xx}}{v}, \tilde{\phi}_{xx} \ra + \la \bar{v}^X e^{\bar{\phi}^X} \tilde{\phi}_x, \tilde{\phi}_x \ra & = \la \frac{v_x \tilde{\phi}_x}{v^2} + \left( \frac{\bar{\phi}^X_x \tilde{v}}{v\bar{v}^X} \right)_x, \tilde{\phi}_{xx} \ra - \la (e^{\bar{\phi}^X}\tilde{v})_x + (\bar{v}^X e^{\bar{\phi}^X})_x \tilde{\phi}, \tilde{\phi}_x \ra \\
& \quad + \la \left( \bar{v}^X e^{\bar{\phi}^X}\left( 1 + \tilde{\phi} - e^{\tilde{\phi}} \right) + e^{\bar{\phi}^X}\tilde{v}\left( 1 - e^{\tilde{\phi}} \right) \right)_x, \tilde{\phi}_x \ra.
\end{split}
\]
Similarly to the estimate of $\lVert \tilde{\phi} \rVert_{H^1}$, we obtain by Young's inequality
\begin{equation} \label{Ap2}
\begin{split}
\lVert \tilde{\phi}_{xx} \rVert_{L^2}^2 + \lVert \tilde{\phi}_x \rVert_{L^2}^2 & \leq  \eta  \lVert \tilde{\phi}_x \rVert_{L^2}^2 + \frac{C}{\eta} \lVert \tilde{v}_x \rVert_{L^2}^2 \\
& \quad + C \left( \delta_S^2 + \lVert \tilde{v} \rVert_{L^\infty} + \lVert \tilde{\phi} \rVert_{W^{1,\infty}} \right) \left( \lVert \tilde{\phi} \rVert_{H^2}^2 + \lVert \tilde{v} \rVert_{H^1}^2 \right) 
\end{split}
\end{equation}
for any $0<\eta<1$. For sufficiently small $\eta$, $\delta_1$, and $\varepsilon_1$, the inequality \eqref{Ap2}, together with \eqref{Apphi1}, yields
\[
\lVert \tilde{\phi} \rVert_{H^2}^2 \leq C \lVert \tilde{v} \rVert_{H^1}^2.
\]
Similarly, one can obtain the $H^3$-estimate by differentiating \eqref{1c'} twice and taking the $L^2$ inner product against $\tilde{\phi}_{xx}$. We omit the details. 
\end{proof}

\section{Global existence of perturbations}
We show that the local solution $(v,u,\phi)$ obtained in Proposition \ref{local} remains close to the shifted shock profile $(\bar{v}^X,\bar{u}^X,\bar{\phi}^X)$ in a small time interval, provided the reference state $(\underline{v},\underline{u},\underline{\phi})$ is appropriately chosen. We choose smooth functions $\underline{v}$, $\underline{u}$, and $\underline{\phi}$ satisfying
\[
\begin{split}
& \sum_{\pm} \left( \lVert \underline{v} - v_\pm \rVert_{L^2(\mathbb{R}_\pm)} + \lVert \underline{u} - u_\pm \rVert_{L^2(\mathbb{R}_\pm)} + \lVert \underline{\phi} - \phi_\pm \rVert_{L^2(\mathbb{R}_\pm)}  \right) \\
& \quad + \lVert \underline{v}_x \rVert_{H^1(\mathbb{R})} + \lVert \underline{u}_x \rVert_{H^1(\mathbb{R})} + \lVert  \underline{\phi}_x \rVert_{H^2(\mathbb{R})} < C \delta_S.
\end{split}
\]
Then, by \eqref{shderiv}, we have
\[
\begin{split}
& \lVert \underline{v} - \bar{v} \rVert_{H^2(\mathbb{R})} + \lVert \underline{u} - \bar{u} \rVert_{H^2(\mathbb{R})} + \lVert \underline{\phi} - \bar{\phi} \rVert_{H^3(\mathbb{R})} \\
& \quad \leq \sum_{\pm} \left( \lVert \underline{v} - v_\pm \rVert_{L^2(\mathbb{R}_\pm)} +  \lVert \underline{u} - u_\pm \rVert_{L^2(\mathbb{R}_\pm)} +  \lVert \underline{\phi} - \phi_\pm \rVert_{L^2(\mathbb{R}_\pm)} \right) \\
& \qquad + \sum_{\pm} \left( \lVert \bar{v} - v_\pm \rVert_{L^2(\mathbb{R}_\pm)} + \lVert \bar{u} - u_\pm \rVert_{L^2(\mathbb{R}_\pm)} + \lVert \bar{\phi} - \phi_\pm \rVert_{L^2(\mathbb{R}_\pm)} \right) \\
& \qquad + \lVert \underline{v}_x \rVert_{H^1(\mathbb{R})} + \lVert  \underline{u}_x \rVert_{H^1(\mathbb{R})} + \lVert \underline{\phi}_x \rVert_{H^2(\mathbb{R})}  + \lVert \bar{v}_x \rVert_{H^1(\mathbb{R})} + \lVert \bar{u}_x \rVert_{H^1(\mathbb{R})} + \lVert \bar{\phi}_x \rVert_{H^2(\mathbb{R})} \\
& \quad \leq C \sqrt{\delta_S}.
\end{split}
\]
For sufficiently small $\delta_S$, we choose $\varepsilon_2$ as
\[
\varepsilon_2 < \frac{\varepsilon_1}{3} - C \delta_S - C \sqrt{\delta_S}.
\]
Consider any initial data $(v_0,u_0)$ satisfying
\[
\sum_{\pm} \left( \lVert v_0 - v_\pm \rVert_{L^2(\mathbb{R}_\pm)} + \lVert u_0 - u_\pm \rVert_{L^2(\mathbb{R}_\pm)} \right) + \lVert v_{0x} \rVert_{H^1(\mathbb{R})} + \lVert u_{0x} \rVert_{H^1(\mathbb{R})} < \varepsilon_2.
\]
Then, we obtain
\[
\begin{split}
& \lVert v_0 - \underline{v} \rVert_{H^2(\mathbb{R})} + \lVert u_0 - \underline{u} \rVert_{H^2(\mathbb{R})} \\
& \quad \leq \sum_{\pm} \left( \lVert v_0 - v_\pm \rVert_{L^2(\mathbb{R}_\pm)} + \lVert u_0 - u_\pm \rVert_{L^2(\mathbb{R}_\pm)} + \lVert \underline{v} - v_\pm \rVert_{L^2(\mathbb{R}_\pm)} +\lVert \underline{u} - u_\pm \rVert_{L^2(\mathbb{R}_\pm)} \right) \\
& \qquad + \lVert \underline{v}_x \rVert_{H^1(\mathbb{R})} + \lVert \underline{u}_x \rVert_{H^1(\mathbb{R})} + \lVert v_{0x} \rVert_{H^1(\mathbb{R})} + \lVert u_{0x} \rVert_{H^1(\mathbb{R})} \\
& \quad \leq \varepsilon_2 + C \delta_S < \frac{\varepsilon_1}{3}.
\end{split}
\]
By the result of Proposition \ref{local}, there exists $T_0>0$ such that
\begin{equation} \label{sol-s}
\lVert v - \underline{v} \rVert_{L^\infty(0,T_0;H^2(\mathbb{R}))} + \lVert u - \underline{u} \rVert_{L^\infty(0,T_0;H^2(\mathbb{R}))} + \lVert \phi - \underline{\phi} \rVert_{L^\infty(0,T_0;H^3(\mathbb{R}))} \leq \frac{\varepsilon_1}{2}.
\end{equation}
On the other hand, it holds that
\[
\begin{split}
& \lVert \underline{v} - \bar{v}^X \rVert_{H^2(\mathbb{R})} + \lVert \underline{u} - \bar{u}^X \rVert_{H^2(\mathbb{R})} + \lVert \underline{\phi} - \bar{\phi}^X \rVert_{H^3(\mathbb{R})} \\
& \quad \leq \sum_{\pm} \left( \lVert \underline{v} - v_\pm \rVert_{L^2(\mathbb{R}_\pm)} +  \lVert \underline{u} - u_\pm \rVert_{L^2(\mathbb{R}_\pm)} +  \lVert \underline{\phi} - \phi_\pm \rVert_{L^2(\mathbb{R}_\pm)} \right) \\
& \qquad + \sum_{\pm} \left( \lVert \bar{v}^X - v_\pm \rVert_{L^2(\mathbb{R}_\pm)} + \lVert \bar{u}^X - u_\pm \rVert_{L^2(\mathbb{R}_\pm)} + \lVert \bar{\phi}^X - \phi_\pm \rVert_{L^2(\mathbb{R}_\pm)} \right) \\
& \qquad + \lVert \underline{v}_x \rVert_{H^1(\mathbb{R})} + \lVert  \underline{u}_x \rVert_{H^1(\mathbb{R})} + \lVert \underline{\phi}_x \rVert_{H^2(\mathbb{R})}  + \lVert \bar{v}^X_x \rVert_{H^1(\mathbb{R})} + \lVert \bar{u}^X_x \rVert_{H^1(\mathbb{R})} + \lVert \bar{\phi}^X_x \rVert_{H^2(\mathbb{R})} \\
& \quad \leq C \sqrt{\delta_S} \left( 1 + \sqrt{\lvert t \rvert} + \sqrt{\lvert X(t) \rvert} \right) \leq C \sqrt{\delta_S} \left( 1 + \sqrt{t} \right).
\end{split}
\]
Taking $T \in (0,T_0)$ small enough so that $C \sqrt{\delta_S}(1+\sqrt{T}) \leq \varepsilon_1/ 2$, we have
\begin{equation} \label{s-sp}
\lVert \underline{v} - \bar{v}^X \rVert_{L^\infty(0,T;H^2(\mathbb{R}))} + \lVert \underline{u} - \bar{u}^X \rVert_{L^\infty(0,T;H^2(\mathbb{R}))} + \lVert \underline{\phi} - \bar{\phi}^X \rVert_{L^\infty(0,T;H^3(\mathbb{R}))} \leq \frac{\varepsilon_1}{2}.
\end{equation}
Combining \eqref{sol-s} and \eqref{s-sp}, we have
\[
\lVert v - \bar{v}^X \rVert_{L^\infty(0,T;H^2(\mathbb{R}))} + \lVert u - \bar{u}^X \rVert_{L^\infty(0,T;H^2(\mathbb{R}))} + \lVert \phi - \bar{\phi}^X \rVert_{L^\infty(0,T;H^3(\mathbb{R}))} \leq \varepsilon_1
\]
for the initial data satisfying
\[
\begin{split}
& \lVert v_0 - \bar{v} \rVert_{H^2(\mathbb{R})} + \lVert u_0 - \bar{u} \rVert_{H^2(\mathbb{R})} \\
& \quad \leq \lVert v_0 - \underline{v} \rVert_{H^2(\mathbb{R})} + \lVert u_0 - \underline{u} \rVert_{H^2(\mathbb{R})} + \lVert \bar{v} - \underline{v} \rVert_{H^2(\mathbb{R})} + \lVert \bar{u} - \underline{u} \rVert_{H^2(\mathbb{R})} \\
& \quad \leq  \varepsilon_2 + C \delta_S + C \sqrt{\delta_S} \leq \frac{\varepsilon_1}{3} =: \varepsilon_0.
\end{split}
\]
Then, the a priori estimate \eqref{apriori} implies that $T$ can be extended to $+\infty$, thereby proving the first assertion of Theorem \ref{Main}. Moreover, the global solution satisfies the estimate
\begin{equation} \label{MainR}
\begin{split}
& \sup_{t >0} \lVert (v-\bar{v}^X,u-\bar{u}^X, \phi-\bar{\phi}^X ) \rVert_{H^2(\mathbb{R})}^2 + \delta_S \int_0^\infty \lvert \dot{X}(t) \rvert^2 \, dt \\
& + \frac{\sigma}{\sqrt{\delta_S}} \int_0^\infty \int_\mathbb{R} \bar{v}^X_x \bigg\lvert \tilde{p}(v)-\tilde{p}(\bar{v}^X) - \frac{u-\bar{u}^X}{2C_*} \bigg\rvert^2 \, dx dt + \int_0^\infty \int_\mathbb{R} \bar{v}^X_x \lvert u - \bar{u}^X \rvert^2 \, dx dt \\
& + \int_0^\infty \int_\mathbb{R} \left( \lvert \partial_x(v-\bar{v}^X) \rvert^2 + \lvert \partial_{xx}(v-\bar{v}^X) \rvert^2 + \lvert \partial_x(u-\bar{u}^X) \rvert^2 + \lvert \partial_{xx}(u-\bar{u}^X) \rvert^2 \right) \, dx dt \\
& + \int_0^\infty \int_\mathbb{R} \left( \lvert \partial_{xxx}(u-\bar{u}^X) \rvert^2 + \lvert \partial_{xx}(\phi-\bar{\phi}^X) \rvert^2 + \lvert \partial_{xxx}(\phi-\bar{\phi}^X) \rvert^2 \right) \, dx dt \\
& \qquad \leq C \lVert (v_0-\bar{v},u_0-\bar{u}) \rVert_{H^2(\mathbb{R})}^2
\end{split}
\end{equation}
for some positive constants $C>0$, $C_*>0$.

\section{Time-asymptotic behavior}
Now we investigate the time-asymptotic behavior of the global solution obtained in the previous Appendix. For this purpose, we define the function $g(t)$ as
\[
g(t) := \lVert ( \tilde{v}_x, \tilde{u}_x, \tilde{\phi}_x) (t,\cdot) \rVert_{L^2}^2,
\]
where
\[
(\tilde{v},\tilde{u},\tilde{\phi})(t,x) := (v,u,\phi)(t,x) - (\bar{v}^X,\bar{u}^X,\bar{\phi}^X)(t,x).
\]
Our goal is to verify that $g(t)$ goes to zero as $t \rightarrow +\infty$, by showing $g \in W^{1,1}(\mathbb{R}_+)$. Thanks to \eqref{MainR}, we have
\begin{equation} \label{vxuxL1}
\int_0^\infty \lVert (\tilde{v}_x,\tilde{u}_x) (t,\cdot) \rVert_{L^2}^2 \, dt \leq C \lVert (v_0-\bar{v},u_0-\bar{u}) \rVert_{H^2}^2.
\end{equation}
By \eqref{derivphi}, together with \eqref{MainR}, we also have
\begin{equation} \label{phixL1}
\begin{split}
\int_0^\infty \lVert \tilde{\phi}_x (t,\cdot) \rVert_{L^2}^2 \, dt & \leq C \int_0^\infty \int_\mathbb{R} \bigg( \frac{\bar{v}^X_x}{\sqrt{\delta_S}} \bigg\lvert \tilde{p}(v) - \tilde{p}(\bar{v}^X) - \frac{\tilde{u}}{2C_*} \bigg\rvert^2 + \bar{v}^X_x \tilde{u}^2 + \tilde{v}_x^2 \bigg) \, dx dt \\
& \leq C \lVert (v_0-\bar{v},u_0-\bar{u}) \rVert_{H^2}^2.
\end{split}
\end{equation}
Thus, we obtain that $g(t) \in L^1$.

Next, we consider the first derivative $g'(t)$. First, using \eqref{1a'}, we have
\[
\begin{split}
\bigg\lvert \frac{d}{dt} \int_\mathbb{R} \tilde{v}_x^2  \, dx \bigg\rvert  & = \bigg\lvert 2 \int_\mathbb{R} \tilde{v}_x \tilde{v}_{xt}  \, dx  \bigg\rvert \leq \bigg\lvert 2 \int_\mathbb{R} \tilde{v}_x \tilde{u}_{xx}  \, dx \bigg\rvert  +  \bigg\lvert 2 \dot{X}(t)  \int_\mathbb{R} \bar{v}^X_{xx} \tilde{v}_x \, dx \bigg\rvert.
\end{split}
\]
Here, we apply Young's inequality and the H\"older inequality on the right-hand side of the inequality to obtain
\[
\begin{split}
RHS & \leq  C  \int_\mathbb{R} \tilde{v}_x^2 + \tilde{u}_{xx}^2 \, dx + C \delta_S \lvert \dot{X} \rvert^2 + \frac{C}{\delta_S} \left( \int_\mathbb{R} \lvert \bar{v}^X_{xx} \rvert \lvert \tilde{v}_x \rvert \, dx \right)^2 \\
& \leq C  \int_\mathbb{R} \tilde{v}_x^2 + \tilde{u}_{xx}^2 \, dx + C \delta_S \lvert \dot{X} \rvert^2 + \frac{C}{\delta_S} \left( \int_\mathbb{R} \lvert \bar{v}^X_{xx} \rvert \, dx \right) \left( \int_\mathbb{R} \lvert \bar{v}^X_{xx} \rvert \lvert \tilde{v}_x \rvert^2 \, dx \right) \\
& \leq C \int_\mathbb{R} \tilde{v}_x^2 + \tilde{u}_{xx}^2 \, dx + C \delta_S \lvert \dot{X} \rvert^2,
\end{split}
\]
where in the last inequality, we used the bound $\lvert \bar{v}^X_{xx} \rvert \leq C \lvert \bar{v}^X_x \rvert$. Also, by \eqref{derivphi} and \eqref{phit}, we obtain
\[
\begin{split}
\bigg\lvert \frac{d}{dt} \int_\mathbb{R} \tilde{\phi}_x^2 \, dx \bigg\rvert & = \bigg\lvert 2 \int_\mathbb{R} \tilde{\phi}_x \tilde{\phi}_{xt} \, dx \bigg\rvert \leq C \left( \int_\mathbb{R} \tilde{\phi}_x^2 \, dx + \int_\mathbb{R} \tilde{\phi}_{xt}^2 \, dx \right) \\
& \leq  C \int_\mathbb{R} \bigg( \frac{\bar{v}^X_x}{\sqrt{\delta_S}} \bigg\lvert \tilde{p}(v) - \tilde{p}(\bar{v}^X) - \frac{\tilde{u}}{2C_*} \bigg\rvert^2 + \bar{v}^X_x \tilde{u}^2 + \tilde{v}_x^2 + \tilde{u}_x^2 + \delta_S \lvert \dot{X} \rvert^2 \bigg) \, dx.
\end{split}
\]
Lastly, we have by using \eqref{1b'} and integrating by parts
\[
\begin{split}
\bigg\lvert \frac{d}{dt} \int_\mathbb{R} \tilde{u}_x^2 \, dx \bigg\rvert & = \bigg\lvert 2 \int_\mathbb{R} \tilde{u}_x \tilde{u}_{xt} \, dx \bigg\rvert \\
& = 2 \bigg\lvert \int_\mathbb{R} \tilde{u}_{xx} \left[ \left( \frac{1}{v} - \frac{1}{\bar{v}^X} \right)_x - \left( \frac{u_x}{v} - \frac{\bar{u}^X_x}{\bar{v}^X} \right)_x + \left( \frac{\phi_x}{v} - \frac{\bar{\phi}^X_x}{\bar{v}^X} \right) \right] + \tilde{u}_x \dot{X}(t) \bar{u}^X_{xx} \, dx \bigg\rvert \\
& \leq C \bigg\lvert \int_\mathbb{R} \tilde{u}_{xx} \left( \frac{\tilde{v}}{v\bar{v}^X}  \right)_x \, dx \bigg\rvert + C \bigg\lvert \int_\mathbb{R} \tilde{u}_{xx} \left( \frac{\tilde{u}_x}{v} - \frac{\bar{u}^X_x\tilde{v}}{v\bar{v}^X}  \right)_x \, dx \bigg\rvert \\
& \quad + C \bigg\lvert \int_\mathbb{R} \tilde{u}_{xx} \left( \frac{\tilde{\phi}_x}{v} - \frac{\bar{\phi}^X_x\tilde{v}}{v\bar{v}^X} \right) \, dx \bigg\rvert + C \bigg\lvert \int_\mathbb{R} \tilde{u}_x \dot{X}(t) \bar{u}^X_{xx} \, dx \bigg\rvert \\
& =: \mathcal{U}_1 + \mathcal{U}_2 + \mathcal{U}_3 + \mathcal{U}_4.
\end{split}
\]
The term $\mathcal{U}_1$ is estimated by using \eqref{v^2} as
\[
\begin{split}
\lvert \mathcal{U}_1 \rvert & \leq C \int_\mathbb{R} \lvert \tilde{u}_{xx} \rvert \left( \lvert \tilde{v}_x \rvert  + \lvert \tilde{v} \rvert \lvert \tilde{v}_x \rvert + \lvert \bar{v}^X_x \rvert \lvert \tilde{v} \rvert \right) \, dx \\
& \leq C \left( \int_\mathbb{R} \tilde{u}_{xx}^2 \, dx + \left( 1 + \lVert \tilde{v} \rVert_{L^\infty}\right) \int_\mathbb{R} \tilde{v}_x^2 \, dx + \int_\mathbb{R} \bar{v}^X_x  \tilde{v}^2 \, dx \right) \\
& \leq C \int_\mathbb{R} \bigg( \frac{\bar{v}^X_x}{\sqrt{\delta_S}} \bigg\lvert \tilde{p}(v) - \tilde{p}(\bar{v}^X) - \frac{\tilde{u}}{2C_*} \bigg\rvert^2 + \bar{v}^X_x \tilde{u}^2 + \tilde{v}_x^2 + \tilde{u}_{xx}^2 \bigg) \, dx.
\end{split}
\]
For the terms $\mathcal{U}_2$ and $\mathcal{U}_3$, we use the bounds $\lvert \bar{u}^X_{xx} \rvert, \lvert \bar{u}^X_x \rvert, \lvert \bar{\phi}^X_x \rvert \leq C \bar{v}^X_x \leq C \delta_S^2$ and \eqref{v^2} to obtain
\[
\begin{split}
\lvert \mathcal{U}_2 \rvert &  \leq C \int_\mathbb{R} \lvert \tilde{u}_{xx} \rvert \left( \lvert \tilde{u}_{xx} \rvert + \lvert \tilde{v}_x \rvert \lvert \tilde{u}_x \rvert + \lvert \bar{v}^X_x \rvert \lvert \tilde{u}_x \rvert \right) \, dx +  C \int_\mathbb{R} \lvert \bar{v}^X_x \rvert \lvert \tilde{u}_{xx} \rvert \left( \lvert \tilde{v} \rvert +  \lvert \tilde{v}_x \rvert + \lvert \tilde{v}_x \rvert  \lvert \tilde{v} \rvert \right) \, dx \\
& \leq C \left( \int_\mathbb{R} \tilde{u}_{xx}^2 \, dx + \left( \delta_S^2 + \lVert \tilde{v}_x \rVert_{L^\infty} \right) \int_\mathbb{R}  \tilde{u}_x^2 \, dx \right) \\
& \quad + C \left( \int_\mathbb{R} \bar{v}^X_x \tilde{v}^2 \, dx + \left( \lVert \tilde{v} \rVert_{L^\infty} + 1 \right) \int_\mathbb{R} \tilde{v}_x^2 \, dx \right) \\
& \leq C \int_\mathbb{R} \bigg( \frac{\bar{v}^X_x}{\sqrt{\delta_S}} \bigg\lvert \tilde{p}(v) - \tilde{p}(\bar{v}^X) - \frac{\tilde{u}}{2C_*} \bigg\rvert^2 + \bar{v}^X_x \tilde{u}^2 + \tilde{v}_x^2 + \tilde{u}_x^2 + \tilde{u}_{xx}^2 \bigg) \, dx
\end{split}
\]
and
\[
\begin{split}
\lvert \mathcal{U}_3 \rvert & \leq C \int_\mathbb{R} \lvert \tilde{u}_{xx} \rvert \left( \lvert \tilde{\phi}_x \rvert + \lvert \bar{v}^X_x \rvert \lvert \tilde{v} \rvert \right) \, dx \\
& \leq  C \int_\mathbb{R} \bigg( \frac{\bar{v}^X_x}{\sqrt{\delta_S}} \bigg\lvert \tilde{p}(v) - \tilde{p}(\bar{v}^X) - \frac{\tilde{u}}{2C_*} \bigg\rvert^2 + \bar{v}^X_x \tilde{u}^2 + \tilde{u}_{xx}^2 + \tilde{\phi}_x^2 \bigg) \, dx.
\end{split}
\]
By applying Young's inequality and the H\"oder inequality, we obtain the bound
\[
\begin{split}
\lvert \mathcal{U}_4 \rvert & \leq C \delta_S \lvert \dot{X} \rvert^2 + \frac{C}{\delta_S} \left( \int_\mathbb{R} \bar{v}^X_x \lvert ( u - \bar{u}^X)_x \rvert \, dx \right)^2 \\
& \leq C \delta_S \lvert \dot{X} \rvert^2 + \frac{C}{\delta_S} \left( \int_\mathbb{R} \bar{v}^X_x \, dx \right) \left( \int_\mathbb{R} \bar{v}^X_x \lvert \tilde{u}_x \rvert^2 \, dx \right) \\
& \leq C \left( \delta_S \lvert \dot{X} \rvert^2 + \int_\mathbb{R} \tilde{u}_x^2 \, dx \right).
\end{split}
\]
Thus, we have
\begin{equation} \label{g'L1}
\int_0^\infty \lvert g'(t) \rvert \, dt = \int_0^\infty \bigg\lvert \frac{d}{dt} \int_\mathbb{R} \left( \tilde{v}_x^2 + \tilde{u}_x^2 + \tilde{\phi}_x^2 \right) \, dx \bigg\rvert \, dt \leq C \lVert (v_0-\bar{v},u_0-\bar{u}) \rVert_{H^2}^2,
\end{equation}
by \eqref{MainR}. Collecting \eqref{vxuxL1}, \eqref{phixL1}, and \eqref{g'L1}, we have that $g \in W^{1,1}(\mathbb{R}_+)$. Therefore, by Gagliardo-Nirenberg interpolation inequality, we obtain
\[
\begin{split}
& \lVert \tilde{v} (t,\cdot) \rVert_{L^\infty} + \lVert \tilde{u} (t,\cdot) \rVert_{L^\infty} + \lVert \tilde{\phi} (t,\cdot) \rVert_{L^\infty} \\
& \quad \leq C \left( \lVert \tilde{v} \rVert_{L^2}^{1/2} \lVert \tilde{v}_x \rVert_{L^2}^{1/2} + \lVert \tilde{u} \rVert_{L^2}^{1/2} \lVert \tilde{u}_x \rVert_{L^2}^{1/2} + \lVert \tilde{\phi} \rVert_{L^2}^{1/2} \lVert \tilde{\phi}_x \rVert_{L^2}^{1/2} \right) \rightarrow 0 \quad \text{as} \ t \rightarrow +\infty.
\end{split}
\]
Furthermore, by the estimate \eqref{Xest}, we have
\[
\lim_{t \to +\infty} \lvert \dot{X} (t) \rvert \leq C \lim_{t \to +\infty} \lVert \tilde{u} (t,\cdot) \rVert_{L^\infty}  = 0,
\]
which completes the proof of Theorem \ref{Main}.

\end{document}